\documentclass[a4paper, 12pt]{article}

\RequirePackage{amsthm,amsmath,amsfonts,amssymb}
\RequirePackage[authoryear]{natbib}
\RequirePackage{url}
\RequirePackage{graphicx}
\usepackage{subcaption}
\usepackage{fullpage}
\usepackage{color}
\usepackage{soul}
\usepackage{ifthen}

\newtheorem{theorem}{Theorem}[section]

\theoremstyle{remark}
\newtheorem{definition}[theorem]{Definition}

\theoremstyle{remark}
\newtheorem{condition}{Condition}
\newtheorem{asmp}{Assumption}

\theoremstyle{plain}
\newtheorem{proposition}{Proposition}
\newtheorem{lem}{Lemma}

\newcommand{\Eann}{P^{\text{\sc ann}(\omega)}}

\newcommand{\nm}{{\sf N}}

\newcommand{\sign}{\mathrm{sign}}

\newcommand{\RR}{\mathbb{R}}
 
\newcommand{\U}{\mathcal{U}}
\newcommand{\Z}{\mathcal{Z}}
\newcommand{\XX}{\mathbb{X}}

\newcommand{\ZZ}{\mathbb{Z}}
\newcommand{\UU}{\mathbb{U}}

\newcommand{\PP}{\mathbb{P}}

\newcommand{\Pitilde}{\widetilde{\Pi}}

\newcommand{\comp}{\text{\sc c}}

\newcommand{\eps}{\varepsilon}

\title{Gibbs posterior concentration rates under sub-exponential type losses}
\author{Nicholas Syring\footnote{Department of Statistics, Iowa State University; {\tt nsyring@iastate.edu}. } \; and \; Ryan Martin\footnote{Department of Statistics, North Carolina State University; {\tt rgmarti3@ncsu.edu}}}
\date{\today}

\begin{document}

\maketitle 

\begin{abstract}    
Bayesian posterior distributions are widely used for inference, but their dependence on a statistical model creates some challenges.  In particular, there may be lots of nuisance parameters that require prior distributions and posterior computations, plus a potentially serious risk of model misspecification bias.  Gibbs posterior distributions, on the other hand, offer direct, principled, probabilistic inference on quantities of interest through a loss function, not a model-based likelihood.  Here we provide simple sufficient conditions for establishing Gibbs posterior concentration rates when the loss function is of a sub-exponential type.  We apply these general results in a range of practically relevant examples, including mean regression, quantile regression, and sparse high-dimensional classification.  We also apply these techniques in an important problem in medical statistics, namely, estimation of a personalized minimum clinically important difference.

\smallskip

\end{abstract}

\section{Introduction}
\label{S:intro}  

A major selling point of the Bayesian framework is that it is normative: to solve a new problem, one only needs a statistical model/likelihood, a prior distribution for the model parameters, and the means to compute the corresponding posterior distribution.  Bayesians' obligation to specify a prior attracts criticism, but their need to specify a likelihood has a number of potentially negative consequences too, especially when the quantity of interest has meaning independent of a statistical model, like a quantile.  On the one hand, even if the posited model is ``correct,'' it is rare that all the parameters of that model are relevant to the problem at hand.  For example, if we are interested in a quantile of a distribution, which we model as skew-normal, then the focus is only on a specific real-valued function of the model parameters.  In such a case, the non-trivial effort invested in dealing with these nuisance parameters, e.g., specifying prior distributions and designing computational algorithms, is effectively wasted.  On the other hand, in the far more likely case where the posited model is ``wrong,'' that model misspecification can negatively impact one's conclusions about the quantity of interest.  For example, both skew-normal and Pareto models have a $\tau^\text{th}$ quantile, but the quality of inferences drawn about that quantile will vary depending on which of these two models is chosen.  

The non-negligible dependence on the posited statistical model puts a burden on the data analyst, and those reluctant to take that risk tend to opt for a non-Bayesian approach.  After all, if one can get a solution without specifying a statistical model, then it is impossible to incur model misspecification bias.  But in taking such an approach, they give up the normative advantage of Bayesian analysis.  Can they get the best of both worlds?  That is, can one construct a posterior distribution for the quantity of interest directly, incorporating available prior information, without specifying a statistical model and incurring the associated model misspecification risks, and without the need for marginalization over nuisance parameters?  Fortunately, the answer is {\em Yes}, and this is the present paper's focus. 

The so-called \emph{Gibbs posterior distribution} is the proper prior-to-posterior update when data and the interest parameter are linked by a loss function rather than a likelihood \citep{catoni, zhang.2006, bissiri.2016}.  Intuitively, the Gibbs and Bayesian posterior distributions coincide when the loss function linking data and parameter is a (negative) log-likelihood.  In that case the properties of the Gibbs posterior can be inferred from the literature on Bayesian asymptotics in both the well-specified and misspecified contexts.  For cases where the link is not through a likelihood, the large-sample behavior of the Gibbs posterior is less clear and elucidating this behavior under some simple and fairly general conditions is our goal here.   

As a practical example, medical investigators may want to know if a treatment whose effect has been judged to be statistically significant is also {\em clinically significant} in the sense that the patients  feel better post-treatment.  Therefore, they are interested in inference about the effect size cutoff beyond which patients feel better; this is called the {\em minimum clinically important difference} or {\em MCID}, e.g., \citet{jaescheke.1989}. Estimation of the MCID boils down to a classification problem, and standard Bayesian approaches to binary regression do not perform well in this setting; misspecifying the link function leads to bias, and nonparametric modeling of the link function is inefficient \citep{choudhuri.2007}. Instead, we found that a Gibbs posterior distribution, as described above, provided a very reasonable and robust solution to the MCID problem \citep{syring.martin.mcid}.  In some applications, one seeks a ``personalized'' or subject-specific cutoff that depends on a set of additional covariates.  This personalized MCID could be high- or even infinite-dimensional, and our previous Gibbs posterior analysis is not equipped to handle such situations.  But the framework developed here is; see Section~\ref{S:mcid}.  

In the following sections we lay out and apply conditions under which a Gibbs posterior distribution concentrates, asymptotically, on a neighborhood of the true value of the inferential target as the sample sizes increases.  Our focus is not on the most general set of sufficient conditions for concentration; rather, we aim for conditions that are both widely applicable and easily verified.  To this end, we consider loss functions of a {\em sub-exponential type}, ones that satisfy an inequality similar to the moment-generating function bound for sub-exponential random variables \citep{boucheron.etal.2012}.  We can apply this condition in a variety of problems, from regression to classification, and in both fixed- and high-dimensional settings.  An added advantage is that our conditions lead to straightforward proofs of concentration.  

Section~\ref{S:gibbs} provides some background and formally defines the Gibbs posterior distribution.  In Section~\ref{S:main}, we state our theoretical objectives and present our main results, namely, sets of sufficient conditions under which the Gibbs posterior achieves a specified asymptotic concentration rate. A unique attribute of the Gibbs posterior distribution is its dependence on a tuning parameter called the {\em learning rate}, and our results cover a constant, vanishing sequence, and even data-dependent learning rates. Section~\ref{s:extensions} further discusses verifying our conditions and extends our conditions and main results to handle certain unbounded loss functions.  Section~\ref{S:examples} applies our general theorems to establish Gibbs posterior concentration rates in a number of practically relevant examples, including nonparametric curve estimation, and high-dimensional sparse classification.  Section~\ref{S:mcid} formulates the personalized MCID application, presents a relevant Gibbs posterior concentration rate result, and gives a brief numerical illustration.  Concluding remarks are given in Section~\ref{S:discuss}, and proofs, etc.~are postponed to the appendix. 

\section{Background on Gibbs posteriors}
\label{S:gibbs}

\subsection{Notation and definitions}
\label{SS:gibbs.defn}

Consider a measurable space $(\UU, \U)$, with $\U$ a $\sigma$-algebra of subsets of $\UU$, on which a probability measure $P$ is defined.  A random element $U \sim P$ need not be a scalar, and many of the applications we have in mind involve $U=(X,Y)$ or $U=(X,Y,Z)$, where $Y$ denotes a ``response'' variable and $X$ or $(X,Z)$ denotes a ``predictor'' variable, and $P$ encodes the dependence between the entries in $U$.  Then the real-world phenomenon under investigation is determined by $P$ and our goal is to make inference on a relevant feature of $P$, which we define as a given functional $\theta=\theta(P)$, taking values in $\Theta$. Note that $\theta$ could be finite-, high-, or even infinite-dimensional. 

The specific way $\theta$ relates to $P$ guides our posterior construction.  Suppose there is a {\em loss function}, $\ell_\theta(u)$, that measures how closely a generic value of $\theta$ agrees with a data point $u$.  (As is customary, ``$\theta$'' will denote both the quantity of interest and a generic value of that quantity; when we need to distinguish the true from a generic value, we will write ``$\theta^\star$.'')  For example, if $u=(x,y)$ is a predictor--response pair, and $\theta$ is a function, then the loss might be 
\begin{equation}
\label{eq:loss.examples}
\ell_\theta(u) = |y - \theta(x)| \quad \text{or} \quad \ell_\theta(u) = 1\{y \neq \theta(x)\},
\end{equation} 
depending on whether $y$ is continuous or discrete/binary, where $1(A)$ denotes the indicator function at the event $A$.  Another common situation is when one specifies a statistical model, say, $\{P_\theta: \theta \in \Theta\}$, indexed by a parameter $\theta$, and sets $\ell_\theta(u) = -\log p_\theta(u)$, where $p_\theta$ is the density of $P_\theta$ with respect to some fixed dominating measure.  In all of these cases, the idea is that a loss is incurred when there is a certain discrepancy between $\theta$ and the data point $u$.  Then our inferential target is the value of $\theta$ that minimizes the risk or average loss/discrepancy.  

\begin{definition}
\label{def:R}
Consider a real-valued loss function $\ell_\theta(u)$ defined on $\UU \times \Theta$, and define the {\em risk function} $R(\theta) = P \ell_\theta$, the expected loss with respect to $P$; throughout, $P f$ denotes expected value of $f(U)$ with respect to $U \sim P$.  Then the inferential target is 
 \begin{equation}
\label{eq:target}
\theta^\star \in \arg\min_{\theta \in \Theta} R(\theta).
\end{equation}
\end{definition}

Given that estimation/inference is our goal, our focus will be on case where the risk minimizer, $\theta^\star$, is unique, so that the ``$\in$'' in \eqref{eq:target} is an equality.  But this is not absolutely necessary for our theory.  Indeed, the main results in Section~\ref{S:main} remain valid even if the risk minimizer is not unique, and we make a few brief comments about this extension in the discussion following Theorem~\ref{thm:rate}.  

The risk function itself is unavailable---it depends on $P$---and, therefore, so is $\theta^\star$.  However, suppose that we have an independent and identically distributed (iid) sample $U^n = (U_1,\ldots,U_n)$ of size $n$, with each $U_i$ having marginal distribution $P$ on $\UU$.  The iid assumption is not crucial, but it makes the notation and discussion easier; an extension to independent but not identically distributed (inid) cases is discussed in the context of an example in Section~\ref{SS:pred_sq}.  In general, we have data $U^n$ taking values in the measurable space $(\UU^n, \U^n)$, with joint distribution denoted by $P^n$.  From here, we can naturally replace the unknown risk in Definition~\ref{def:R} with an empirical version and proceed accordingly. 

\begin{definition}
\label{def:Rn}
For a loss function $\ell_\theta$ as described above, define the {\em empirical risk} as 
\begin{equation}
R_n(\theta) = \PP_n \ell_\theta = \frac1n \sum_{i=1}^n \ell_\theta(U_i), 
\end{equation}
where $\PP_n = n^{-1} \sum_{i=1}^n \delta_{U_i}$, with $\delta_u$ the Dirac point-mass measure at $u$, is the empirical distribution.  
\end{definition}

Naturally, if the inferential target is the risk minimizer, then it makes sense to estimate that quantity based on data $U^n$ by minimizing the empirical risk, i.e., 
\begin{equation}
 \hat\theta_n \in \arg\min_{\theta \in \Theta} R_n(\theta). 
 \end{equation}
This is the {\em M-estimator} based on an objective function determined by the loss $\ell_\theta$; when $R_n$ is differentiable, the root of $\dot R_n$, the derivative of $R_n$, is a {\em Z-estimator} and ``$\dot R_n(\theta) = 0$'' is often called an {\em estimating equation} \citep{godambe, vaart}. Since $\theta \mapsto \ell_\theta$ need not be smooth or convex, and $R_n$ is an average over a finite set of data, it is possible that its minimizer is not unique, even if $\theta^\star$ is.  These computational challenges are, in fact, part of what motivates the Gibbs posterior, as we discuss below.  

There is a rich literature on the asymptotic distribution properties of M-estimators, which can be used for developing hypothesis tests and confidence intervals \citep{maronna.2006, huber}.  As an alternative, one might consider a Bayesian approach to quantify uncertainty, but there is an immediate obstacle, namely, no statistical model/likelihood connecting the data to the quantity of interest.  If we did have a statistical model, with a density function $p_\theta$, then the most natural loss is $\ell_\theta(u) = -\log p_\theta(u)$ and the likelihood is $\exp\{-n R_n(\theta)\}$.  It is, therefore, tempting to follow that same strategy for general losses, resulting in a sort of generalized posterior distribution for $\theta$.  

\begin{definition}
\label{def:gibbs}
Given a loss function $\ell_\theta$ and the corresponding empirical risk $R_n$ in Definition~\ref{def:Rn}, define the {\em Gibbs posterior distribution} as 
\begin{equation}
\label{eq:gibbs}
\Pi_n(d\theta) \propto e^{-\omega \, n R_n(\theta)} \, \Pi(d\theta), \quad \theta \in \Theta, 
\end{equation}
where $\Pi$ is a prior distribution and $\omega > 0$ is a so-called {\em learning rate} \citep{holmes.2017,syring.martin.scaling, grunwald2012, vanerven.etal.2015}.  The dependence of $\Pi_n$ on $\omega$ will generally be omitted from the notation, but see Sections~\ref{SS:gibbs.lr} and \ref{SS:exp.rates}.
\end{definition}

We will assume that the right-hand side of \eqref{eq:gibbs} is integrable in $\theta$, so that the proportionality constant is well-defined.  Integrability holds whenever the loss function is bounded from below, like for those in \eqref{eq:loss.examples}, but this could fail in some cases where the loss is not bounded away from $-\infty$, e.g., when $\ell_\theta(u)$ is a negative log-density.  In such cases, extra conditions on the prior distribution would be required to ensure the Gibbs posterior is well-defined. 

An immediate advantage of this approach, compared to the M-estimation strategy described above, is that the user is able to incorporate available prior information about $\theta$ directly into the analysis.  This is especially important in cases where the quantity of interest has a real-world meaning, as opposed to being just a model parameter, where having genuine prior information is the norm rather than the exception.  Additionally, even though there is no likelihood, the same computational methods, such as Markov chain Monte Carlo \citep{chernozhukov.hong.2003} and variational approximations \citep{alquier.2016}, common in Bayesian analysis, can be employed to numerically approximate the Gibbs posterior.  

We have opted here to define the Gibbs posterior directly as an object to be used and studied, but there is a more formal, more principled way in which Gibbs posteriors emerge.  In the {\em PAC-Bayes} literature, the goal is to construct a {\em randomized estimator} that concentrates in regions of $\Theta$ where the risk, $R(\theta)$, or its empirical version, $R_n(\theta)$, is small \citep{valiant.1984, mcallester.1999, alquier.2008, guedj.2019}.  That is, the Gibbs posterior can be viewed as a solution to an optimization problem rather than a solution to the updating-prior-beliefs problem.  More formally, for a given prior $\Pi$ on $\Theta$, suppose the goal is to find 
\[ \inf_\mu \Bigl\{ \int R_n(\theta) \, \mu(d\theta) + (\omega n)^{-1} K(\mu, \Pi) \Bigr\}, \]
where the infimum is over all probability measures $\mu$ that are absolutely continuous with respect to $\Pi$, and $K$ denotes the Kullback--Leibler divergence.  Then it can be shown \citep[e.g.,][]{zhang.2006, bissiri.2016} that the unique solution is $\Pi_n$, the Gibbs posterior defined in \eqref{eq:gibbs}.  Therefore, the Gibbs posterior distribution is the measure minimizing a penalized risk, averaged with respect to a given prior, $\Pi$.

\subsection{Learning rate}
\label{SS:gibbs.lr}

Readers familiar with M-estimation may not recognize the learning rate, $\omega$.  This does not appear in the M-estimation context because all that influences the optimization problem---and the corresponding asymptotic distribution theory---is the shape of the loss/risk function, not its magnitude or scale.  On the other hand, the learning rate is an essential component of the Gibbs posterior distribution in \eqref{eq:gibbs} since the distribution depends on both the shape and scale of the loss function.  Data-driven strategies for tuning the learning rate are available \citep{holmes.2017, syring.martin.scaling, lyddon.etal.2019, wu.martin.2020, wu.martin.2021}.  

Here we focus on how the learning rate affects posterior concentration.  In typical examples, our results require the learning rate to be a sufficiently small constant.  That constant depends on features of $P$, which are generally unknown, so, in practice, the learning rate can be taken to be a slowly vanishing sequence, which has a negligible effect on the concentration rate.  In more challenging examples, we require the learning rate to vanish sufficiently fast in $n$; this is also the case in \citet{grunwald.2018}.

\subsection{Relation to other generalized posterior distributions}

A {\em generalized posterior} is any data-dependent distribution other than a well-specified Bayesian posterior.  Examples include Gibbs and misspecified Bayesian posteriors, which we compare first.  

A key characteristic of the misspecified Bayesian posterior is that it is {\em accidentally misspecified}.  That is, the data analyst does his/her best to posit a sound model, $Q_\gamma$, for the data-generating process and obtains the corresponding posterior for the model parameter $\gamma$.  That model will typically be misspecified, i.e., $P \not\in \{Q_\gamma: \gamma \in \Gamma\}$, so the aforementioned posterior will, under certain conditions, concentrate around the point $\gamma^\dagger$ that minimizes the Kullback--Leibler divergence of $Q_\gamma$ from $P$ \citep{kleijn, deblasi.walker.2013, rama.etal.2015}.  For a feature $\theta(\cdot)$ of interest, typically $\theta(Q_{\gamma^\dagger}) \neq \theta(P)$, so there is generally a bias that the data analyst can do nothing about.  A Gibbs posterior, on the other hand, is {\em purposely misspecified}---no attempt is made to model $P$.  Rather, it directly targets the feature of interest via the loss function that defines it. This strategy avoids model misspecification bias, but its point of view also sheds light on the importance of the choice of learning rate. Since the data analyst knows the Gibbs posterior is not a correctly specified Bayesian posterior, they know the learning rate must be handled with care. 

A number of authors have studied generalized posterior distributions formed using a likelihood raised to a power $\eta \in (0, 1]$; see, e.g., \citet{martin.etal.2017} and \citet{grunwald.2017} among others.  These $\eta$-generalized posteriors tend to be robust to misspecification of the probability model, and data-driven methods to tune $\eta$ are developed in, e.g., \citet{grunwald.2017}.  These posteriors coincide with Gibbs posteriors based on a log-loss and with the learning rate $\omega$ corresponding to $\eta$.

A common non-Bayesian approach to statistics bases inferences on moment conditions of the form $Pf_j = 0$ for functions $f_1,\ldots f_J$ , rather than on a fully-specified likelihood. A number of authors have extended this framework to posterior inference.  \citet{kim.2002} uses the moment conditions to construct a so-called limited-information likelihood to use in place of a fully-specified likelihood in a Bayesian formulation.  Similarly, \citet{chernozhukov.hong.2003} construct a posterior by taking a (pseudo) log-likelihood equal to a quadratic form determined by the set of moment conditions. \citet{chib.etal.2018} studies Bayesian exponentially-tilted empirical likelihood posterior distributions, also defined by moment conditions.  In some cases the Gibbs posterior distribution coincides with the above moment-based methods.  For instance, in the special case that the risk $R$ is differentiable at $\theta^\star$ with derivative $\dot R(\theta^\star)$, risk-minimization corresponds to the single moment condition $\dot R(\theta^\star)=0$.

\section{Asymptotic concentration rates}
\label{S:main}

\subsection{Objective}

A large part of the Bayesian literature concerns the asymptotic concentration properties of their posterior distributions.  Roughly, if data are generated from a distribution for which the quantity of interest takes value $\theta^\star$, then, as $n \to \infty$, the posterior distribution ought to concentrate its mass around that same $\theta^\star$.  As we will show, optimal concentration rates are possible with Gibbs posteriors, so the robustness achieved by not specifying a statistical model has no cost in terms of (asymptotic) efficiency.  

Towards this, for a fixed $\theta^\star \in \Theta$, let $d(\theta; \theta^\star)$ denote a divergence measure on $\Theta$ in the sense that $d(\theta; \theta^\star) \geq 0$ for all $\theta$, with equality if and only if $\theta=\theta^\star$.  The divergence measure could depend on the sample size $n$ or other deterministic features of the problem at hand, especially in the independent but not iid setting; see Section~\ref{SS:pred_sq}.  Our objective is to provide conditions under which the Gibbs posterior will concentrate asymptotically, at a certain rate, around $\theta^\star$ relative to the divergence measure $d$.  Throughout this paper, $(\eps_n)$ denotes a deterministic sequence of positive numbers with $\eps_n \to 0$, which will be referred to as the Gibbs posterior {\em concentration rate}.   

\begin{definition}
\label{def:rate}
The Gibbs posterior $\Pi_n$ in \eqref{eq:gibbs} asymptotically concentrates around $\theta^\star$ at rate (at least) $\eps_n$, with respect to $d$, if
\begin{equation}
\label{eq:rate}
P^n\Pi_n(\{\theta: d(\theta;\theta^\star) > M_n\eps_n\}) \rightarrow 0, \quad \text{as $n \to \infty$,}
\end{equation}
where $M_n>0$ is either a (deterministic) sequence satisfying $M_n \to \infty$ arbitrarily slowly or is a sufficiently large constant, $M_n \equiv M$. 
\end{definition}

In the PAC-Bayes literature, the Gibbs posterior distribution is interpreted as a ``randomized estimator,'' a generator of random $\theta$ values that tend to make the risk difference small.  For iid data and with risk divergence $d(\theta;\theta^\star) = \{R(\theta)-R(\theta^\star)\}^{1/2}$, a concentration result like that in Definition~\ref{def:rate} makes this strategy clear, since the $\Pi_n$-probability of the event $\{R(\theta) - R(\theta^\star) \leq \eps_n^2\}$ would be $\to 1$.

If the Gibbs posterior concentrates around $\theta^\star$ in the sense of Definition~\ref{def:rate}, then any reasonable estimator derived from that distribution, such as the mean, should inherit the $\eps_n$ rate at $\theta^\star$ relative to the divergence measure $d$.  This can be made formal under certain conditions on $d$; see, e.g., Corollary~1 in \citet{bsw1999} and the discussion following the proof of Theorem~2.5 in \citet{ggv2000}. 

Besides concentration rates, in certain cases it is possible to establish distributional approximations to Gibbs posteriors, i.e., Bernstein--von Mises theorems.  Results for finite-dimensional problems and with sufficiently smooth loss functions can be found in, e.g., \citet{bhattacharya.martin.2020} and \citet{chernozhukov.hong.2003}.

\subsection{Conditions}
\label{SS:discussion}

Here we discuss a general strategy for proving Gibbs posterior concentration and the kinds of sufficient conditions needed for the strategy to be successful.  To start, set $A_n=\{\theta:d(\theta,\theta^\star)>M_n\eps_n\} \subset \Theta$.
Then our first step towards proving concentration is to express $\Pi_n(A_n)$ as the ratio
\begin{equation}
\label{eq:gibbs.frac}\Pi_n(A_n) = \frac{N_n(A_n)}{D_n} = \frac{\int_{A_n}\exp[-\omega n\{R_n(\theta)-R_n(\theta^\star)\}] \, \Pi(d\theta)}{\int_{\Theta} \exp[-\omega n\{R_n(\theta)-R_n(\theta^\star)\}] \, \Pi(d\theta)}.
\end{equation}
The goal is to suitably upper and lower bound $N_n(A_n)$ and $D_n$, respectively, in such a way that the ratios of these bounds vanish.  Two sufficient conditions are discussed below, the first primarily dealing with the loss function and aiding in bounding $N_n(A_n)$ and the second primarily concerning the prior distribution and aiding in bounding $D_n$.  Both conditions concern the excess loss $\ell_\theta(U) - \ell_{\theta^\star}(U)$ and its mean and variance:
\[ m(\theta, \theta^\star):=P(\ell_\theta - \ell_{\theta^\star}) \quad \text{and} \quad v(\theta,\theta^\star):=P(\ell_\theta - \ell_{\theta^\star})^2 - m(\theta,\theta^\star)^2. \]

\subsubsection{Sub-exponential type losses}
\label{SSS:exp.loss}

Our method for proving posterior concentration requires a vanishing upper bound on the expectation of the numerator term $N_n(A_n)$ in \eqref{eq:gibbs.frac}.  Since the integrand, $\exp[-\omega n\{R_n(\theta)-R_n(\theta^\star)\}]$, in $N_n(A_n)$ is non-negative, Fubini's theorem says  it suffices to bound its expectation.  Further, by independence
\[ P^n e^{-\omega n \{R_n(\theta) - R_n(\theta^\star)\}} = \{ P e^{-\omega(\ell_\theta - \ell_{\theta^\star})}\}^n, \]
which reveals that the key to bounding $P^n N_n(A_n)$ is to bound $P e^{-\omega(\ell_\theta - \ell_{\theta^\star})}$, the expected exponentiated excess loss.  In order for the bound on $P^n N_n(A_n)$ to vanish it must be that $P e^{-\omega(\ell_\theta - \ell_{\theta^\star})}<1$, but this is not enough to identify the concentration rate in \eqref{eq:rate}.  Rather, the speed at which $P e^{-\omega(\ell_\theta - \ell_{\theta^\star})}$ vanishes must be a function of $d(\theta,\theta^\star)$, and we take this relationship as our key condition for Gibbs posterior concentration.  When this holds we say the loss function is of {\em sub-exponential type}.   

\begin{condition}
\label{cond:loss}
There exists an interval $(0,\bar \omega)$ and constants $K, \,r>0$, such that for all $\omega \in (0,\bar \omega)$ and for all sufficiently small $\delta > 0$, for $\theta\in\Theta$ 
\begin{equation}
\label{eq:cond.loss.global}
d(\theta; \theta^\star) > \delta \implies P e^{-\omega
(\ell_\theta-\ell_{\theta^\star})} < e^{-K \omega \delta^r} 
\end{equation}
(The constant $r>0$ that appears here and below can take on different values depending on the context.  The case is $r=2$ common, but some ``non-regular'' problems require $r \neq 2$; see Section~\ref{SS:classification}.)
\end{condition}

An immediate consequence of Condition~\ref{cond:loss} and the definition of $A_n$ is the key finite-sample, exponential bound on the Gibbs posterior numerator
\begin{equation}
\label{eq:consequence}
P^n N_n(A_n)  \leq e^{-K \omega nM_n^r\eps_n^r}. 
\end{equation}

For some intuition behind Condition~\ref{cond:loss}, consider the following.  Let $f$ be a real-valued function such that the random variable $f(U)$ has a distribution with sufficiently thin tails.  This is automatic when $f$ is bounded, but suppose $f(U)$ has a moment-generating function, i.e., is sub-exponential.  For bounding the moment-generating function, the dream case would be if $P e^{-\omega f} \leq e^{-\omega Pf}$.  Unfortunately, Jensen's inequality implies the dream is an impossibility.  It is possible, however, to show 
\[ P e^{-\omega f} \leq e^{-\omega Pf + \omega^2 G(f)}, \]
for suitable $G(f) > 0$ depending on certain features of $f$ (and of $P$). If it could also be shown, e.g., that $G(f) \lesssim Pf$, then we are in a ``near-dream'' case where 
\[ P e^{-\omega f} \leq e^{-c \omega Pf}, \quad \text{for a constant $c \in (0,1)$ and sufficiently small $\omega$.} \]
The little extra needed beyond sub-exponential to achieve the ``near-dream'' case bound above is why we refer to such $f$ as being sub-exponential type. 

Towards verifying Condition~\ref{cond:loss} we briefly review some developments in \citet{grunwald.2018}; further comments can be found in Section~\ref{ss:check1}.   These authors focus on an {\em annealed expectation} which, for a real-valued function $f$ of the random element $U \sim P$ and a fixed constant $\omega > 0$, is defined as $\Eann f = -\omega^{-1} \log P e^{-\omega f}$.  With this, we find that $P e^{-\omega f} = \exp\{ -\omega \Eann f\}$. So an upper bound as we require in Condition~\ref{cond:loss} is equivalent to a corresponding lower bound on $\Eann (\ell_\theta - \ell_{\theta^\star})$.  The ``strong central condition'' in \citet{grunwald.2018} states that, \[ \Eann (\ell_\theta - \ell_{\theta^\star}) \geq 0, \quad \text{for all $\theta \in \Theta$}. \]The above inequality implies $P e^{-\omega(\ell_\theta - \ell_{\theta^\star})} \leq 1$ and, in turn, that $P^n N_n(A_n) = O(1)$ as $n \to \infty$.  The other conditions in \citet{grunwald.2018} aim to lower-bound the annealed expected excess loss by a suitable function of the excess risk.  For example, Lemma~13 in \citet{grunwald.2018} shows that a ``witness condition'' implies \[ \Eann (\ell_\theta - \ell_{\theta^\star}) \gtrsim R(\theta) - R(\theta^\star), \quad \text{for all $\theta \in \Theta$}, \] so, if $R(\theta) - R(\theta^\star)$ is of the order $d(\theta,\theta^\star)^r$ for some $r > 0$, then we recover Condition~\ref{cond:loss}.  Therefore, our Condition~\ref{cond:loss} is {\em exactly} what is needed to control the numerator of the Gibbs posterior, and the strong central and witness conditions developed in \citet{grunwald.2018} and elsewhere constitute a set of sufficient conditions for our Condition~\ref{cond:loss}. 

A subtle difference between our approach and Gr\"unwald and Mehta's is that the bounds in the above two displays are required {\em for all $\theta \in \Theta$}; Condition~\ref{cond:loss} only deals with {\em $\theta$ bounded away from $\theta^\star$}.  This difference arises because we directly target a bound on the Gibbs posterior probability, $\Pi_n(A_n)$, an integration over $A_n \not\ni \theta^\star$, whereas \citet{grunwald.2018} targets a bound on the Gibbs posterior mean of $\theta \mapsto R(\theta) - R(\theta^\star)$, an integration over all of $\Theta$.  In Example~3.8 of \citet{vanerven.etal.2015} the global requirement in \citet{grunwald.2018}'s condition is a disadvantage as \citet{vanerven.etal.2015} illustrate it creates some challenges in checking the bounds in the above two displays. 

Besides Condition~\ref{cond:loss}, other strategies for bounding $N_n(A_n)$ in \eqref{eq:gibbs.frac} have been employed in the literature.  For example, empirical process techniques are used in \citet{syring.martin.mcid, bhattacharya.martin.2020} to prove Gibbs posterior concentration in finite-dimensional applications.  Generally, such proofs hinge on a uniform law of large numbers, which can be challenging to verify in non-parametric problems.  \citet{chernozhukov.hong.2003} require the stronger condition that $\sup_{d(\theta,\theta^\star)>\delta} \{R_n(\theta)-R_n(\theta^\star)\} \to 0$ in $P^n$-probability.  When it holds they immediately recover an in-probability bound on $N_n(A_n)$, irrespective of $\omega$, but they do not obtain a finite-sample bound on $P^n\Pi_n(A_n)$.  Their analysis is limited to finite-dimensional parameters and  empirical risk functions $R_n(\theta)$ that are smooth in a neighborhood of $\theta^\star$.  This latter condition excludes important examples like the misclassification-error loss function; see Section~\ref{SS:classification}.

There are situations in which Condition~\ref{cond:loss} does not hold, and we address this formally in Section~\ref{s:extensions}. As an example, note that both Condition~\ref{cond:loss} and the witness condition used in \citet{grunwald.2018} are closely related to a Bernstein-type condition relating the first two moments of the excess loss: $P(\ell_\theta - \ell_{\theta^\star})^2 \leq c\{R(\theta)-R(\theta^\star)\}^\alpha$ for constants $(c, \,\alpha)>0$.  For the ``check'' loss used in quantile estimation, the Bernstein condition is generally not satisfied and, consequently, neither Gr\"unwald and Mehta's witness condition nor our Condition~\ref{cond:loss}---at least not in their original forms---can be verified.  Similarly, if the excess loss, $\ell_\theta(U) - \ell_{\theta^\star}(U)$, a function of data $U$, is heavy-tailed, then the moment-generating function bound in Condition~\ref{cond:loss} does not hold. For these cases, some modifications to the basic setup are required, which we present in Section~\ref{s:extensions}.  

\subsubsection{Prior distribution}
\label{SSS:prior}

Generally, the prior must place enough mass on certain ``risk-function" neighborhoods $G_n:=\{\theta: m(\theta,\theta^\star) \vee v(\theta,\theta^\star) \leq \eps_n^r\}$.  This is analogous to the requirement in the Bayesian literature that the prior place sufficient mass on Kullback--Leibler neighborhoods of $\theta^\star$; see, e.g., \citet{shen.wasserman.2001} and \citet{ggv2000}.  Some version of the following prior bound is needed
\begin{equation}
\label{eq:neighborhoods}
\log\Pi(G_n) \gtrsim -n\eps_n^r, \end{equation}
for $r$ as in Condition~\ref{cond:loss}.  
Lemma~1 in the appendix, along with \eqref{eq:neighborhoods}, provides a lower probability bound on the denominator term $D_n$ defined in \eqref{eq:gibbs.frac}.  That is, 
\begin{equation}
\label{eq:prior.bound.tmp}
P^n\bigl(D_n \leq b_n \bigr) 
\leq (n\eps_n^r)^{-1}, 
\end{equation}
where $b_n = \frac12 \Pi(G_n) e^{-2\omega n \eps_n^r}$.
Both the form of the risk-function neighborhoods and the precise lower bound in \eqref{eq:neighborhoods} depend on the concentration rate and the learning rate, so the results in Section~\ref{SS:exp.rates} all require their own version of the above prior bound.

\citet{grunwald.2018} only require bounds on the prior mass of the larger neighborhoods $\{\theta: m(\theta,\theta^\star) \leq \eps_n^r\}$.  Under their condition, we can derive a lower bound on $P^n D_n$ similar to Lemma~1 in \citet{martin.hong}.  However, our proofs require an in-probability lower bound on $D_n$, which in turn requires stronger prior concentration like that in \eqref{eq:neighborhoods}.  While there are important examples where the lower bounds on $m(\theta,\theta^\star)$ and $v(\theta,\theta^\star)$ are of different orders, none of the applications we consider here are of that type.  Therefore, the stronger prior concentration condition in \eqref{eq:neighborhoods} does not affect the rates we derive for the examples in Section~\ref{S:examples}.  Moreover, as discussed following the statement of Theorem~\ref{thm:rate} below, our finite-sample bounds are better than those in \citet{grunwald.2018}, a direct consequence of our method of proof that uses a smaller neighborhood $G_n$.

\subsection{Main results}
\label{SS:exp.rates}

In this section we present general results on Gibbs posterior concentration.  Proofs can be found in Section 1 of the supplementary material.  Our first result establishes Gibbs posterior concentration, under Condition~\ref{cond:loss} and a local prior condition, for sufficiently small constant learning rates.  

\begin{theorem}
\label{thm:rate}
Let $\eps_n$ be a vanishing sequence satisfying $n\eps_n^r \to \infty$ for a constant $r>0$.  Suppose the prior satisfies
\begin{equation} 
\label{eq:prior.rate}
\log\Pi(\{\theta: m(\theta,\theta^\star) \vee v(\theta,\theta^\star) \leq \eps_n^r\}) \geq -C_1n\eps_n^r,
\end{equation}
for $C_1>0$ and for divergence measure $d$, the same $r>0$ as above, and learning rate $\omega\in (0,\bar\omega)$ for some $\bar\omega>0$.  If the loss function satisfies Condition~\ref{cond:loss}, then the Gibbs posterior distribution in \eqref{eq:gibbs} has asymptotic concentration rate $\eps_n$, with $M_n \equiv M$ a large constant as in Definition~\ref{def:rate}. 
\end{theorem}

For a brief sketch of the proof bound the posterior probability of $A_n$ by
\begin{align*}
\Pi_n(A_n) & \leq \frac{N_n(A_n)}{D_n} \, 1(D_n > b_n) + 1(D_n \leq b_n) \\
& \leq b_n^{-1} N_n(A_n) + 1(D_n \leq b_n), 
\end{align*}
where $b_n$ is as above.
Taking expectation of both sides, and applying \eqref{eq:prior.bound.tmp} and the consequence \eqref{eq:consequence} of Condition~\ref{cond:loss}, we get 
\[ P^n \Pi_n(A_n) \leq b_n^{-1} e^{-K \omega n M_n^r \eps_n^r} + (n \eps_n^r)^{-1}. \]
Then the right-hand side is generally of order $(n \eps_n^r)^{-1} \to 0$.  To compare with the results in \citet[][Example~2]{grunwald.2018}, their upper bound on $P^n \Pi_n(A_n)$ is $O(M_n^{-1})$, which vanishes arbitrarily slowly. 

For the case where the risk minimizer in \eqref{eq:target} is not unique, certain modifications of the above argument can be made, similar to those in \citet[][Theorem~2.4]{kleijn}.  Roughly, to our Theorem~\ref{thm:rate} above, we would add the requirements that \eqref{eq:prior.rate} and Condition~\ref{cond:loss} hold uniformly in $\theta^\star \in \Theta^\star$, where $\Theta^\star$ is the set of risk minimizers.  Then virtually the same proof shows that $P^n\Pi_n(\{\theta: d(\theta, \Theta^\star) \gtrsim \eps_n\}) \to 0$, where $d(\theta,\Theta^\star) = \inf_{\theta^\star \in \Theta^\star} d(\theta,\theta^\star)$.

Theorem~\ref{thm:rate} is quite flexible and can be applied in a range of settings; see Section~\ref{S:examples}.  However, one case in which it cannot be directly applied is when $n \eps_n^r$ is bounded.  For example, in sufficiently smooth finite-dimensional problems, we have $r=2$ and the target rate is $\eps_n = n^{-1/2}$.  The difficulty is caused by the prior bound in \eqref{eq:prior.rate}, since it is impossible---at least with a fixed prior---to assign mass bounded away from 0 to a shrinking neighborhood of $\theta^\star$.  One option is to add a logarithmic factor to the rate, i.e., take $\eps_n = (\log n)^k n^{-1/2}$, so that $e^{-C n \eps_n^2}$ is a power of $n^{-1/2}$.  Alternatively, a refinement of the proof of Theorem~\ref{thm:rate} lets us avoid slowing down the rate.  

\begin{theorem}
\label{thm:rate.rootn}
Consider a finite-dimensional $\theta$, taking values in $\Theta \subseteq \RR^q$ for some $q \geq 1$.  Suppose that the target rate $\eps_n$ is such that $n\eps_n^r$ is bounded for some constant $r>0$.  If the prior $\Pi$ satisfies 
\begin{equation}
\label{eq:prior.rootn}
\Pi(\{\theta: m(\theta,\theta^\star) \vee v(\theta,\theta^\star) \leq \eps_n^2\}) \gtrsim \eps_n^q, 
\end{equation}
and if Condition~\ref{cond:loss} holds, then the Gibbs posterior distribution in \eqref{eq:gibbs}, with any learning rate $\omega \in (0,\bar\omega)$, has asymptotic concentration rate $\eps_n$ at $\theta^\star$ with respect to any divergence $d(\theta,\theta^\star)$ satisfying $\|\theta-\theta^\star\|\lesssim d(\theta,\theta^\star)\lesssim \|\theta-\theta^\star\|$ and for any diverging, positive sequence $M_n$ in Definition~\ref{def:rate}.  
\end{theorem}

The learning rate is critical to the Gibbs posterior's performance, but in applications it may be challenging to determine the upper bound $\bar\omega$ for which Condition~\ref{cond:loss} holds.  For a simple illustration, suppose the excess loss $\ell_\theta-\ell_\theta^\star$ is normally distributed with variance $\sigma^2(\theta)$  satisfying $\sigma^2(\theta)\leq c\{R(\theta)-R(\theta^\star)\}$, a kind of Bernstein condition.  In this case, 
Condition~\ref{cond:loss} holds for $d(\theta,\theta^\star) = \{R(\theta)-R(\theta^\star)\}^{1/2}$, with $r=2$, if $\omega < 2c^{-1}$.  Bernstein conditions can be verified in many practical examples, but the factor $c$ would rarely be known.  Consequently, we need $\omega$ to be sufficiently small, but the meaning of ``sufficiently small'' depends on unknowns involving $P$.  However, any positive, vanishing learning rate sequence $\omega=\omega_n$ satisfies $\omega_n \in (0,\bar\omega)$ for all sufficiently large $n$.  And if $\omega_n$ vanishes arbitrarily slowly, then it has no effect on the Gibbs posterior concentration rate; see Section~\ref{SS:quantile.reg}.  All we require to accommodate a vanishing learning rate is a slightly stronger $\omega_n$-dependent version of the prior concentration bound in \eqref{eq:prior.rate}.  

\begin{theorem}
\label{thm:rate.seq}
Let $\eps_n$ be a vanishing sequence and $\omega_n$ be a learning rate sequence satisfying $n\omega_n\eps_n^r \to \infty$ for a constant $r>0$.  Consider a Gibbs posterior distribution $\Pi_n=\Pi_n^{\omega_n}$ in \eqref{eq:gibbs} based on this sequence of learning rates.  If the prior satisfies 
\begin{equation}
\label{eq:prior.seq}
\log\Pi(\{\theta: m(\theta,\theta^\star) \vee v(\theta,\theta^\star) \leq \eps_n^r\}) \geq -C n \omega_n \eps_n^r, 
\end{equation}
and if Condition~\ref{cond:loss} holds for $\omega_n$, then the Gibbs posterior distribution in \eqref{eq:gibbs}, with learning rate sequence $\omega_n$, has concentration rate $\eps_n$ at $\theta^\star$ for a sufficiently large constant $M>0$ in Definition~\ref{def:rate}.
\end{theorem}

The proof of Theorem~\ref{thm:rate.seq} is almost identical to that of Theorem~\ref{thm:rate}, hence omitted.  But to see the basic idea, we mention two key observations.  First, since Condition~\ref{cond:loss} holds for $\omega_n$ for all sufficiently large $n$, and since $\omega_n$ is deterministic, by the same argument producing the bound in \eqref{eq:consequence}, we get
\[ P^nN_n(A_n)\leq e^{-Kn\omega_nM_n^r\eps_n^r} \quad \text{for all {\em sufficiently large} $n$}.\]
Second, the same argument producing the bound in \eqref{eq:prior.bound.tmp} shows that
\[P^n\bigl(D_n \leq \tfrac12\Pi(G_n)e^{-2 n\omega_n\eps_n^r}\bigr) \leq P^n\bigl(D_n \leq \tfrac12e^{-(C+2)n\omega_n\eps_n^r}\bigr) \leq (n\omega_n\eps_n^r)^{-1}.\]
Then the only difference between this situation and that in Theorem~\ref{thm:rate} is that, here, the numerator bound only holds for ``sufficiently large $n$,'' instead of for all $n$.  This makes Theorem~\ref{thm:rate.seq} is slightly weaker than Theorem~\ref{thm:rate} since there are no finite-sample bounds. 

When $\omega_n \equiv \omega$ the constant learning rate is absorbed by $C$ and there is no difference between the prior bounds in \eqref{eq:prior.rate} and \eqref{eq:prior.seq}.  But, the prior probability assigned to the $(m \vee v)$-neighborhood of $\theta^\star$ does not depend on $\omega_n$, so if it satisfies \eqref{eq:prior.rate}, then the only way it could also satisfy \eqref{eq:prior.seq} is if $\eps_n$ is bigger than it would have been without a vanishing learning rate.  In other words, Theorem~\ref{thm:rate} requires $n\eps_n^r\rightarrow\infty$ whereas Theorem~\ref{thm:rate.seq} requires $n\omega_n\eps_n^r\rightarrow \infty$, which implies that for a given vanishing $\omega_n$ Theorem~\ref{thm:rate} potentially can accommodate a faster rate $\eps_n$.  Therefore, we see that a vanishing learning rate can slow down the Gibbs posterior concentration rate if it does not vanish arbitrarily slowly.  There are applications that require the learning rate to vanish at a particular $n$-dependent rate, and these tend to be those where adjustments like in Section~\ref{s:extensions} are needed; see Sections~\ref{SSS:heavy} and \ref{SSS:tsybakov}.

If we can use the data to estimate $\bar\omega$ consistently, then it makes sense to choose a learning rate sequence depending on this estimator. Suppose our data-dependent learning rate $\hat\omega_n$ satisfies $\hat\omega_n < \bar\omega$ with probability converging to $1$ as $n\rightarrow\infty$.  For $\omega\equiv \hat\omega_n$ the conclusion of Theorem~\ref{thm:rate} holds for all sufficiently large $n$; and see Section 4.2 in the Supplementary Material.  One advantage of this strategy is that it avoids using a vanishing learning rate, which may slow concentration.

\begin{theorem}
\label{thm:rate.est}
Fix a positive deterministic learning rate sequence $\omega_n$ such that the conditions of Theorem~\ref{thm:rate.seq} hold and as a result $\Pi_n^{\omega_n}$ has asymptotic concentration rate $\eps_n$.  Consider a random learning rate sequence $\hat\omega_n$ satisfying 
\begin{equation} P^n(\omega_n/2<\hat\omega_n<\omega_n)\rightarrow 1, \quad n \to \infty. \end{equation}
Then $\Pi_n^{\hat\omega_n}$, the Gibbs posterior distribution in \eqref{eq:gibbs} scaled by the random learning rate sequence $\hat\omega_n$, also has concentration rate $\eps_n$ at $\theta^\star$ for a sufficiently large constant $M>0$ in Definition~\ref{def:rate}.  
\end{theorem}

\subsection{Checking Condition~\ref{cond:loss}}
\label{ss:check1}

Of course, Condition~\ref{cond:loss} is only useful if it can be checked in practically relevant examples.  As mentioned in Section~\ref{SS:discussion}, Gr\"unwald and Mehta's strong central and witness conditions are sufficient for Condition~\ref{cond:loss}.  A pair of slightly stronger, but practically verifiable conditions require the excess loss is sub-exponential with first and second moments that obey a Bernstein condition.  For the examples we consider in Section~\ref{S:examples} where Condition~\ref{cond:loss} can be used, these two conditions are convenient.

\subsubsection{Bounded excess losses}
\label{sss:check1}

For bounded excess losses, i.e., $\ell_\theta(u) - \ell_{\theta^\star}(u) < C$ for all $(\theta,u)$, Lemma~7.26 in \citep{wasserman} gives:
\begin{equation}
\label{eq:wasserman}
Pe^{-\omega(\ell_\theta - \ell_{\theta^\star})}\leq \exp\bigl[ -\omega m(\theta,\theta^\star)+\omega^2v(\theta,\theta^\star)\bigl\{\tfrac{\exp(C\omega)-1-C\omega}{C^2\omega^2}\bigr\} \bigr].
\end{equation}
Whether Condition~\ref{cond:loss} holds depends on the choice of $\omega$ and the relationship between $m(\theta,\theta^\star)$ and $v(\theta,\theta^\star)$ as defined by the Bernstein condition:
\begin{equation}
\label{eq:bernstein}
v(\theta,\theta^\star)\leq C_1m(\theta,\theta^\star)^\alpha, \quad \alpha \in (0,1], \quad C_1 > 0. 
\end{equation}
Denote the bracketed expression in the exponent of \eqref{eq:wasserman} by $C(\omega)=(C^2\omega^2)^{-1}\{\exp(C\omega)-1-C\omega\}$.  When $\alpha = 1$, the $m$ and $v$ functions are of the same order and, if $d(\theta,\theta^\star)=m(\theta,\theta^\star)$, and $\omega C(\omega) \leq (2C_1)^{-1}$ then Condition~\ref{cond:loss} holds with $K = 1/2$.  Since $C(\omega)\approx C\omega$ for small $\omega$, it suffices to take the learning rate a sufficiently small constant.

On the other hand, when $v$ is larger than $m$, i.e., the Bernstein condition holds with $\alpha < 1$, Condition~\ref{cond:loss} requires that $\omega$ depend on $\alpha$ and $m(\theta,\theta^\star)$.  Suppose $m(\theta,\theta^\star)\geq d(\theta,\theta^\star) > \eps_n$.  For all small enough $\eps_n$, we can set $\omega_n = (4C)^{-1/2}\eps_n^{(1-\alpha)/2}$ and Condition~\ref{cond:loss} is again satisfied with $K = 1/2$.  

The above strategies are implemented in the classification examples in Section~\ref{SS:classification} where we also discuss connections to the Tsybakov \citep{tsybakov} and Massart \citep{massart.nedelec.2006} conditions.  For the former, see Theorem~22 and the subsequent discussion in \citet{grunwald.2018}, where the learning rate ends up being a vanishing sequence depending on the concentration rate and the Bernstein exponent $\alpha$.     

\subsubsection{Sub-exponential excess losses}

Unbounded but light-tailed loss differences $\ell_\theta - \ell_{\theta^\star}$ may also satisfy Condition~\ref{cond:loss}.  Both {\em sub-exponential} and {\em sub-Gaussian} random variables \citep[][Sec.~2.3--4]{boucheron.etal.2012} admit an upper bound on their moment-generating functions.  When the loss difference $\ell_\theta - \ell_{\theta^\star}$ is sub-Gaussian
\begin{equation} P e^{-\omega (\ell_\theta - \ell_{\theta^\star})} \leq \exp\bigl\{ \tfrac{\omega^2}{2}  \sigma^2(\theta,\theta^\star) - \omega m(\theta,\theta^\star) \bigr\}, \end{equation}
for all $\omega$, where the variance proxy $\sigma^2(\theta,\theta^\star)$ may depend on $(\theta,\theta^\star)$.  If $\ell_\theta - \ell_{\theta^\star}$ is sub-exponential, then the above bound holds for all $\omega \leq (2b)^{-1}$ for some $b < \infty$ indexing the tail behavior of $P$.  

If $\sigma^2(\theta,\theta^\star)$ is upper-bounded by $L m(\theta,\theta^\star)$ for a constant $L > 0$, then the above bound can be rewritten as
\begin{equation}
P e^{-\omega (\ell_\theta - \ell_{\theta^\star})} \leq \exp\bigl\{ - \omega m(\theta,\theta^\star)(1-\tfrac{\omega L}{2}) \bigr\}.
\end{equation}
Then Condition~\ref{cond:loss} holds if the loss difference is sub-Gaussian and sub-exponential if $\omega<2L^{-1}$ and $\omega< 2L^{-1} \wedge (2b)^{-1}$, respectively.  

In practice it may be awkward to assume $\ell_\theta(U) - \ell_{\theta^\star}(U)$ is sub-exponential, but in certain problems it is sufficient to make such an assumption about features of $U$, which may be more reasonable. See Section~\ref{SS:pred_sq} for an application of this idea to a fixed-design regression problem where the fact the response variable is sub-Gaussian implies the excess loss is itself sub-Gaussian.  

\section{Extensions}
\label{s:extensions}

\subsection{Locally sub-exponential type loss functions}
\label{SSS:local.exp.loss}

In some cases the moment generating function bound in Condition~\ref{cond:loss} can be verified in a neighborhood of $\theta^\star$ but not for all $\theta\in\Theta$.  For example, suppose $\theta \mapsto \ell_\theta(u)$ is Lipschitz with respect to a metric $\|\cdot\|$ with uniformly bounded Lipschitz constant $L=L(u)$, and that, for some $\delta>0$,
\begin{align*}
    \|\theta-\theta^\star\| < \delta \implies v(\theta,\theta^\star)\lesssim m(\theta,\theta^\star) \quad \text{and} \quad
\|\theta-\theta^\star\| > \delta \implies v(\theta,\theta^\star)\lesssim m(\theta,\theta^\star)^2.
\end{align*}
That is, the Bernstein condition \eqref{eq:bernstein} holds but for different values of $\alpha$ depending on $\theta$.  This is the case for quantile regression; see Section~\ref{SS:quantile.reg}.  This class of problems, where the Bernstein exponent varies across $\Theta$, apparently has not been considered previously.  For example, \citet{grunwald.2018} only consider cases where the Bernstein exponent $\alpha$ is constant across the entire parameter space; consequently, their results assume $\Theta$ is bounded \citep[e.g.,][Example~10]{grunwald.2018}, so \eqref{eq:bernstein} holds trivially with exponent $\alpha=0$.  

To address this problem, our idea is a simple one.  We propose to introduce a sieve which is large enough that we can safely assume it contains $\theta^\star$ but small enough that, on which, the $m$ and $v$ functions can be appropriately controlled.  Towards this, let $\Theta_n$ be an increasing sequence of subsets of $\Theta$, indexed by the sample size $n$.  The ``size'' of $\Theta_n$ will play an important role in the result below.  While more general sieves are possible, to keep the notion of size concrete, let $\Theta_n = \{\theta \in \Theta: \|\theta\| \leq \Delta_n\}$, so that size is controlled by the non-decreasing sequence $\Delta_n > 0$, which would typically satisfy $\Delta_n \to \infty$ as $n \to \infty$.  The metric $\|\cdot\|$ in the definition of $\Theta_n$ is at the user's discretion, it would be chosen so that $\theta \in \Theta_n$ provides information that can be used to control the excess loss $\ell_\theta - \ell_{\theta^\star}$.  This leads to the following straightforward modification of Condition~\ref{cond:loss}.  

\begin{condition}
\label{cond:loss.b}
For $\Theta_n$ with size controlled by $\Delta_n$, there exists an interval $(0,\bar \omega)$, a constant $r>0$, and a sequence $K_n=K(\Delta_n)>0$ such that, for all $\omega \in (0,\bar \omega)$ and for all small $\delta > 0$, 
\begin{equation}
\label{eq:cond.loss.local}
\theta \in \Theta_n \; \text{ and } \; d(\theta; \theta^\star) > \delta \implies P e^{-\omega
(\ell_\theta-\ell_{\theta^\star})} < e^{-\omega K_n^r\delta^r}. 
\end{equation}
\end{condition}

Aside from the restriction $\Theta_n$, the key difference between the bounds here and in Condition~\ref{cond:loss} is in the exponent.  Instead of there being a constant $K$, there is a sequence $K_n$ which is determined by the sieve's size, controlled by $\Delta_n$.  If the sequence $K_n$ is increasing, as we expect it will be, then we can anticipate that the overall concentration rate would be adversely affected---unless the learning rate is vanishing suitably fast.  The following theorem explains this more precisely.  

Condition~\ref{cond:loss.b} can be used exactly as Condition~\ref{cond:loss} to upper bound $P^n N_n(A_n \cap \Theta_n)$.  However, the Gibbs posterior probability assigned to $A_n \cap \Theta_n^\comp$ must be handled separately.

\begin{theorem}
\label{thm:rate.b}
Let $\Theta_n$ be a sequence of subsets for which the loss function satisfies Condition~\ref{cond:loss.b} for a sequence $K_n>0$, a constant $r>0$, and a learning rate $\omega_n\in(0,\overline\omega)$ for all sufficiently large $n$.  Let $\eps_n$ be a vanishing sequence satisfying $n\omega_nK_n^r\eps_n^r\rightarrow\infty$.  Suppose the prior satisfies
\begin{equation} 
\label{eq:prior.rate.b}
\log\Pi(\{\theta: m(\theta,\theta^\star) \vee v(\theta,\theta^\star) \leq (K_n\eps_n)^r\}) \gtrsim -Cn\omega_nK_n^r\eps_n^r,
\end{equation}
for some $C>0$ and the same $K_n$, $r$, and $\omega_n$ as above.  Then the Gibbs posterior in \eqref{eq:gibbs} satisfies 
\[ \limsup_{n \to \infty} P^n \Pi_n(A_n) \leq \limsup_{n \to \infty} P^n \Pi_n(A_n \cap \Theta_n^\comp). \]
Consequently, if 
\begin{equation} 
\label{eq:post.b}
P^n\Pi_n(\Theta_n^\comp) \to 0 \quad \text{as $n \to \infty$},
\end{equation}  
then the Gibbs posterior has concentration rate $\eps_n$ for all large constants $M>0$ in Definition~\ref{def:rate}. 
\end{theorem}

There are two aspects of Theorem~\ref{thm:rate.b} that deserve further explanation.  We start with the point about separate handling of $\Theta_n^\comp$.  Condition \eqref{eq:post.b} is easy to check for well-specified Bayesian posteriors, but these results do not carry over even to a misspecified Bayesian model.  Of course, \eqref{eq:post.b} can always be arranged by restricting the support of the prior distribution to $\Theta_n$, which is the suggestion in \citet[][Theorem~2.3]{kleijn} for the Bayesian case and how we handle \eqref{eq:post.b} for our infinite-dimensional Gibbs example in Section~\ref{SS:quantile.reg}.  However, as \eqref{eq:post.b} suggests, this is not entirely necessary.  Indeed, \citet{kleijn} describe a trade-off between model complexity and prior support, and offer a more complicated form of their sufficient condition, which they do not explore in that paper.  The fact is, without a well-specified likelihood, checking a condition like our \eqref{eq:post.b} or Equation (2.13) in \citet{kleijn} is a challenge, at least in infinite-dimensional problems. For finite-dimensional problems, it may be possible to verify \eqref{eq:post.b} directly using properties of the loss functions.  For example, we use convexity of the check loss function for quantile estimation to verify \eqref{eq:post.b} in Section~\ref{SS:quantile_reg}.

Next, how might one proceed to check Condition~\ref{cond:loss.b}?  Go back to the Lipschitz loss case at the start of this subsection.  For a sieve $\Theta_n$ as described above, if $\theta$ and $\theta^\star$ are in $\Theta_n$, then $\|\theta-\theta^\star\|$ is bounded by a multiple of $\Delta_n$.  This, together with the Lipschitz property and \eqref{eq:wasserman}, implies
\[ P e^{-\omega_n(\ell_\theta - \ell_{\theta^\star})} \leq \exp\bigl\{ C_n \omega_n^2 v(\theta,\theta^\star) - \omega_n m(\theta,\theta^\star) \bigr\}, \]
where $C_n = O(1+\Delta_n\omega_n)$.  Suppose the user chooses $\omega_n$ and $\Delta_n$ such that $\omega_n \Delta_n = O(1)$; then we can replace $C_n$ by a constant $C$.  If there exist functions $g$ and $G$ such that
\begin{equation} 
m(\theta,\theta^\star) \geq g(\Delta_n) \|\theta-\theta^\star\|^2 \quad \text{and} \quad v(\theta,\theta^\star)\leq G(\Delta_n) \|\theta-\theta^\star\|^2, \quad \theta \in \Theta_n, 
\end{equation}
then the above upper bound simplifies to 
\[ \exp[ -\omega_n \|\theta-\theta^\star\|^2 \{g(\Delta_n) - C \omega_n G(\Delta_n)\}], \quad \theta \in \Theta_n. \]
Then Condition~\ref{cond:loss.b} holds with $K_n = g(\Delta_n) - C \omega_n G(\Delta_n)$, and it remains to balance the choices of $\omega_n$ and $\Delta_n$ to achieve the best possible concentration rate $\eps_n$; see Section~\ref{SS:quantile.reg}.  

\subsection{Clipping the loss function}
\label{ss:surrogate}

When the excess loss is heavy-tailed, i.e., not sub-exponential, like in Section~\ref{SS:heavy}, its moment-generating function does not exist and, therefore, Condition~\ref{cond:loss} cannot be satisfied.  In this section, we assume the loss function $\ell_\theta(u)$ is non-negative or lower-bounded by a negative constant.  In the latter case we can work with the shifted loss---$\ell_\theta$ minus its lower bound. Many practically useful loss functions are non-negative, including squared-error loss, which we cover in Section~\ref{SS:heavy}.

In such cases, it may be reasonable to replace the heavy-tailed loss with a clipped version
\[ \ell_\theta^n(u) = \ell_\theta(u) \wedge t_n, \]
where $t_n > 0$ is a diverging clipping sequence.  Since $\ell_\theta^n(u)$ is bounded in $(u,\theta)$ for each fixed $n$, the strategy for checking Condition~\ref{cond:loss} described in Section~\ref{sss:check1} suggests that, for certain choices of $(t_n, \eps_n, \omega_n)$, $\Pi_n$ places vanishing mass on the sequence of sets $\{\theta: P\ell^n_\theta - P\ell_{\theta_n^\star}^n > \eps_n\}$, where $\theta_n^\star = \arg\min P\ell_\theta^n$.  This makes the $\theta_n^\star$ the (moving) target of the Gibbs posterior, instead of the fixed $\theta^\star$. On the other hand, if the loss function admits more than one finite moment, then for a corresponding, increasing clipping sequence $t_n$ the clipped risk neighborhoods of $\theta^\star_n$ contain the risk neighborhoods of $\theta^\star$ for all large $n$, that is, 
\begin{equation}
\label{eq:nbhds}
\{\theta:R(\theta) - R(\theta^\star) > C\eps_n\}\subset \{\theta: P\ell^n_\theta - P\ell_{\theta_n^\star}^n > \eps_n\}, \quad \text{for some $C > 0$}.
\end{equation}
Then we further expect concentration of the clipped loss-based Gibbs posterior at $\theta^\star$ with respect to the excess risk divergence at rate $\eps_n$.  Condition~\ref{cond:loss.c} and Theorem~\ref{thm:rate.surrogate} below provide a set of sufficient conditions under which these expectations are realized and a concentration rate can be established. 

\begin{condition}
\label{cond:loss.c}
Let $\ell_\theta$ be the loss.  Define $\ell_\theta^n = \ell_\theta \wedge t_n$ as the clipped loss and $\Theta_n = \{\theta: \|\theta\| \leq \Delta_n\}$ as a sieve, depending on constants $t_n, \Delta_n \to \infty$.  
\begin{enumerate}
\item For some $s > 1$, the sequence $B_n = \sup_{\theta \in \Theta_n} P|\ell_\theta|^s$ is finite for all $n$.
\item There exists a sequence $\bar \omega_n > 0$, and a sequence $K_n>0$, such that for all sequences $0<\omega_n\leq \bar \omega_n$ and for all sufficiently small $\delta > 0$, 
\begin{equation}
\label{eq:cond.loss.surrogate}
\theta \in \Theta_n \text{ and } P\ell_\theta^n - P\ell_{\theta_n^\star}^n > \delta \implies P e^{-\omega_n
(\ell_\theta^n-\ell^n_{\theta_n^\star})} < e^{-K_n \omega_n \delta} .
\end{equation}
\end{enumerate}
\end{condition}

\begin{theorem}
\label{thm:rate.surrogate}
For a given loss $\ell_\theta$ and sieve $\Theta_n$, suppose that Condition~\ref{cond:loss.c} holds for $(\omega_n, \Delta_n, t_n)$; and let $B_n=B_n(s)$ be as defined in Condition~\ref{cond:loss.c}.1, for $s > 1$. Let $\eps_n$ be a vanishing sequence such that $n\omega_nK_n\eps_n\rightarrow\infty$, and suppose the prior satisfies 
 \begin{equation}
    \label{eq:prior.surrogate}
    \log\Pi(\{\theta: m_n(\theta,\theta_n^\star) \vee v_n(\theta,\theta_n^\star) \leq K_n\eps_n\}) \gtrsim -Cn\omega_nK_n\eps_n.
\end{equation}
Then the Gibbs posterior in \eqref{eq:gibbs} based on the clipped loss $\ell_\theta^n$ satisfies
\[ \limsup_{n \to \infty} P^n \Pi_n(A_n) \leq \limsup_{n \to \infty} P^n \Pi_n(A_n \cap \Theta_n^\comp) \]
where $A_n:=\{\theta:R(\theta)-R(\theta^\star)>\eps_n \vee B_nt_n^{1-s}\}$.  Consequently, if
\[P^n\Pi_n(\Theta_n^\comp) \to 0 \quad \text{as $n \to \infty$},\]
then the Gibbs posterior has asymptotic concentration rate $\eps_n \vee B_nt_n^{1-s}$ at $\theta^\star$ with respect to the excess risk $d(\theta,\theta^\star)=R(\theta)-R(\theta^\star)$. 
\end{theorem}

The setup here is more complicated than in previous sections, so some further explanation is warranted.  First, we sketch out how Condition~\ref{cond:loss.c}.1 leads to the critical property \eqref{eq:nbhds}.  A well-known bound on the expectation of a non-negative random variable, plus the moment bound in Condition~\ref{cond:loss.c}.1, for $s > 1$, and Markov's inequality leads to  
\begin{align}
\label{eq:moments}
P\ell_\theta1(\ell_\theta>t_n) =  \int_{t_n}^\infty P(\ell_\theta > x) \, dx 
&\leq  B_n \int_{t_n}^\infty x^{-s} \, dx = B_n t_n^{1-s}. 
\end{align}
This, in turn, implies $R(\theta) = P\ell_\theta^n + O(B_n t_n^{1-s})$.  If $B_n t_n^{1-s}\rightarrow 0$, then the difference between the risk and the clipped risk is vanishing, so we can bound $P\ell_\theta^n - P\ell_{\theta_n^\star}^n$ by a multiple of the excess risk $R(\theta) - R(\theta^\star)$ for all sufficiently large $n$.  In the application we explore in Section~\ref{SS:heavy}, $B_n$ is related to the radius of the sieve $\Theta_n$, which grows logarithmically in $n$, while $t_n$ is related to the polynomial tail behavior of $\ell_\theta$, so that $B_n t_n^{1-s}\rightarrow 0$ happens naturally if $s>1$. Additional details are given in the proof of Theorem~\ref{thm:rate.surrogate}, which can be found in Appendix~\ref{S:proofs}.

Second, how might Condition~\ref{cond:loss.c}.2 be checked? Start by defining the excess clipped risk $m_n(\theta; \theta_n^\star) = P\ell_\theta^n - P\ell_{\theta_n^\star}^n$ and the corresponding variance $v_n(\theta, \theta_n^\star) = P(\ell_\theta^n - \ell_{\theta_n^\star}^n)^2 - m_n(\theta,\theta_n^\star)^2$.  Now suppose it can be shown that there exists a function $G$ such that 
\[ v_n(\theta, \theta_n^\star) \leq G(\Delta_n) \, m_n(\theta,\theta_n^\star), \quad \text{for all $\theta\in\Theta_n$}, \]
$\Delta_n$ is the size index of the sieve.  This amounts to the excess clipped loss satisfying a Bernstein condition \eqref{eq:bernstein}, with exponent $\alpha=1$.  The excess clipped loss itself is $\lesssim t_n$ (and maybe substantially smaller, depending on form of $\ell_\theta$). So, we can apply the moment-generating function bound in \eqref{eq:wasserman} for bounded excess losses to get 
\[ Pe^{-\omega_n(\ell_\theta^n - \ell_{\theta^\star}^n)} < e^{-K_n\omega_n m_n(\theta, \theta_n^\star)}, \]
where $K_n = \{1 - C \omega_n G(\Delta_n)\}$, for a constant $C > 0$, provided that $\omega_n t_n = O(1)$.  Now it is easy to see that the above display implies Condition~\ref{cond:loss.c}.2. 

For the rate calculation with respect to the excess risk, the decomposition of $R(\theta)$ based on \eqref{eq:moments} implies we need $\eps_n \geq B_nt_n^{1-s}$, subject to the constraint $n\omega_nK_n\eps_n \rightarrow \infty$, for $K_n$ as above.  As we see in Section~\ref{SS:heavy}, $B_n$, $\Delta_n$, and $K_n$ can often be taken as powers of $\log n$, so the critical components determining the rate are $s$ and $t_n$.  The optimal rate depends on the upper bound of the the excess clipped loss, which, in the worst case, equals $t_n$. To apply \eqref{eq:wasserman} we need the learning rate to vanish like the reciprocal of this bound, so we take $\omega_n = t_n^{-1}$.  Then, we determine $\eps_n$ to satisfy $n t_n^{-1}\eps_n\rightarrow \infty$ and $\eps_n\geq  t_n^{1-s}$, up to a log term.  The clipping sequence $t_n \approx n^{1/s}$ is sufficient, and yields the rate $\eps_n \approx n^{1/s - 1}$, modulo log terms. 

\section{Examples}
\label{S:examples}

This section presents several illustrations of the general theory presented in Sections~\ref{S:main} and \ref{s:extensions}.  The strategies laid out in Sections~\ref{ss:check1}, \ref{SSS:local.exp.loss}, and \ref{ss:surrogate} are put to use in the following examples to verify our sufficient conditions for Gibbs posterior concentration.  All proofs of results in this section can be found in Appendix~\ref{proofs:examples}.

\subsection{Quantile regression}
\label{SS:quantile_reg}
Consider making inferences on the $\tau^{\text{th}}$ conditional quantile of a response $Y$ given a predictor $X=x$.  We model this quantile, denoted $Q_{Y|X=x}(\tau)$, as a linear combination of functions of $x$, that is, $Q_{Y|X=x}(\tau) = \theta^\top f(x)$, for a fixed, finite dictionary of functions $f(x)=(f_1(x), \dots, f_J(x))^\top$ and where $\theta=(\theta_1, \dots, \theta_J)^\top$ is a coefficient vector with $\theta\in\Theta$.  Here we assume the model is well-specified so the true conditional quantile is ${\theta^\star}^\top f(x)$ for some $\theta^\star\in\Theta$.  The standard \emph{check loss} for quantile estimation is
\begin{equation}\ell_\theta(u) = (y - \theta^\top f(x))(\tau - 1\{y<\theta^\top f(x)\}).\end{equation}
We show $\theta^\star$ minimizes $R(\theta)$ in the proof of Proposition~\ref{prop:quantreg} below.   It can be shown that $\theta \mapsto \ell_\theta(u)$ is $L$-Lipschitz, with $L<1$, and convex, so the strategy in Section~4.1 of the main article is helpful here for verifying Condition~2 and Lemma~\ref{prop:convex_cons} in Section~\ref{proof:prop:convex_cons} can be used to verify \eqref{eq:post.b}.  

Inference on quantiles is a challenging problem from a Bayesian perspective because the quantile is well-defined irrespective of any particular likelihood.  \citet{sriram.2013} interprets the check loss as the negative log-density of an asymmetric Laplace distribution and constructs a corresponding pseudo-posterior using this likelihood, but their posterior is effectively a Gibbs posterior as Definition~3 of the main article.

With a few mild assumptions about the underlying distribution $P$, our general result in Theorem~2 can be used to establish Gibbs posterior concentration at rate $n^{-1/2}$.  

\begin{asmp}
\label{asmp:quantile_reg}
\mbox{}
\begin{enumerate}
\item The marginal distribution of $X$ is such that $P ff^\top $ exists and is positive definite; 
\item the conditional distribution of $Y$, given $X=x$, has at least one finite moment and admits a continuous density $p_x$ such that $p_x(\theta^{\star \top} f)$ is bounded away from zero for $P$-almost all $x$; and
\item the prior $\Pi$ has a density bounded away from 0 in a neighborhood of $\theta^\star$.
\end{enumerate}
\end{asmp}
\begin{proposition}
\label{prop:quantreg}
Under Assumption~\ref{prop:quantreg}, if the learning rate is sufficiently small, then the Gibbs posterior concentrates at $\theta^\star$ with rate $\eps_n = n^{-1/2}$ with respect to $d(\theta,\theta^\star) = \|\theta-\theta^\star\|$.
\end{proposition}

\subsection{Area under receiver operator characteristic curve}
\label{S:auc}

The receiver operator characteristic (ROC) curve and corresponding area under the curve (AUC) are diagnostic tools often used to judge the effectiveness of a binary classifier.  Suppose a binary classifier produces a score $U$ characterizing the likelihood an individual belongs to Group $1$ versus Group $0$.  We can estimate an individual's group by $1(U>t)$ where different values of the cutoff score $t$ may provide more or less accurate estimates.  Suppose $U_0$ and $U_1$ are independent scores corresponding to individuals from Group $0$ and Group $1$, respectively.  The specificity and sensitivity of the test of $H_0:\text{individual }i\text{ belongs to Group 0}$ that rejects when $U>t$ are defined by $\text{spec}(t) = P(U_0<t)$ and $\text{sens}(t) = P(U_1>t)$.  When the type 1 and 2 errors of the test are equally costly the optimal cutoff is the value of $t$ maximizing $1-\text{spec}(t)+\text{sens}(t)$, or, in other words, the test maximizing the sum of power and one minus the type 1 error probability.  The ROC is the parametric curve $(1-\text{spec}(t), \text{sens}(t))$ in $[0,1]^2$ which provides a graphical summary of the tradeoff between Type~1 and Type~2 errors for different choices of the cutoff.  The AUC, equal to $P(U_1>U_0)$, gives an overall numerical summary of the quality of the binary classifier, independent of the choice of threshold.  

Our goal is to make posterior inferences on the AUC, but the usual Bayesian approach immediately runs into the kinds of problems we see in the examples in Sections~\ref{SS:quantile_reg} and later in \ref{S:mcid}.  The parameter of interest is one-dimensional, but it depends on a completely unknown joint distribution $P$.  Within a Bayesian framework, the options are to fix a parametric model for this joint distribution and risk model misspecification or work with a complicated nonparametric model.  \citet{wang.martin.auc} constructed a Gibbs posterior for the AUC that avoids both of these issues.

Suppose $U_{0,1}, \ldots, U_{0,m}$ and $U_{1,1}, \ldots, U_{1,n}$ denote random samples of size $m$ and $n$, respectively, of binary classifier scores for individuals belonging to Groups 0 and 1, and denote $\theta = P(U_1 > U_0)$.  \citet{wang.martin.auc} consider the loss function
\begin{equation}\ell_\theta(u_0, u_1) = \{\theta - 1(u_1>u_0)\}^2, \quad \theta\in [0,1],\end{equation}
for which the risk satisfies $R(\theta) = (\theta-\theta^\star)^2$.  If we interpret $m=m_n$ as a function of $n$, then it makes sense to write the empirical risk function as 
\begin{equation}R_n(\theta) = \frac{1}{mn}\sum_{i=1}^m\sum_{j=1}^n \{\theta - 1(U_{1,i} > U_{0,j})\}^2. \end{equation}
Note the minimizer of the empirical risk is equal to
\begin{equation}\hat\theta_n = \frac{1}{mn}\sum_{i=1}^m\sum_{j=1}^n 1(U_{1,i} > U_{0,j}).\end{equation}

\citet{wang.martin.auc} prove concentration of the Gibbs posterior at rate $n^{-1/2}$ under the following assumption.

\begin{asmp}
\label{asmp:auc}
\mbox{}
\begin{enumerate}
\item The sample sizes $(m,n)$ satisfy $m(m+n)^{-1}\rightarrow \lambda \in (0,1)$.
\item The prior distribution has a density function $\pi$ that is bounded away from zero in a neighborhood of $\theta^\star$. 
\end{enumerate}
\end{asmp}

\citet{wang.martin.auc} note that their concentration result holds for fixed learning rates and deterministic learning rates that vanish more slowly that $\min(m,n)^{-1}$.  As discussed in \citet{syring.martin.scaling} one motivation for choosing a particular learning rate is to calibrate Gibbs posterior credible intervals to attain a nominal coverage probability, at least approximately.  With this goal in mind, \citet{wang.martin.auc} suggest the following random learning rate.  Define the covariances 
\begin{align}
\tau_{10} & = \text{Cov}\{1(U_{1,1}>U_{0,1}), \, 1(U_{1,1}>U_{0,2})\}\nonumber \\ 
\tau_{01} & = \text{Cov}\{1(U_{1,1}>U_{0,1}), \, 1(U_{1,2}>U_{0,1})\}. 
\end{align}
\citet{wang.martin.auc} note the asymptotic covariance of $\hat\theta_n$ is given by
\begin{equation}\frac{1}{m+n}\Bigl(\frac{\tau_{10}}{\lambda} + \frac{\tau_{01}}{1-\lambda}\Bigr),\end{equation}
and that the Gibbs posterior variance can be made to match this, at least asymptotically, by using the random learning rate
\begin{equation}\hat \omega_n = \frac{m+n}{2mn}\Bigl(\frac{\hat\tau_{10}}{\lambda}+\frac{\hat\tau_{01}}{1-\lambda}\Bigr)^{-1},\end{equation}
where $\hat\tau_{10}$ and $\hat\tau_{01}$ are the corresponding empirical covariances:
\begin{align}
\hat\tau_{10} 
&= \frac{2}{mn(n-1)}\sum_{i=1}^m\sum_{j\ne j'} 1(U_{1,i}>U_{0,j})1(U_{1,i}>U_{0,j'}) - \hat\theta_n^2, \nonumber\\
\hat\tau_{01} 
&= \frac{2}{nm(m-1)}\sum_{j=1}^n\sum_{i\ne i'} 1(U_{1,i}>U_{0,j})1(U_{1,i'}>U_{0,j}) - \hat\theta_n^2.
\end{align}
The hope is that the Gibbs posterior with the learning rate $\hat\omega_n$ has asymptotically calibrated credible intervals.  It turns out that the concentration result in \citet{wang.martin.auc} along with Theorem~4 imply the Gibbs posterior with a slightly adjusted version of learning rate $\hat\omega_n$ also concentrates at rate $n^{-1/2}$.  The adjustment to the learning rate has the effect of slightly widening Gibbs posterior credible intervals, so their asymptotic calibration is not adversely affected.    

\begin{proposition}
\label{prop:auc}
Suppose Assumption~\ref{asmp:auc} holds and let $a_n$ denote any diverging sequence.  Then, the Gibbs posterior with learning rate $a_n\hat\omega_n$ concentrates at rate $n^{-1/2}$ with respect to $d(\theta,\theta^\star)=|\theta-\theta^\star|$.
\end{proposition}

\subsection{Finite-dimensional regression with squared-error loss}
\label{SS:heavy}

Consider predicting a response $y \in\RR$ using a linear function $x^\top \theta$ by minimizing the sum of squared-error losses $\ell_\theta (u) = (y - x^\top \theta)^2$, with $u=(x,y) \in \RR^{J+1}$, over a parameter space $\theta \in \Theta \subseteq \RR^{J}$. Suppose the covariate-response variable pairs $(X_i,Y_i)$ are iid with $X$ taking values in a compact subset $\mathcal{X}\subset \RR^J$.   To complement this example we present a more flexible, non-parametric regression problem in Section~\ref{SS:pred_sq} below.  For the current example, we focus on how the tail behavior of the response variable affects posterior concentration; see Assumptions~\ref{asmp:light} and \ref{asmp:heavy} below.  

\subsubsection{Light-tailed response}

When the response is sub-exponential so is the excess loss, and by the argument outlined in Section~\ref{ss:check1} we can verify Condition~\ref{cond:loss} for $d(\theta,\theta^\star)=\|\theta - \theta^\star\|_2$ and with $r=2$.  Then, the Gibbs posterior distribution concentrates at rate $n^{-1/2}$ as a consequence of Theorem~\ref{thm:rate.rootn}.

\begin{asmp}
\label{asmp:light} \mbox{}
\begin{enumerate}
    \item The response $Y$, given $x$, is sub-exponential with parameters $(\sigma^2, \, b)$ for all $x$; 
    \item $X$ is bounded and its marginal distribution is such that $PXX^\top$ exists and is positive definite with eigenvalues bounded away from $0$; and
    \item the prior $\Pi$ has a density bounded away from $0$ in a neighborhood of $\theta^\star$.
\end{enumerate}
\end{asmp}

\begin{proposition}
\label{prop:light}
If Assumption~\ref{asmp:light} holds, and the learning rate $\omega$ is a sufficiently small constant, then the Gibbs posterior concentrates at rate $\eps_n = n^{-1/2}$ with respect to $d(\theta,\theta^\star)=\|\theta - \theta^\star\|_2$.
\end{proposition}

\subsubsection{Heavy-tailed response}
\label{SSS:heavy}

As discussed in Section~\ref{ss:check1}, Condition~\ref{cond:loss} can be expected to hold when the response is light-tailed, but not when it is heavy-tailed.  However, for a capped loss, $\ell_\theta^n = \ell_\theta \wedge t_n$, with increasing $t_n$, and for a suitable sieve on the parameter space, we can show concentration of the Gibbs posterior at the risk minimizer $\theta^\star$ via the argument given in Section~\ref{ss:surrogate}.

\begin{asmp}
\label{asmp:heavy}
The marginal distribution of $Y$ satisfies $P|Y|^{s}<\infty$ for $s > 2$.
\end{asmp}
 
Define a sieve by $\Theta_n=\{\theta\in \mathbb{R}^J: \|\theta\|_2 < \Delta_n\}$ for an increasing sequence $\Delta_n$, e.g., $\Delta_n = \log n$.  The moment condition in Assumption~\ref{asmp:heavy} implies three important properties of $\ell_\theta^n$ for $\theta\in \Theta_n$:
\begin{itemize}
\item The excess clipped loss satisfies $\sup_{y,x}|\ell_\theta^n(y,x) - \ell_{\theta_n^\star}(y,x)| < \Delta_n t_n^{1/2}$. 
\item The risk and clipped risk are equivalent up to an error of order $\Delta_n^s t_n^{1-s/2}$.  
\item The clipped loss satisfies $m_n(\theta,\theta^\star_n) \vee v_n(\theta,\theta^\star_n) \lesssim \Delta_n^{2} (\|\theta-\theta^\star\|_2^2 + \Delta_n^s t_n^{1-s/2})$.
\end{itemize} 
These three properties can be used to verify Condition~\ref{cond:loss.c} using the argument sketched out in Section~\ref{ss:surrogate} and to verify the prior bound in \eqref{eq:prior.surrogate} required to apply Theorem~\ref{thm:rate.surrogate}.

\begin{proposition}
\label{prop:heavy}
Suppose Assumption~\ref{asmp:heavy} holds for some $s>2$.  Let $t_n = n^{2/(s-1)}$, $\Delta_n = \log n$, $\omega_n = \Delta_n^{-1}t_n^{-1/2}$, and $\eps_n = \Delta_n^{s/2} n^{-(s-2)/(2s-2)}$.  If Assumption~\ref{asmp:light}.2--3 also holds, then the Gibbs posterior with learning rate $\omega_n$ concentrates at rate $\eps_n$ with respect to $d(\theta,\theta^\star)=\|\theta - \theta^\star\|_2$. 
\end{proposition}

Proposition~\ref{prop:heavy} continues to hold if we replace $s$ in the definitions of $t_n$ and $\eps_n$ by $s'$ satisfying $2<s'<s$.  That means we only need an accurate lower bound for $s$ to construct a consistent Gibbs posterior for $\theta^\star$, albeit with a slower concentration rate. 

It is clear the clipping sequence may bias the clipped risk minimizer.  For a given clip, the bias is less for light-tailed (large $s$) losses compared to heavy-tailed losses because a loss in excess of the clip is rarer for the former.  This explains why the clipping sequence $t_n = n^{2/(s - 1)}$ {\em decreases} in $s$.  

Note that our $\eps_n^2$ can be compared to the rate derived in \citet[][Example~11]{grunwald.2018}.  Indeed, up to log factors, their rate, $n^{-s/(s+2)}$, is smaller than ours for $s < 4$ but larger for $s > 4$.  That is, their rate is slightly better when the response has between 2 and 4 moments, while ours is better when the response has 4 or more moments. Also, their result assumes that the parameter space is fixed and bounded, whereas we avoid this assumption with a suitably chosen sieve.

\subsection{Mean regression curve}
\label{SS:pred_sq}

Let $Y_1,\ldots,Y_n$ be independent, where the marginal distribution of $Y_i$ depends on a fixed covariate $x_i \in [0,1]$ through the mean, i.e., the expected value of $Y_i$ is $\theta^\star(x_i)$, $i=1,\ldots,n$.  For simplicity, set $x_i=i/n$, corresponding to an equally-spaced design.  Then the goal is estimation of the mean function $\theta^\star: [0,1] \to \RR$, which resides in a specified function class $\Theta$ defined below. 

A natural starting point is to define an empirical risk based on squared error loss, i.e.,
\begin{equation} R_n(\theta) = \frac1n \sum_{i=1}^n \{Y_i - \theta(x_i)\}^2. \end{equation}
However, any function $\theta$ that passes through the observations would be an empirical risk minimizer, so some additional structure is needed to make the solution to the empirical risk minimization problem meaningful.  Towards this, as is customary in the literature, we parametrize the mean function as a linear combination of a fixed set of basis functions, $f(x) = (f_1(x), \ldots, f_{J}(x))^\top$.  That is, we consider only functions $\theta=\theta_{\beta}$, where 
\begin{equation} \theta_{\beta}(x) = \beta^\top f(x) , \quad \beta \in \RR^J. \end{equation}
Note that we do not assume that $\theta^\star$ is of the specified form; more specifically, we do not assume existence of a vector $\beta^\star$ such that $\theta^\star = \theta_{\beta^\star}$.  The idea is that the structure imposed via the basis functions will force certain smoothness, etc., so that minimization of the risk over this restricted class of functions would identify a suitable estimate.  

This structure changes the focus of our investigation from the mean function $\theta$ to the $J$-vector of coefficients $\beta$.  We now proceed by first constructing a Gibbs posterior for $\beta$ and then obtain the corresponding Gibbs posterior for $\theta$ by pushing the former through the mapping $\beta \mapsto \theta_{\beta}$.  In particular, define the empirical risk function in terms of $\beta$:
\begin{equation}
\label{eq:r_n}
r_n(\beta) = R_n(\theta_{\beta}) = \frac1n \sum_{i=1}^n \{Y_i - \theta_{\beta}(x_i)\}^2 = \tfrac1n(Y - F_n\beta)^\top (Y - F_n\beta), \end{equation}
where $\beta \in \mathbb{R}^J$ and where $F_n$ is the $n \times J$ matrix whose $(i,j)$ entry is $f_j(x_i)$, assumed to be positive definite; see below.  Given a prior distribution $\Pitilde$ for $\beta$---which determines a prior $\Pi$ for $\theta$ through the aforementioned mapping---we can first construct the Gibbs posterior for $\beta$ as in \eqref{eq:gibbs} with the pseudo-likelihood $\beta \mapsto \exp\{-\omega n r_n(\beta) \}$.  If we write $\Pitilde_n$ for this Gibbs posterior for $\beta$, then the corresponding Gibbs posterior for $\theta$ is given by 
\begin{equation} \Pi_n(A) = \Pitilde_n(\{\beta: \theta_\beta \in A\}), \quad A \subseteq \Theta. \end{equation}
Therefore, the concentration properties of $\Pi_n$ are determined by those of $\Pitilde_n$.  

We can now proceed very much like we did before, but the details are slightly more complicated in the present inid case.  Taking expectation with respect to the joint distribution of $(Y_1,\ldots,Y_n)$ is, as usual, the average of marginal expectations; however, since the data are not iid, these marginal expectations are not all the same.  Therefore, the expected empirical risk function is
\begin{align}
\bar r_n(\beta) = P^n r_n(\beta) &= \frac1n \sum_{i=1}^n P_i \{Y_i - \theta_\beta(x_i)\}^2
\end{align}
where $P_i = P_{x_i}$ is the marginal distribution of $Y_i$ and where $F_{n,i}$ is the $i^{\text{th}}$ row of $F_n$.  Since the expected empirical risk function depends on $n$, through $(x_1,\ldots,x_n)$, so too does the risk minimizer
\begin{equation}\beta^\dagger_n = \arg\min_\beta \bar r_n(\beta).\end{equation}
If $P_i$ has finite variance, then $\bar r_n(\beta)$ differs from $\{\theta^\star(x_{1:n}) - F_n\beta\}^2$ by only an additive constant not depending on $\beta$, and this becomes a least-squares problem, with solution
\begin{equation} \beta_n^\dagger =(F_n^\top F_n)^{-1} F_n^\top \,\theta^\star(x_{1:n}), \end{equation}
where $\theta^\star(x_{1:n})$ is the $n$-vector $(\theta^\star(x_1),\ldots,\theta^\star(x_n))^\top$.  Our expectation is that the Gibbs posterior $\Pitilde_n$ for $\beta$ will suitably concentrate around $\beta_n^\dagger$, which implies that the Gibbs posterior $\Pi_n$ for $\theta$ will suitably concentrate around $\theta_{\beta_n^\dagger}$.  Finally, if the above holds and the basis representation is suitably flexible, then $\theta_{\beta_n^\dagger}$ will be close to $\theta^\star$ in some sense and, hence, we achieve the desired concentration.  

The flexibility of the basis representation depends on the dimension $J$.  Since $\theta^\star$ need not be of the form $\theta_\beta$, a good approximation will require that $J=J_n$ be increasing with $n$.  How fast $J=J_n$ must increase depends on the smoothness of $\theta^\star$.  Indeed, if $\theta^\star$ has smoothness index $\alpha > 0$ (made precise below), then many systems of basis functions---including Fourier series and B-splines---have the following approximation property: there exists an $H>0$ such that for every $J$
\begin{equation}
    \label{eq:approx.smooth}
    \text{there exists $\beta \in \RR^J$ such that $\|\beta\|_\infty < H$ and $\|\theta_\beta - \theta^\star\|_\infty \lesssim J^{-\alpha}$}.
\end{equation}
Then the idea is to set the approximation error in \eqref{eq:approx.smooth} equal to the target rate of convergence, which depends on $n$ and on $\alpha$, and then solve for $J=J_n$.  

For Gibbs posterior concentration at or near the optimal rate, we need the prior distribution for $\beta$ to be sufficiently concentrated in a bounded region of the $J$-dimensional space in the sense that 
\begin{equation}
\label{eq:prior.sqr}
\Pitilde(\{\beta:\|\beta-\beta'\|_2 \leq \eps\}) \gtrsim (C\eps)^J, \quad \text{for all $\beta' \in \RR^J$ with $\|\beta'\|_\infty \leq H$},  
\end{equation}
for the same $H$ as in \eqref{eq:approx.smooth}, for a small constant $C>0$, and for all small $\eps > 0$.  

\begin{asmp}
\label{asmp:pred_sq}
\mbox{}
\begin{enumerate}
\item The function $\theta^\star: [0,1] \to \RR$ belongs to a class $\Theta=\Theta(\alpha,L)$ of H\"older smooth functions parametrized by $\alpha \geq 1/2$ and $L > 0$.  That is, $\theta^\star$ satisfies  
\[ |\theta^{\star ([\alpha])}(x) - \theta^{\star ([\alpha])}(x')| \leq L |x - x'|^{\alpha - [\alpha]}, \quad \text{for all $x,x' \in [0,1]$}, \]
where the superscript ``$(k)$'' means $k^\text{th}$ derivative and $[\alpha]$ is the integer part of $\alpha$;  
\item for a given $x$, the response $Y$ is sub-Gaussian and with variance and variance proxy---both of which can depend on $x$---uniformly upper bounded by $\sigma^2$; 
\item the eigenvalues of $F_n^\top F_n$ are bounded away from zero and $\infty$;
\item the approximation property  \eqref{eq:approx.smooth} holds; and
\item the prior for $\beta$ satisfies \eqref{eq:prior.sqr} and has a bounded density on the $J_n$-dimensional parameter space.
\end{enumerate}
\end{asmp}

The bounded variance assumption is implied, for example, if the variance of $Y$ is a smooth function of $x$ in $[0,1]$, which is rather mild.  And assuming the eigenvalues of $F_n^\top F_n$ are bounded is not especially strong since, in many cases, the basis functions would be orthonormal.  In that case, the diagonal and off-diagonal entries of $F_n^\top F_n$ would be approximately 1 and 0, respectively, and the bounds are almost trivial.  The conditions on the prior distribution are weak; as we argue in the proof, it can be satisfied by taking the joint prior density to be the product of $J$ independent prior densities on the components of $\beta$, and where each component density is strictly positive.  

\begin{proposition}
\label{prop:pred_sq}
If Assumption~\ref{prop:pred_sq} holds, and the learning rate $\omega$ is a sufficiently small constant, then the Gibbs posterior $\Pi_n$ for $\theta$ concentrates at $\theta^\star$ with rate $\eps_n = n^{-\alpha/(1+2\alpha)}$ with respect to the empirical $L_2$ norm $\|\theta-\theta^\star\|_{n,2}$, where $\|f\|_{n,2}^2 =  n^{-1}\sum_{i=1}^n f^2(x_i)$.
\end{proposition}

We should emphasize that the quantity of interest, $\theta$, is high-dimensional, and the rate $\eps_n$ given in Proposition~\ref{prop:pred_sq} is optimal for the given smoothness $\alpha$; there are not even any nuisance logarithmic terms.  

The simpler fixed-dimensional setting with constant $J$ can be analyzed similarly as above.  In that case, we can simultaneously weaken the requirement on the response $Y$ in Assumption~\ref{asmp:pred_sq}.2 from sub-Gaussian to sub-exponential, and strengthen the conclusion to a root-$n$ concentration rate.

\subsection{Binary classification}
\label{SS:classification}

Let $Y \in \{0,1\}$ be a binary response variable and $X = (X_0, X_1,\ldots,X_q)^\top$ a $(q+1)$-dimensional predictor variable.  We consider classification rules of the form 
\begin{equation} \phi_\theta(X) = 1\{X^\top \theta > 0\} = 1\{\alpha X_0 + (X_1,\ldots,X_q)^\top \beta > 0\}, \quad \theta = (\alpha, \beta) \in \RR^{q+1}, \end{equation}
and the goal is to learn the optimal $\theta$ vector, i.e., $\theta^\star = \arg \min_\theta R(\theta)$, where $R(\theta) = P\{Y \neq \phi_\theta(X)\}$ is the misclassification error probability, and $P$ is the joint distribution of $(X,Y)$.  This optimal $\theta^\star$ is such that $\eta(x) > \frac12$ if $x^\top \theta^\star > 0$ and $\eta(x) < \frac12$ if $x^\top \theta^\star \leq 0$, where $\eta(x) = P(Y=1 \mid X=x)$ is the conditional probability function.  Below we construct a Gibbs posterior distribution that concentrates around this optimal $\theta^\star$ at rate that depends on certain local features of that $\eta$ function.  

Suppose our data consists of iid copies $(X_i,Y_i)$, $i=1,\ldots,n$, of $(X,Y)$ from $P^\star$, and define the empirical risk function 
\begin{equation} R_n(\theta) = \frac{1}{n} \sum_{i=1}^n 1\{Y_i \neq \phi_\theta(X_i)\}. \end{equation}
In addition to the empirical risk function we need specify a prior $\Pi$ and here the prior plays a significant role in the Gibbs posterior concentration results.  

A unique feature of this problem, which makes the prior specification a little different than in a linear regression problem, is that the scale of $\theta$ does not affect classification performance, e.g., replacing $\theta$ with $1000\theta$ gives exactly the same classification performance.  To fix a scale, we follow \citet{jiang.tanner.2008}, and 
\begin{itemize}
\item assume that the $x_0$ component of $x$ is of known importance and always included in the classifier, 
\item and constrain the corresponding coefficient, $\alpha$, to take values $\pm 1$.
\end{itemize}
This implies that the $\alpha$ and $\beta$ components of the $\theta$ vector should be handled very differently in terms of prior specification.  In particular, $\alpha$ is a scalar with a discrete prior---which we take here to be uniform on $\pm 1$---and $\beta$, being potentially high-dimensional, will require setting-specific prior considerations.   

The characteristic that determines the difficulty of a classification problem is the distribution of $\eta(X)$ or, more specifically, how concentrated $\eta(X)$ is near the value $\frac12$, where one could do virtually no worse by classifying according to a coin flip.  The set $\{x: \eta(x) = \frac12\}$ is called the {\em margin}, and conditions that control the concentration of the distribution of $\eta(X)$ around $\frac12$ are generally called {\em margin conditions}.  Roughly, if $\eta$ has a ``jump'' or discontinuity at the margin, then classification is easier and $\eta(X)$ does not have to be so tightly concentrated around $\frac12$.  On the other hand, if $\eta$ is smooth at the margin, then the classification problem is more challenging in the sense that more data near the margin is needed to learn the optimal classifier, hence, tighter concentration of $\eta(X)$ near $\frac12$ is required. 

In Sections~\ref{SSS:massart} and \ref{SSS:tsybakov} that follow, we consider two such margin conditions, namely, the so-called Massart and Tsybakov conditions.  The first is relatively strong, corresponding to a jump in $\eta$ at the margin, and the result we we establish in Proposition~\ref{prop:massart} is accordingly strong.  In particular, we show that the Gibbs posterior achieves the optimal and adaptive concentration rate in a class of high-dimensional problems ($q \gg n$) under a certain sparsity assumption on $\theta^\star$.  The Tsybakov margin condition we consider below is weaker than the first, in the sense that $\eta$ can be smooth near the ``$\eta=\frac12$'' boundary and, as expected, the Gibbs posterior concentration rate result is not as strong as the first.  

\subsubsection{Massart's noise condition}
\label{SSS:massart}

Here we allow the dimension $q+1$ of the coefficient vector $\theta = (\alpha, \beta)$ to exceed the sample size $n$, i.e., we consider the so-called {\em high-dimensional} problem with $q \gg n$.  Accurate estimation and inference is not possible in high-dimensional settings without imposing some low-dimensional structure on the inferential target, $\theta^\star$.  Here, as is typical in the literature on high-dimensional inference, we assume that $\theta^\star$ is {\em sparse} in the sense that most of its entries are exactly zero, which corresponds to most of the predictor variables being irrelevant to classification.  Below we construct a Gibbs posterior distribution for $\theta$ that concentrates around the unknown sparse $\theta^\star$ at a (near) optimal rate.  

Since the sparsity in $\theta^\star$ is crucial to the success of any statistical method, the prior needs to be chosen carefully so that sparsity is encouraged in the posterior.  The prior $\Pi$ for $\theta$ will treat $\alpha$ and $\beta$ independent, and the prior for $\beta$ will be defined hierarchically.  Start with the reparametrization $\beta \to (S,\beta_S)$, where $S \subseteq \{1,2,\ldots,q\}$ denotes the configuration of zeros and non-zeros in the $\beta$ vector, and $\beta_S$ denotes the $|S|$-vector of non-zero values.  Following \citet{castillo.etal.2015}, for the marginal prior $\pi(S)$ for $S$, we take  
\[ \pi(S) = \textstyle \binom{q}{|S|}^{-1} \, f(|S|), \]
where the $f$ is a prior for the size $|S|$ and the first factor on the right-hand side is the uniform prior for $S$ of the given size $|S|$.  Various choices of $f$ are possible, but here we take the {\em complexity prior} $f(s) \propto (cq^a)^{-s}, \,s=0,1,\ldots,q$, a truncated geometric density, where $a$ and $c$ are fixed (and here arbitrary) hyperparameters; a similar choice is also made in \citet{martin.etal.2017}.  Second, for the conditional prior of $\beta_S$, given $S$, again following \citet{castillo.etal.2015}, we take its density to be 
\[ g_S(\beta_S) = \prod_{k \in S} \tfrac{\lambda}{2} e^{-\lambda |\beta_k|}, \]
a product of $|S|$ many Laplace densities with rate $\lambda$ to be specified.  

\begin{asmp}
\label{asmp:massart}
\mbox{}
\begin{enumerate}
\item The marginal distribution of $X$ is compactly supported, say, on $[-1,1]^{q+1}$.  
\item The conditional distribution of $X_0$, given $\tilde X =(X_1, \ldots, X_q)$, has a density with respect to Lebesgue measure that is uniformly bounded.  
\item The rate parameter $\lambda$ in the Laplace prior satisfies $\lambda \lesssim (\log q)^{1/2}$.
\item The optimal $\theta^\star = (\alpha^\star, \beta^\star)$ is sparse in the sense that $|S^\star| \log q = o(n)$, where $S^\star$ is the configuration of non-zero entries in $\beta^\star$, and $\|\beta^\star\|_\infty = O(1)$. 
\item There exists $h \in (0,\frac12)$ such that $P(|\eta(X) - \frac12| \leq h) = 0$. 
\end{enumerate}
\end{asmp}

The first two parts of Assumption~\ref{asmp:massart} correspond to Conditions $0'$ and $0''$ in \citet{jiang.tanner.2008}.    and Assumption~\ref{asmp:massart} (5) above is precisely the margin condition imposed in Equation~(5) of \citet{massart.nedelec.2006}; see, also, \citet{mammen.tsybakov.1999} and \citet{koltchinskii.2006}.  This concisely states that there is either a jump in the $\eta$ function at the margin or that the marginal distribution of $X$ is not supported near the margin; in either case, there is separation between the $Y=0$ and $Y=1$ cases, which makes the classification problem relatively easy. 
 
With these assumptions, we get the following Gibbs posterior asymptotic concentration rate result.  Note that, in order to preserve the high-dimensionality in the asymptotic limit, we let the dimension $q=q_n$ increase with the sample size $n$.  So, the data sequence actually forms a triangular array but, as is common in the literature, we suppress this formulation in our notation.  

\begin{proposition}
\label{prop:massart}
Consider a classification problem as described above, with $q \gg n$.  Under Assumption~\ref{asmp:massart}, the Gibbs posterior, with sufficiently small constant learning rate, concentrates at rate $\eps_n = (n^{-1} |S^\star| \log q)^{1/2}$ with respect to $d(\theta,\theta^\star) = \{R(\theta) - R(\theta^\star)\}^{1/2}$.  
\end{proposition}

This result shows that, even in very high dimensional settings, the Gibbs posterior concentrates on the optimal rule $\theta^\star$ at a fast rate.  For example, suppose that the dimension $q$ is polynomial in $n$, i.e., $q \sim n^b$ for any $b > 0$, while the ``effective dimension,'' or complexity, is sub-linear, i.e., $|S^\star| \sim n^a$ for $a < 1$.  Then we get that $\{\theta: R(\theta) - R(\theta^\star) \lesssim n^{-(1-a)} \log n\}$ has Gibbs posterior probability converging to 1 as $n \to \infty$.  That is, rates better than $n^{-1/2}$ can easily be achieved, and even arbitrarily close to $n^{-1}$ is possible.  Compare this to the rates in Propositions~2--3 in \citet{jiang.tanner.2008}, also in terms of risk difference, that cannot be faster than $n^{-1/2}$.  Further, the concentration rate in Proposition~\ref{prop:massart} is nearly the optimal rate corresponding to an oracle who has knowledge of $S^\star$.  That is, the Gibbs posterior concentrates at nearly the optimal rate {\em adaptively} with respect to the unknown complexity.

\subsubsection{Tsybakov's margin condition}
\label{SSS:tsybakov}

Next, we consider classification under the more general Tsybakov margin condition \citep[e.g.,][]{tsybakov}.  The problem set up is the same as above, except that here we consider the simpler low-dimensional case, with the number of predictors $(q+1)$ small relative to $n$.  Since the dimension is no longer large, prior specification is much simpler.  We will continue to assume, as before, that the $x_0$ component of $x$ has a constrained coefficient $\alpha \in \{\pm 1\}$, to which we assign a discrete uniform prior.  Otherwise, we simply require the prior  $\Pi$ have a (marginal) density, $\pi$, for $\beta$, with respect to Lebesgue measure on $\RR^q$.  

\begin{asmp}
\label{asmp:tsybakov}
\mbox{}
\begin{enumerate}
\item The marginal prior density for $\beta$ is continuous and bounded away from 0 near $\beta^\star$. 
\item There exists $c>0$ and $\gamma>0$ such that $P(|2\eta(X) - 1| \leq h) \leq ch^\gamma$ for all sufficiently small $h > 0$. 
\end{enumerate}
\end{asmp}

The concentration of the marginal distribution of $\eta(X)$ around $\frac12$ controls the difficulty of the classification problem, and Condition~\ref{asmp:tsybakov}.2 concerns exactly this.  Note that smaller $\gamma$ implies $\eta(X)$ is less concentrated around $\frac12$, so we expect our Gibbs posterior concentration rate, say, $\eps_n=\eps_n(\gamma)$, to be a decreasing function of $\gamma$.  The following result confirms this. 

\begin{proposition}
\label{prop:tsybakov}
Suppose Assumption~\ref{asmp:tsybakov} holds and, for the specified $\gamma > 0$, let 
\[ \eps_n = (\log n)^{\gamma/(2+2\gamma)}n^{-\gamma/(3+2\gamma)}. \]
Then the Gibbs posterior distribution, with learning rate $\omega_n=\eps_n^{1/\gamma}$, concentrates at rate $\eps_n$ with respect to $d(\theta,\theta^\star) = \{R(\theta) - R(\theta^\star)\}^{1/2}$.  
\end{proposition}

Note that Massart's condition from Section~\ref{SSS:massart} corresponds to Tsybakov's condition above with $\gamma = \infty$.  In that case, the Gibbs posterior concentration rate we recover from Proposition~\ref{prop:tsybakov} is $\eps_n = (\log n)^{1/2} n^{-1/2}$, which is achieved with a suitable constant learning rate.  This is within a logarithmic factor of the optimal rate for finite-dimensional problems. Moreover, for both the $\gamma < \infty$ and $\gamma=\infty$ cases, we expect that the logarithmic factor could be removed following an approach like that described in Theorem~\ref{thm:rate.rootn}, but we do not explore this possible extension here.

We should emphasize that this case is unusual because the learning rate $\omega_n$ depends on the (likely unknown) smoothness exponent $\gamma$.  This means the rate in Proposition~\ref{prop:tsybakov} is not adaptive to the margin.  However, this dependence is not surprising, as it also appears in \citet[][Section~6]{grunwald.2018}.  The reason the learning rate depends on $\gamma$ is that the Tsybakov margin condition in Assumption~\ref{asmp:tsybakov}.2 implies the Bernstein condition in \eqref{eq:bernstein} takes the form
\[v(\theta,\theta^\star)\lesssim m(\theta,\theta^\star)^{\gamma/(1+\gamma)}. \]
Therefore, in order to verify Condition~\ref{cond:loss} using the strategy in Section~\ref{ss:check1}, we need  $\omega_n m(\theta,\theta^\star)$ and $\omega_n^3 v(\theta,\theta^\star)$ to have the same order when $d(\theta,\theta^\star)\geq M_n\eps_n$.  This requires that the learning rate depends on $\gamma$, in particular, $\omega_n = \eps_n^{1/\gamma}$. 

\subsection{Quantile regression curve}
\label{SS:quantile.reg}

In this section we revisit inference on a conditional quantile, covered in Section~\ref{SS:quantile_reg}.  The $\tau^{\text{th}}$ conditional quantile of a response $Y$ given a covariate $X = x$ is modeled by a linear combination of basis functions $f(x) = (f_1(x), ..., f_J(x))^\top$:
\[ Q_{Y|X=x}(\tau) = \beta^\top f(x), \quad \beta \in \RR^J. \]

In Section~\ref{SS:quantile_reg}, we made the rather restrictive assumption that the true conditional quantile function $\theta^\star(x)$ belonged to the span of a fixed set of $J$ basis functions.  In practice, it may not be possible to identify such a set of functions, which is why we considered using a sample-size dependent sequence of sets of basis functions in Section~\ref{SS:pred_sq} to model a smooth function, $\theta^\star$.  When the degree of smoothness, $\alpha$, of $\theta^\star$ is known we can choose the number of basis functions to use in order to achieve the optimal concentration rate.  But, in practice, $\alpha$ may not be known, which creates a challenge because, as mentioned, the number of terms needed in the basis function expansion modeling $\theta^\star$ depends on this unknown degree of smoothness.  

To achieve optimal concentration rates adaptive to unknown smoothness, the choice of prior is crucial.  In particular, the prior must support a very large model space in order to guarantee it places sufficient mass near $\theta^\star$. Our approximation of $\theta^\star$ by a linear combination of basis functions suggests a hierarchical prior for $\theta\equiv(J,\beta_J)$, similar to Section~\ref{SSS:massart}, with a marginal prior $\pi$ for the number of basis functions $J$ and a conditional prior $\widetilde \Pi_J$ for $\beta_J$, given $J$.  The resulting prior for $\theta$ is given by a mixture,
\begin{equation}
    \label{eq:prior_bar}
\Pi(A) = \sum_{j=1}^\infty \pi(j) \, \Pitilde_j(\{\beta_j\in \RR^j: \beta_j^\top f \in A\}), \quad A\subseteq \Theta.
\end{equation}
Then, in order for $\Pi$ to place sufficient mass near $\theta^\star$, it is sufficient the marginal and conditional priors satisfy the following conditions: the marginal prior $\pi$ for $J$ satisfies for some $c_1>0$ for every $J=j$
\begin{equation}
\label{eq:prior.j}
    \pi(j)\geq e^{-c_1j\log j};
\end{equation}
and, the conditional prior $\widetilde \Pi$ for $\beta_J$ given $J=j$ satisfies  for every $j$
\begin{equation}
\label{eq:prior.beta}
\Pitilde(\{\beta:\|\beta-\beta'\|_2 \leq \eps\}) \gtrsim e^{-Cj\log(1/\eps)}, \quad \text{for all $\beta' \in \RR^j$ with $\|\beta'\|_\infty \leq H$},  
\end{equation}
for the same $H$ as in \eqref{eq:approx.smooth} and for some constant $C>0$ for all sufficiently small $\eps > 0$.  Fortunately, many simple choices of $(\pi, \widetilde\Pi_J)$ are satisfactory for obtaining adaptive concentration rates, e.g., a Poisson prior on $J$ and a $J$-dimensional normal conditional prior for $\beta$, given $J$; see Conditions (A1) and (A2) and Remark~1 in \citet{shen.ghosal.2015}.  Besides the conditions in \eqref{eq:prior.j} and \eqref{eq:prior.beta} we need to make a minor modification of $\Pi$ to make it suitable for our proof of Gibbs posterior concentration; see below.

Similar to our choice in Section~\ref{SS:quantile_reg} we link the data and parameter through the check loss function 
\[ \ell_\theta(u) = \tfrac12 (|\theta(x)-y|-|y|)+(1-\tau)\theta(x),\] 
where $\theta(x) = \beta^\top f(x)$.  See \citet{koltchinskii.1997} for a proof that $P\ell_\theta$ is minimized at $\theta^\star$.  It is straightforward to show the check loss $\theta\mapsto\ell_\theta(u)$ is $L$-Lipschitz with $L < 1$.  From there, if $Y$ were bounded we could use Condition~\ref{cond:loss} to compute the concentration rate.  However, to handle an unbounded response we need the flexibility of Condition~\ref{cond:loss.b} and Theorem~\ref{thm:rate.b}.  To verify \eqref{eq:post.b}, we found it necessary to limit the complexity of the parameter space by imposing a constraint on the prior distribution, namely that the sequence of prior distributions places all its mass on the set $\Theta_n:=\{\theta:\|\theta\|_\infty \leq \Delta_n\}$ for some diverging sequence $\Delta_n$; see Assumption~\ref{prop:quant_reg}.4.  This constraint implies the prior depends on $n$, and we refer to this sequence of prior distributions by $\Pi^{(n)}$.  Given the hierarchical prior $\Pi$ in \eqref{eq:prior_bar} one straightforward way to define a sequence of prior distributions satisfying the constraint is to restrict and renormalize $\Pi$ to $\Theta_n$, i.e., define $\Pi^{(n)}$ as
\begin{align}
    \label{eq:prior_Delta}
\Pi^{(n)}(A) = \Pi(A)/\Pi(\Theta_n), \quad A\subseteq \Theta\cap\Theta_n 
\end{align}
This particular construction of $\Pi^{(n)}$ in \eqref{eq:prior_Delta} is not the only way to define a sequence of priors satisfying the restriction to $\Theta_n$, but it is convenient.  That is, if $\Pi$ places mass $\eta$ on a sup-norm neighborhood of $\theta^\star$ (see the proof of Proposition~\ref{prop:quant_reg}), then, by construction, $\Pi^{(n)}$ in \eqref{eq:prior_Delta} places at least mass $\eta$ on the same neighborhood.  

We should emphasize this restriction of the prior to $\Theta_n$ is only a technical requirement needed for the proof, but it is not unreasonable.  Since the true function $\theta^\star$ is bounded, it is eventually in the growing support of the prior $\Pi$.  Similar assumptions have been used in the literature on quantile curve regression; for example, Theorem~6 in \citet{takeuchi.2006} requires that the parameter space consists only of bounded functions, which is a stricter assumption than ours here.

\begin{asmp}
\label{asmp:quant_reg}
\mbox{}
\begin{enumerate}
\item The function $\theta^\star:\mathbb{X}\mapsto\mathbb{R}$ is H\"older smooth with parameters $(\alpha,L)$ (see Assumption~\ref{prop:pred_sq}.1);
\item the basis functions satisfy the approximation property in \eqref{eq:approx.smooth};
\item the covariate space $\mathbb{X}$ is compact and there exists a $\delta>0$ such that the conditional density of $Y$, given $X=x$, is continuous and bounded away from $0$ by a constant $\beta>0$ in the interval $(\theta^\star(x)-\delta,\, \theta^\star(x)+\delta)$ for every $x$; and,
\item the sequence $\Pi^{(n)}$ of prior distributions satisfies \eqref{eq:prior_Delta} for a sequence of subsets of the parameter space $\Theta_n:=\{\theta:\|\theta\|_\infty \leq \Delta_n\}$ for some sequence $\Delta_n>0$, for $\Pi$ as defined in \eqref{eq:prior_bar}, and for marginal and conditional priors $(\pi, \widetilde \Pi)$ for $J$ and $\beta_J$ given $J=j$ that satisfy \eqref{eq:prior.j} and \eqref{eq:prior.beta}.
\end{enumerate}
\end{asmp}

\begin{proposition}
\label{prop:quant_reg}
Define $\eps_n = (\log n)^{1/2}\Delta_n^2 n^{-\alpha/(1+2\alpha)}$.  If the learning rate satisfies $\omega_n = c\Delta_n^{-2}$ for some $0<c<1/2$ and Assumption~\ref{asmp:quant_reg} holds, then the Gibbs posterior distribution concentrates at rate $\eps_n$ with respect to $d(\theta,\theta^\star) = \|\theta-\theta^\star\|_{L_2(P)}$.  
\end{proposition}

Since the mathematical statement does not give sufficient emphasis to the adaptation feature, we should follow up on this point.  That is, $n^{-\alpha/(2\alpha+1)}$ is the optimal rate \citet{shen.ghosal.2015} for estimating an $\alpha$-H\"older smooth function, and it is not difficult to construct an estimator that achieves this, at least approximately, if $\alpha$ is known.  However, $\alpha$ is unknown in virtually all practical situations, so it is desirable for an estimator, Gibbs posterior, etc.~to adapt to the unknown $\alpha$.  The concentration rate result in Proposition~\ref{prop:quant_reg} says that the Gibbs posterior achieves nearly the optimal rate adaptively in the sense that it concentrates at nearly the optimal rate as if $\alpha$ were known.

The concentration rate in Proposition~\ref{prop:quant_reg} depends on the complexity of the parameter space as determined by $\Delta_n$ in Assumption~\ref{prop:quant_reg}(3).  For example, if the sup-norm bound on $\theta^\star$ were known, then $\Delta_n$ and the learning rate $\omega_n$ could be taken as constants and the rate would be optimal up to a $(\log n)^{1/2}$ factor.  On the other hand, if greater complexity is allowed, e.g., $\Delta_n = (\log n)^p$ for some power $p>0$, then the concentration rate takes on an additional $(\log n)^{2p}$ factor, which is not a serious concern.  

\section{Application: personalized MCID}
\label{S:mcid}

\subsection{Problem setup}

In the medical sciences, physicians who investigate the efficacy of new treatments are challenged to determine both {\em statistically} and {\em practically} significant effects.  In many applications some quantitative effectiveness score can be used for assessing the statistical significance of the treatment, but physicians are increasingly interested also in patients' qualitative assessments of whether they believed the treatment was effective.  The aim of the approach described below is to find the cutoff on the effectiveness score scale that best separates patients by their reported outcomes.  That cutoff value is called the {\em minimum clinically important difference}, or MCID.  
For this application we follow up on the MCID problem discussed in \citet{syring.martin.mcid} with a covariate-adjusted, or \emph{personalized}, version.  In medicine, there is a trend away from the classical ``one size fits all'' treatment procedures, to treatments that are tailored more-or-less to each individual.  Along these lines, naturally, doctors would be interested to understand how that threshold for practical significance depends on the individual, hence there is interest in a so-called {\em personalized MCID} \citep{hedayat,zhao.2020}.  

Let the data $U^n = (U_1,\ldots,U_n)$ be iid $P$, where each observation is a triple $U_i=(X_i, Y_i, Z_i)$ denoting the patient's diagnostic measurement, their self-reported effectiveness outcome $Y_i \in \{-1,1\}$, and covariate value $Z_i \in \ZZ \subseteq \RR^q$, for $i=1,\ldots,n$ and $q \geq 1$.  In practice the diagnostic measurement has something to do with the effectiveness of the treatment so one can imagine examples including blood pressure, blood glucose level, and viral load.  Examples of covariates include a patient's age, weight, and gender.  The idea is that the $x$-scale cutoff for practical significance would depend on the covariate $z$, hence the MCID is a function, say, $\theta(z)$, and the goal is to learn this function.  

The true MCID $\theta^\star$ is defined as the solution to an optimization problem.  That is, if 
\begin{equation} \ell_\theta(x,y,z) = \tfrac12 [1 - y \, \sign\{x - \theta(z)\}], \quad (x,y,z) \in \RR \times \{-1,1\} \times \ZZ, \end{equation}
then the expected loss is $R(\theta) = P[Y \neq \sign\{X - \theta(Z)\}]$, and the true MCID function is defined as the minimizer $\theta^\star = \arg\min_{\theta \in \Theta} R(\theta)$, where the minimum is taken over a class $\Theta$ of functions on $\ZZ$.  Alternatively, as in Section~\ref{SS:classification}, the true $\theta^\star$ satisfies $\eta_z(x)>\tfrac12$ if $x>\theta^\star(z)$ and $\eta_z(x)\leq \tfrac12$ if $x\leq \theta^\star(z)$, where $\eta_z(x) = P(Y=1 \mid X=x, \,Z=z)$ is the conditional probability function.

As described in Section~\ref{S:gibbs}, the Gibbs posterior distribution is based on an empirical risk function which, in the present case, is given by 
\begin{equation}
\label{eq:emp_Rn_mcid}R_n(\theta) = \frac{1}{2n} \sum_{i=1}^n [ 1 - Y_i \, \sign\{X_i - \theta(Z_i)\}], \quad \theta \in \Theta. \end{equation}
In order to put this theory into practice, it is necessary to give the function space $\Theta$ a lower-dimensional parametrization.  In particular, we consider a true MCID function $\theta^\star$ belonging to a H\"older class as in Assumption~\ref{prop:pred_sq} but with unknown smoothness, as in Section~\ref{SS:quantile.reg}.  And, we model $\theta^\star$ by a linear combination of basis functions $\theta(z) = \theta_{J,\beta}(z) := \textstyle\sum_{j=1}^J \beta_j f_j(z)$, for basis functions $f_j$, $j=1,\ldots,J$.  Then, each $\theta$ is identified by a pair $(J,\beta)$ consisting of a positive integer $J$ and a $J$-dimensional coefficient vector $\beta$.  We use cubic B-splines in the numerical examples in Section~2 of the supplmentary material, but any basis capable of approximating $\theta^\star$ will work, and see \eqref{eq:approx.smooth}.  

The prior setup is similar to that in Section~\ref{SS:quantile.reg}, \eqref{eq:prior_bar}.  That is, the prior is specified hierarchically with a marginal prior $\pi$ for $J$ and a suitable conditional prior $\widetilde\Pi_J$ for $\beta_J$, given $J$.  And, as mentioned before, very simple choices of the marginal and conditional priors achieve the desired adaptive rates.   

\subsection{Concentration rate result}

Assumption~\ref{prop:pers_mcid} below concerns the smoothness of $\theta^\star$ and requires the chosen basis satisfy the approximation property used previously; it also refers to the same mild assumptions on random series priors used in Section~\ref{SS:quantile.reg} sufficient to ensure adequate prior mass is assigned to a neighborhood of $\theta^\star$; finally, it assumes a margin condition on the classifier like that used in Section~\ref{SSS:massart} and Assumption~\ref{asmp:massart}(5).  These conditions are sufficient to establish a Gibbs posterior concentration rate.

\begin{asmp}
\label{assump:bdd_density}
\mbox{}
\begin{enumerate}
\item The true MCID function $\theta^\star:\ZZ \to \RR$ for a compact subset $\ZZ$ of $\RR$ and $\theta^\star$ is H\"older smooth with parameters $(\alpha, L)$ (see Assumption~\ref{prop:pred_sq}.1);  
\item the basis functions satisfy the approximation property in \eqref{eq:approx.smooth};
\item the prior distribution $\Pi$ for $\theta$ is defined hierarchically as in \eqref{eq:prior_bar} with marginal and conditional priors $(\pi, \widetilde \Pi)$ for $J$ and $\beta_J$ given $J=j$ that satisfy \eqref{eq:prior.j} and \eqref{eq:prior.beta}; and,
\item there exists $h \in (0,1)$ such that $P\{|2\eta_Z(X) - 1| \leq h\} = 0$; and,
\item the conditional distribution, $P_z$, of $X$, given $Z=z$, has a density with respect to Lebesgue measure that is uniformly bounded away from infinity.
\end{enumerate}
\end{asmp}

\begin{proposition}
\label{prop:pers_mcid}
Suppose Assumption~\ref{prop:pers_mcid} holds, with $\alpha$ as defined there, and set $\eps_n = (\log n) n^{-\alpha/(1+\alpha)}$.  For any fixed $\omega>0$ the Gibbs posterior concentrates at rate $\eps_n$ with respect to the divergence 
\begin{align}
d(\theta,\theta^\star) & = P\{\theta(Z) \wedge \theta^\star(Z) \leq X \leq \theta(Z) \vee \theta^\star(Z)\} \nonumber \\ 
& = \int_\ZZ \int_{\theta(z) \wedge \theta^\star(z)}^{\theta(z) \vee \theta^\star(z)} P_z(dx) \, P(dz). 
\end{align}
\end{proposition}

The Gibbs posterior distribution we have defined for the personalized MCID function achieves the concentration rate in Proposition~\ref{prop:pers_mcid} adaptively to the unknown smoothness $\alpha$ of $\theta^\star$.  \citet{mammen} consider estimation of the boundary curve of a set, and they show that the minimax optimal rate is $n^{-\alpha/(\alpha+1)}$ when the boundary curve is $\alpha$-H\"older smooth and distance is measured by the Lebesgue measure of the set symmetric difference.  In our case, if $(X,Z)$ has a joint density, bounded away from 0, then our divergence measure $d(\theta,\theta^\star)$ is equivalent to 
\[ \text{Leb}(\{(x,z): x \leq \theta(z)\} \, \triangle \, \{(x,z): x \leq \theta^\star(z)\}), \]
in which case our rate is within a logarithmic factor of the minimax optimal rate.

\citet{hedayat} also study the personalized MCID and derive a convergence rate for an M-estimator of $\theta^\star$ based on a smoothed and penalized version of \eqref{eq:emp_Rn_mcid}.  It is difficult to compare our result with theirs, for instance, because their rate depends on two user-controlled sequences related to the smoothing and penalization of their loss. But, as mentioned above, our rate is near optimal in certain cases, so the asymptotic results in \citet{hedayat} cannot be appreciably better than our rate in Proposition~\ref{prop:pers_mcid}.

\subsection{Numerical illustrations}
\label{SSS:mcid_examples}
We performed two simulation examples to investigate the performance of the Gibbs posterior for the personalized MCID.  In both examples we use a constant learning rate $\omega = 1$, but we generally recommend data-driven learning rates; and see \citet{syring.martin.scaling}.  

For the first example we sample $n=100$ independent observations of $(X,Y,Z)$.  The covariate $Z$ is sampled from a uniform distribution on the interval $[0,3]$.  Given $Z=z$, the diagnostic measure $X$ is sampled from a normal distribution with mean $z^3-3z^2+5$ and variance $1$, and the patient-reported outcome $Y$ is sampled from a Rademacher distribution with probability 
\begin{equation}
\eta_z(x) = \begin{cases}
\Phi(x;z^3-3z^2+5-0.05,1/2), & x>z^3-3z^2+5 \\
\Phi(x;z^3-3z^2+5+0.05,1/2), & x \leq z^3-3z^2+5,
\end{cases} \end{equation}
where $\Phi(x;\mu, \sigma)$ denotes the $\nm(\mu, \sigma)$ distribution function.  The addition of $\pm 0.05$ in the formula of $\eta_z(x)$ is to meet the margin condition in Assumption~5.4 in the main article.  As mentioned above, we parametrize the MCID function by piecewise polynomials, specifically, cubic B-splines.  For highly varying MCID functions, a reversible-jump MCMC algorithm that allows for changing numbers of and break points in the piecewise polynomials may be helpful; see \citet{syring.martin.image}.  However, for this example we fix the parameter dimension to just six B-spline functions, which allows us to use a simple Metropolis--Hastings algorithm to sample from the Gibbs posterior distribution.  Since the dimension is fixed, the prior is only needed for the B-spline coefficients, and for these we use diffuse independent normal priors with mean zero and standard deviation of $6$.  Over $250$ replications, the average empirical misclassification rate is 16\% using the Gibbs posterior mean MCID function compared to 13\% using the true MCID function when applying these two classifiers to a hold-out sample of $100$ data points.  

The left pane of Figure~\ref{fig:compare} shows the results for one simulated data set under the above formulation.  Even with only $n=100$ samples, the Gibbs posterior does a good job of centering on the true MCID function.  The right pane displays the pointwise Gibbs posterior mean MCID function for each of $250$ repetitions of simulation~1, along with the overall pointwise mean of these functions, and the true MCID function.  The Gibbs posterior mean function is stable across repetitions of the simulation.  
\begin{figure}[t]
\centering
\begin{subfigure}{0.45\textwidth}
\centering
\includegraphics[width=1.0\linewidth]{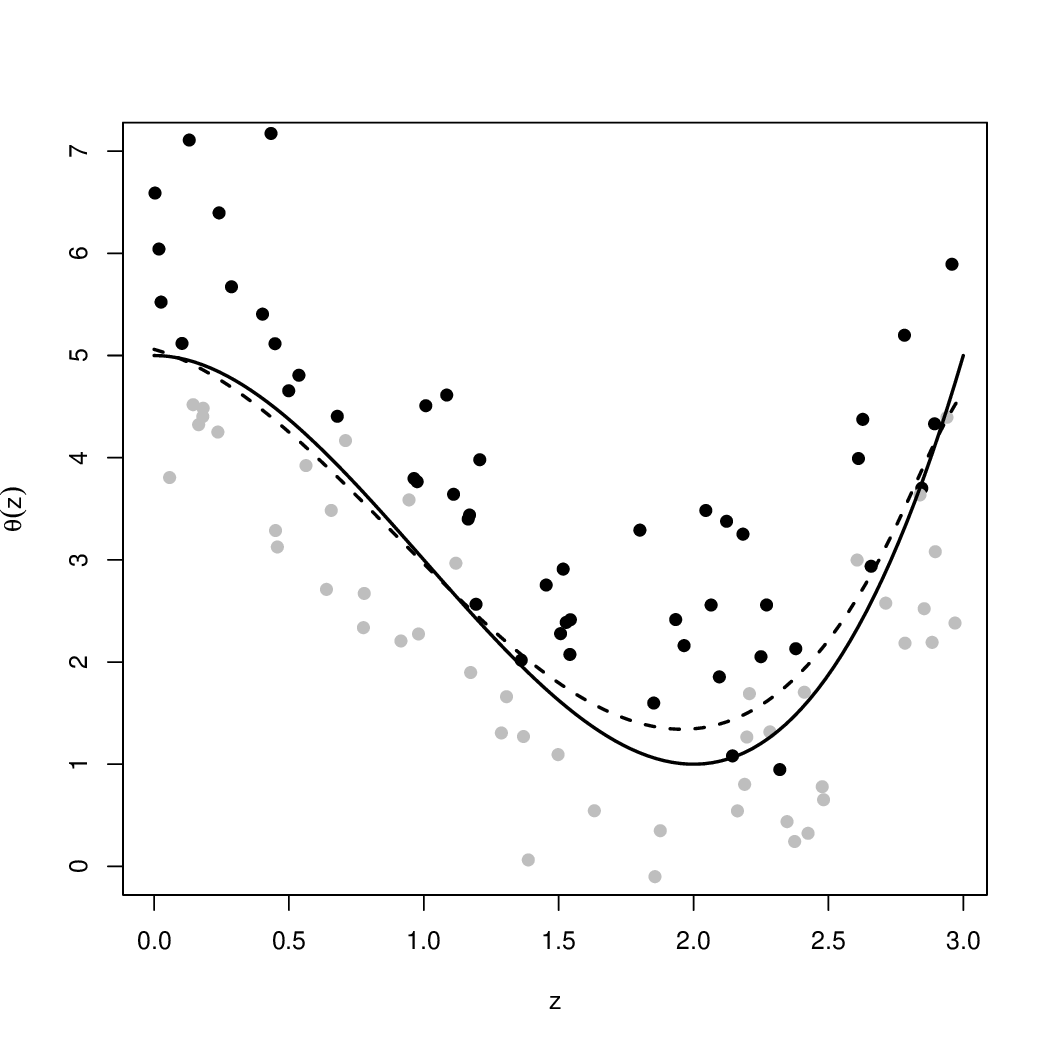}
\end{subfigure}\hspace{.5cm}
\begin{subfigure}{0.45\textwidth}
\centering
\includegraphics[width=1.0\linewidth]{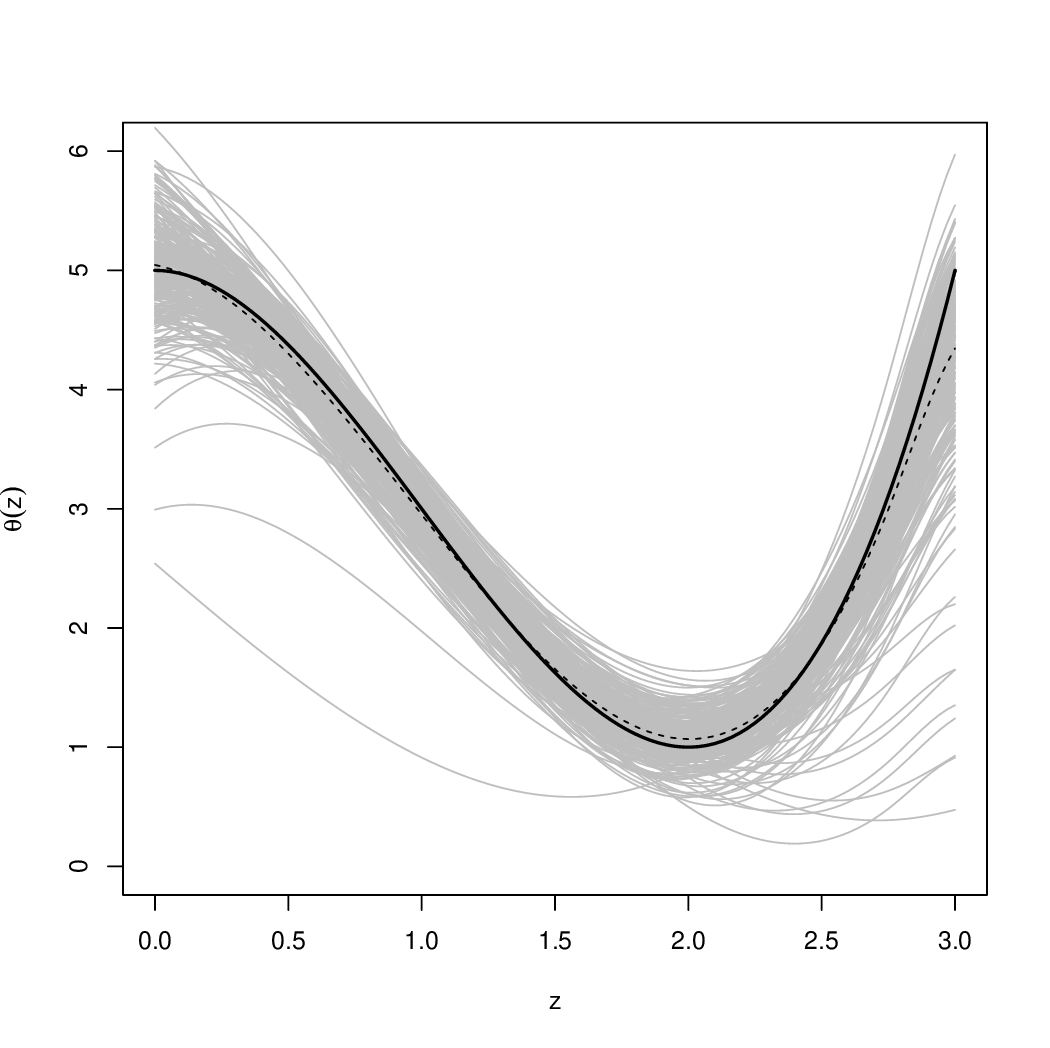}
\end{subfigure}
\caption{Left: the posterior mean function (dashed), true MCID function (solid), and data ($Y=1$ black points, $Y=-1$ gray points) for one replication of the first simulation.  Right: the Gibbs posterior mean MCID functions (solid gray) for each of $250$ repetitions of the first simulation, the overall mean function across those repetitions (dashed black), and the true MCID function (solid black).}
\label{fig:compare}
\end{figure}

The second example we consider includes a vector covariate similar to that in Example~1 of Scenario~2 in \citet{hedayat}.  We sample $n=1000$ independent observations of $(X,Y,Z)$, where $Z=(Z_1,Z_2)$ has a uniform distribution on the square $[0,3]^2$.  Given $Z=z$, the diagnostic measure $X$ has a normal distribution with mean $z_1+2z_2$ and variance $1$, and the patient-reported outcome $Y$ is a Rademacher random variable with probability 
\begin{equation}\eta_z(x) = \begin{cases}
\Phi(x;z_1+2z_2-0.05,1), & x>z_1+2z_2 \\
\Phi(x;z_1+2z_2+0.05,1), & x \leq z_1+2z_2.
\end{cases}  \end{equation}
In practice it is common to have more than one covariate, so this second example is perhaps more realistic than the first.  However, it is much more difficult to visualize the MCID function for more than one covariate, so we do not display any figures for this example.  We use tensor product B-splines with $8$ fixed B-spline functions ($16$ coefficients) to parametrize the MCID function.  Again, we use independent diffuse normal priors with zero mean and standard deviation equal to $6$ for each coefficient.  Over $100$ repetitions of this simulation we observe an average empirical misclassification rate of 24\% using the Gibbs posterior mean MCID function compared to 23\% using the true MCID function when applied to a hold-out sample of $1000$ data points.
 
\ifthenelse{1=0}{}{
\section{Concluding remarks}
\label{S:discuss}

In this paper we focus on developing some simple, yet general, techniques for establishing asymptotic concentration rates for Gibbs posteriors.  A key take-away is that the robustness to model misspecification offered by the Gibbs framework does not come at the expense of slower concentration rates.  Indeed, in the examples presented here---and others presented elsewhere \citep{syring.martin.image}---the rates achieved are the same as those achieved by traditional Bayesian posteriors and (nearly) minimax optimal. Another main point is that Gibbs posterior distributions are not inherently challenging to analyze; on the contrary, the proofs presented herein are concise and transparent.  An additional novelty to the analysis presented here is that we consider cases where the learning rate can be non-constant, i.e., either a vanishing sequence or data-dependent, and prove corresponding posterior concentration rate results.    

Here the focus has been on deriving Gibbs posteriors with the best possible concentration rates, and selection of learning rates has proceeded with these asymptotic properties in mind.  In other works \citep{syring.martin.scaling}, random learning rates are derived for good uncertainty quantification in finite-samples.  We conjecture, however, that the two are not mutually-exclusive, and that learning rates arising as solutions to the calibration algorithm in \citet{syring.martin.scaling} also have the desirable concentration rate properties.  Proving this conjecture seems challenging, but Section~\ref{SS:exp.rates} provides a first step in this direction.} 

\section*{Acknowledgments}

The authors sincerely thank the reviewers for their helpful feedback on previous versions of the manuscript. This work is partially supported by the U.S.~National Science Foundation, DMS--1811802.

\appendix

\section{Proofs of the main theorems}
\label{S:proofs}

\subsection{Proof of Theorem~1}
\label{proof:rate}

As a first step, we first state and prove a result that gives an in-probability lower bound on the denominator of the Gibbs posterior, the so-called partition function.  The proof closely follows that of Lemma~1 in \citet{shen.wasserman.2001} but is, in some sense, more general, so we present the details here for the sake of completeness.  

\begin{lem}
\label{lem:den}
Define 
\begin{equation}
\label{eq:den}
D_n = \int e^{-\omega \, n\{R_n(\theta) - R_n(\theta^\star)\}} \, \Pi(d\theta). 
\end{equation}
If $G_n = \{\theta: m(\theta,\theta^\star) \vee v(\theta,\theta^\star) \leq \eps_n^r\}$ is as in (12), with $\eps_n$ satisfying $\eps_n \to 0$ and $n \eps_n^r \to \infty$, then $D_n > \tfrac12 \Pi(G_n) e^{-2\omega n \eps_n^r}$ with $P^n$-probability converging to 1.  
\end{lem}

\begin{proof}
Define a standardized version of the empirical risk difference, i.e., 
\[ Z_n(\theta) = \frac{\{nR_n(\theta) - nR_n(\theta^\star)\} - n m(\theta)}{\{n v(\theta)\}^{1/2}}, \]
where $m(\theta) = m(\theta,\theta^\star)$ and $v(\theta) = v(\theta,\theta^\star)$, the mean and variance of the risk difference.  Of course, $Z_n(\theta)$ depends (implicitly) on the data $U^n$.  Let 
\[ \Z_n = \{(\theta, U^n): |Z_n(\theta)| \geq (n \eps_n^r)^{1/2} \}. \]
Next, define the cross-sections 
\[ \Z_n(\theta) = \{U^n: (\theta, U^n) \in \Z_n\} \quad \text{and} \quad \Z_n(U^n) = \{\theta: (\theta, U^n) \in \Z_n\}. \]
For $G_n$ as above, since 
\[ nR_n(\theta) - nR_n(\theta^\star) = n m(\theta) + \{n v(\theta)\}^{1/2} Z_n(\theta), \]
and $m$, $v$, and $Z_n$ are suitably bounded on $G_n \cap \Z_n(U^n)^c$, we immediately get 
\[ D_n \geq \int_{G_n \cap \Z_n(U^n)^c} e^{-\omega n m(\theta) - \omega\{n v(\theta)\}^{1/2} Z_n(\theta)} \, \Pi(d\theta) \geq e^{-2\omega n \eps_n^r} \Pi\{G_n \cap \Z_n(U^n)^c\}. \]
From this lower bound, we get  
\begin{align*}
P^n\{D_n \leq \tfrac12 \Pi(G_n) e^{-2\omega n \eps_n^r}\} & \leq P^n\bigl[ e^{-2\omega n \eps_n^r} \Pi\{G_n \cap \Z_n(U^n)^c\} \leq \tfrac12 \Pi(G_n) e^{-2\omega n \eps_n^r} \bigr] \\
& = P^n\bigl[ \Pi\{G_n \cap \Z_n(U^n)\} \geq \tfrac12 \Pi(G_n) \bigr] \\
& \leq \frac{2 P^n \Pi\{G_n \cap \Z_n(U^n)\}}{\Pi(G_n)},
\end{align*}
where the last line is by Markov's inequality.  We can then simplify the expectation in the upper bound displayed above using Fubini's theorem:
\begin{align*}
P^n \Pi\{G_n \cap \Z_n(U^n)\} & = \int \int 1\{\theta \in G_n \cap \Z_n(U^n)\} \, \Pi(d\theta) \, P^n(dU^n) \\
& = \int \int 1\{\theta \in G_n\} \, 1\{\theta \in \Z_n(U^n)\} \, P^n(dU^n) \, \Pi(d\theta) \\
& = \int_{G_n} P^n\{\Z_n(\theta)\} \, \Pi(d\theta).
\end{align*}
By Chebyshev's inequality, $P^n\{\Z_n(\theta)\} \leq (n\eps_n^r)^{-1}$, and hence  
\[ P^n\{D_n \leq \tfrac12 \Pi(G_n) e^{-2\omega n \eps_n^r}\} \leq 2(n\eps_n^r)^{-1}. \]
Finally, since $n\eps_n^r \to \infty$, the left-hand side vanishes, completing the proof.  
\end{proof}

For the proof of Theorem~1, write 
\[ \Pi_n(A_n) = \frac{N_n(A_n)}{D_n}, \]
where $A_n = \{\theta: d(\theta;\theta^\star) > M \eps_n\}$, $D_n$ is as in \eqref{eq:den}, and 
\[ N_n(A_n) = \int_{A_n} e^{-\omega \, n\{R_n(\theta) - R_n(\theta^\star)\}} \, \Pi(d\theta). \]
For $G_n$ as in Lemma~\ref{lem:den}, write $b_n = \frac12 \Pi(G_n) e^{-2\omega n\eps_n^r}$ for the lower bound on $D_n$.  Then 
\begin{align*}
\Pi_n(A_n) & \leq \frac{N_n(A_n)}{D_n} \, 1(D_n > b_n) + 1(D_n \leq b_n) \\
& \leq b_n^{-1} N_n(A_n) + 1(D_n \leq b_n). 
\end{align*}
By Fubini's theorem, independence of the data $U^n$, and Condition~1, we get 
\[ P^n N_n(A_n) = \int_{A_n} \{ P e^{-\omega(\ell_\theta - \ell_{\theta^\star})} \}^n \, \Pi(d\theta) < e^{-K M^r \omega n \eps_n^r}. \]
Take expectation of $\Pi_n(A_n)$ and plug in the upper bound above, along with $\Pi(G_n) \geq e^{-C_1 n \eps_n^r}$ from (15) and $P^n(D_n \leq b_n) \geq 2(n\eps_n^{r})^{-1}$ from Lemma~\ref{lem:den}, to get 
\[ P^n \Pi_n(A_n) \leq 2e^{-(\omega K M^r - C_1 - 2\omega) n \eps_n^r} + 2(n\eps_n^{r})^{-1}. \]
Since the right-hand side is vanishing for sufficiently large $M$, the claim follows.

\subsection{Proof of Theorem~2}
\label{proof:rate.rootn}

A special case of this result was first presented in \citet{bhattacharya.martin.2020}, but we are including the proof here for completeness.  

Recall that the Gibbs posterior probability, $\Pi_n(A_n)$, is a ratio, namely, $N_n(A_n) / D_n$.  Both the numerator and denominator are integrals, and the key idea here is to split the range of integration in the numerator into countably many disjoint pieces as follows:
\begin{align*}
N_n(A_n) & = \int_{d(\theta,\theta^\star) > M_n \eps_n} e^{-\omega n\{R_n(\theta) - R_n(\theta^\star)\}} \, \Pi(d\theta) \\
& = \sum_{t=1}^\infty \int_{tM_n \eps_n < d(\theta,\theta^\star) < (t+1)M_n \eps_n} e^{-\omega n\{R_n(\theta) - R_n(\theta^\star)\}} \, \Pi(d\theta).  
\end{align*}
Taking expectation of the left-hand side and moving it under the sum and under the integral on the right-hand side, we need to bound 
\[ \int_{t M_n \eps_n < d(\theta,\theta^\star) < (t+1)M_n \eps_n} \{P e^{-\omega (\ell_\theta - \ell_{\theta^\star})}\}^n \, \Pi(d\theta), \quad t=1,2,\ldots. \]
By Condition~1, on the given range of integration, the integrand is bounded above by 
\[ e^{-\omega K n (t M_n \eps_n)^r} = e^{-\omega K t^r M_n^r}, \]
so the expectation of the integral itself is bounded by 
\[ e^{-K t^r M_n^r} \Pi(\{\theta: d(\theta,\theta^\star) < (t+1) M_n \eps_n\}), \quad t=1,2,\ldots \]
Since $\Pi$ has a bounded density on the $q$-dimensional parameter space, we clearly have
\[ \Pi(\{\theta: d(\theta,\theta^\star) < (t+1) M_n \eps_n\}) \lesssim \{(t+1)M_n\eps_n\}^q. \]
Plug all this back into the summation above to get 
\[ P^n N_n(A_n) \lesssim (M_n \eps_n)^q \sum_{t=1}^\infty (t+1)^q e^{-\omega K t^r M_n^r}. \]
The above sum is finite for all $n$ and bounded by a multiple of $e^{-\omega M_n^r}$.  Then $M_n^q$ times the sum is vanishing as $n \to \infty$ and, consequently, we find that the expectation of the Gibbs posterior numerator is $o(\eps_n^q)$.  

For the denominator $D_n$, we can proceed just like in the proof of Lemma~\ref{lem:den}.  The key difference is that we redefine 
\[ \Z_n = \{(\theta, U^n): |Z_n(\theta)| \geq (Q n \eps_n^r)^{1/2} \}, \]
with an arbitrary constant $Q > 1$, so that 
\[ P^n\{D_n \leq \tfrac12 \Pi(G_n) e^{-2Q \omega n \eps_n^r}\} \leq 2(Q n \eps_n^r)^{-1}. \]
Then, just like in the proof of Theorem~1, since $n\eps_n^r=1$, we have 
\[ P^n \Pi_n(A_n) \leq \frac{o(\eps_n^q)}{e^{-2Q \omega} \eps_n^q} + 2Q^{-1}, \]
which implies 
\[ \limsup_{n \to \infty} P^n \Pi_n(A_n) \leq 2Q^{-1}. \]
Since $Q$ is arbitrary, we conclude that $P^n \Pi_n(A_n) \to 0$, completing the proof.

\subsection{Proof of Theorem~3}
\label{proof:rate.seq}

The proof is nearly identical to that of Theorem~1.  Begin with 
\[ \Pi_n(A_n) = \frac{N_n(A_n)}{D_n}, \]
where $A_n = \{\theta: d(\theta;\theta^\star) > M \eps_n\}$, $D_n$ is as in \eqref{eq:den}, and 
\[ N_n(A_n) = \int_{A_n} e^{-\omega_n \, n\{R_n(\theta) - R_n(\theta^\star)\}} \, \Pi(d\theta). \]
When the learning rate is a sequence $\omega_n$ rather than constant, Lemma~\ref{lem:den} can be applied with no alterations provided $n\omega_n\eps_n^r\rightarrow\infty$, as assumed in the statement of Theorem~3. Then, for $G_n$ as in Lemma~\ref{lem:den}, write $b_n = \frac12 \Pi(G_n) e^{-2\omega_n n\eps_n^r}$ for the lower bound on $D_n$.  Bound the posterior probability of $A_n$ by
\begin{align*}
\Pi_n(A_n) & \leq \frac{N_n(A_n)}{D_n} \, 1(D_n > b_n) + 1(D_n \leq b_n) \\
& \leq b_n^{-1} N_n(A_n) + 1(D_n \leq b_n). 
\end{align*}
By Fubini's theorem, independence of the data $U^n$, and Condition~1, we get 
\[ P^n N_n(A_n) = \int_{A_n} \{ P e^{-\omega_n(\ell_\theta - \ell_{\theta^\star})} \}^n \, \Pi(d\theta) < e^{-K M^r \omega_n n \eps_n^r}. \]
Take expectation of $\Pi_n(A_n)$ and plug in the upper bound above, along with $\Pi(G_n) \gtrsim e^{-C \omega_n n \eps_n^r}$ from (14) and $P^n(D_n \leq b_n) = o(1)$ from Lemma~\ref{lem:den}, to get 
\[ P^n \Pi_n(A_n) \lesssim e^{-(K M^r - C - 2)\omega_n n \eps_n^r} + o(1). \]
Since the right-hand side is vanishing for sufficiently large $M$, the claim follows.

\subsection{Proof of Theorem~4}
\label{proof:rate.est}

First note that if the conditions of Theorem~3 hold for $\omega_n$, then $\Pi_n^{\omega_n/2}$ also concentrates at rate $\eps_n$.  That is, at least asymptotically, there is no difference between the learning rates $\omega_n$ and $\omega_n/2$.

Next, as in the proof of Theorem~1, denote the numerator and denominator of $\Pi_n^{\hat\omega_n}(A)$ by $N_n^{\hat\omega_n}(A)$ and $D_n^{\hat\omega_n}$.  Let $W=\{U^n:\omega_n/2<\hat\omega_n<\omega_n\}$.  By the assumptions of Theorem~4, $P^n 1(W) \rightarrow 1$, so in the argument below we focus on bounding the numerator and denominator of the Gibbs posterior given $W$.

Restricting to the set $W$, using Lemma~\ref{lem:den}, and noting that $\omega \mapsto e^{-2n\omega\eps_n^r}$ decreases in $\omega$ we have $D_n^{\hat\omega_n} > b_n$ with $P^n-$probability approaching $1$ where 
\[b_n = \tfrac12 \Pi(G_n)e^{-2n\omega_n\eps_n^r} \gtrsim e^{-C_1n\omega_n\eps_n^r}\] 
for some $C_1>0$, where the last inequality follows from (14).

Set $\mathbb{W}=\{\theta:R_n(\theta)-R_n(\theta^\star)>0\}$ and then bound the numerator as as follows:
\begin{align*}
N_n^{\hat\omega_n}(A_n) &= N_n^{\hat\omega_n}(A_n \cap \mathbb{W}) + N_n^{\hat\omega_n}(A_n\cap \mathbb{W}^c)\\
&\leq  \int_{A_n\cap \mathbb{W}}e^{-\omega_n/2[R_n(\theta)-R_n(\theta^\star)]}\Pi(d\theta)+ \int_{A_n\cap \mathbb{W}^c}e^{-\omega_n[R_n(\theta)-R_n(\theta^\star)]}\Pi(d\theta)\\
&\leq \int_{A_n}e^{-\omega_n/2[R_n(\theta)-R_n(\theta^\star)]}\Pi(d\theta)+ \int_{A_n}e^{-\omega_n[R_n(\theta)-R_n(\theta^\star)]}\Pi(d\theta)\\
& = N_n^{\omega_n/2}(A_n) + N_n^{\omega_n}(A_n).
\end{align*}
Then, by Condition~1, Fubini's theorem, and independence of $U^n$, we have 
\[ P^n N_n^{\hat\omega_n}(A_n) \leq 2e^{-(1/2) KM^rn\omega_n\eps_n^r}. \]
Similar to the proof of Theorem~1, we can bound $\Pi_n^{\hat\omega_n}(A_n)$ using the above exponential bounds on $N_n^{\hat\omega_n}(A_n)$ and $D_n^{\hat\omega_n}$:
\begin{align*}
\Pi_n^{\hat\omega_n}(A_n) &\leq 1(W)N_n^{\hat\omega_n}(A_n)/D_n^{\hat\omega_n} + 1(W^c)\\
&\leq 1(W)b_n^{-1}N_n^{\hat\omega_n}(A_n) + 1(W)1(D_n\leq b_n) + 1(W^c).
\end{align*}
Taking expectation of $\Pi_n^{\hat\omega_n}(A_n)$ and applying the numerator and denominator bounds and the fact $P^n(W) \rightarrow 1$, we have
\[ P^n\Pi_n^{\hat\omega_n}(A_n) \lesssim e^{-n\omega_n\eps_n^r(M^rK/2 - C_1)} + o(1). \]
The result follows since $M>0$ is arbitrary.

\subsection{Proof of Theorem~5}
\label{proof:rate.b}

The proof is very similar to that of Theorem~1.  Start with the decomposition 
\[ \Pi_n(A_n) = \Pi_n(A_n\cap\Theta_n)+\Pi_n(A_n\cap \Theta_n^\comp),\]
where $A_n = \{\theta: d(\theta;\theta^\star) > M_n \eps_n\}$ and $\Theta_n$ is defined in Condition~2.  We consider the first term in the above decomposition. As before, we have 
\[ \Pi_n(A_n \cap \Theta_n) = \frac{N_n(A_n \cap \Theta_n)}{D_n}, \]
for $D_n$ is as in \eqref{eq:den}, and 
\[ N_n(A_n\cap\Theta_n) = \int_{A_n\cap\Theta_n} e^{-\omega_n \, n\{R_n(\theta) - R_n(\theta^\star)\}} \, \Pi(d\theta). \]
Apply Lemma~\ref{lem:den}, with $G_n = \{\theta: m(\theta,\theta^\star) \vee v(\theta,\theta^\star) \leq (K_n\eps_n)^r\}$, and write 
\[ b_n = \tfrac12 \Pi(G_n) e^{-2\omega_n nK_n^r\eps_n^r} \]
for the lower bound on $D_n$.  This immediately leads to 
\begin{align*}
\Pi_n(A_n\cap\Theta_n) & \leq \frac{N_n(A_n \cap \Theta_n)}{D_n} \, 1(D_n > b_n) + 1(D_n \leq b_n) \\
& \leq b_n^{-1} N_n(A_n \cap \Theta_n) + 1(D_n \leq b_n). 
\end{align*}
By Fubini's theorem, independence of the data $U^n$, and Condition~2, we get 
\[ P^n N_n(A_n\cap \Theta_n) = \int_{A_n\cap\Theta_n} \{ P e^{-\omega(\ell_\theta - \ell_{\theta^\star})} \}^n \, \Pi(d\theta) < e^{- M_n^r \omega_n n K_n^r \eps_n^r}. \]
Take expectation of $\Pi_n(A_n\cap\Theta_n)$ and plug in the upper bound above, along with $\Pi(G_n) \gtrsim e^{-C n \omega_n K_n^r\eps_n^r}$ from (21) and $P^n(D_n \leq b_n) = o(1)$ from Lemma~\ref{lem:den}, to get 
\[ P^n \Pi_n(A_n \cap \Theta_n) \lesssim e^{-(M_n^r - C-2)\omega_nK_n^r n \eps_n^r} + o(1). \]
Since the right-hand side is vanishing for sufficiently large $M_n$, we can conclude that 
\[ \limsup_{n \to \infty} P^n \Pi_n(A_n \cap \Theta_n) \leq \limsup_{n \to \infty} P^n \Pi_n(A_n \cap \Theta_n^\comp). \]
Of course, if \eqref{eq:post.b}, then the upper bound in the above display is 0 and we obtained the claimed Gibbs posterior concentration rate result.

\subsection{Proof of Theorem~6}
\label{proof:rate.c}

The proof is nearly identical to the proof of Theorem~5 appearing above in Section~\ref{proof:rate.b}.  However, it remains to show concentration at $\theta^\star_n$ with respect to $P\ell_\theta^n - P\ell_{\theta_n^\star}^n$ at rate $\eps_n$ implies concentration at $\theta^\star$ with respect to $R(\theta) - R(\theta^\star)$ at rate $\eps_n \vee B_n t_n^{1-r}$.

By assumption $P|\ell_\theta|^s<B_n$ for some $s>1$ and $0<B_n<\infty$ for all $\theta\in\Theta_n$ so that
\begin{align*}
P\ell_\theta &= P\ell_\theta\cdot 1(\ell_\theta \leq t_n) + P\ell_\theta\cdot 1(\ell_\theta > t_n) \\
    &\leq P\ell_\theta\cdot 1(\ell_\theta \leq t_n) +\int_{t_n}^\infty \frac{B_n}{x^s} \, dx\\
    & \leq P\ell_\theta\cdot 1(\ell_\theta \leq t_n) + B_n t_n^{1-s},
\end{align*}
by Markov's inequality.  By definition of $\ell_\theta^n$
\[ P\ell_\theta^n = P\ell_\theta 1(\ell_\theta \leq t_n)  + t_n \, P(\ell_\theta > t_n), \]
and, therefore, 
\begin{align*}
P\ell_\theta - P\ell_\theta^n &\leq  B_n t_n^{1-s} - t_n \, P(\ell_\theta > t_n) \leq B_n t_n^{1-s}.
\end{align*}
Similarly, 
\begin{align*}
    P\ell_\theta^n - P\ell_\theta & = P\{(t_n - \ell_\theta)\cdot 1(\ell_\theta > t_n)\}\\
    & \leq B_n t_n^{1-s} - P\ell_\theta\cdot 1(\ell_\theta > t_n) \\
    & \leq B_n t_n^{1-s},
\end{align*}
so we can conclude
\begin{equation}
    \label{eq:ineq1}
|P\ell_\theta - P\ell_\theta^n| \leq B_n t_n^{1-s}, \quad \theta\in \Theta_n.
\end{equation}
Next, note that by definition $R(\theta_n^\star) - R(\theta^\star) = P\ell_{\theta_n^\star}-P\ell_{\theta^\star}>0$. Using \eqref{eq:ineq1}, replace the risk by the clipped risk at $\theta^\star$ and at $\theta^\star_n$, up to error of $B_n t_n^{1-s}$ each time to see that
\[ 0<P\ell_{\theta_n^\star}-P\ell_{\theta^\star} < P\ell_{\theta_n^\star}^n-P\ell_{\theta^\star}^n + 2B_n t_n^{1-s},\]
and, since $P\ell_{\theta_n^\star}^n-P\ell_{\theta^\star}^n<0$ by definition, we have
\[ 0<P\ell_{\theta_n^\star}-P\ell_{\theta^\star} < 2B_n t_n^{1-s},\]
and, therefore, 
\begin{equation}
    \label{eq:ineq2}
 |P\ell_{\theta_n^\star}-P\ell_{\theta^\star}|<2B_nt_n^{1-s}.   
\end{equation}

Now, for some constants $C_1, \,C_2>0$ to be determined, define the sets
\begin{align*}
A_n & =\{\theta\in \Theta_n: P\ell_{\theta}^n-P\ell_{\theta_n^\star}^n > C_1B_n t_n^{1-s}\} \\
A_n' & = \{\theta\in \Theta_n: P\ell_{\theta}-P\ell_{\theta^\star} > C_2B_n t_n^{1-s}\}.
\end{align*}
Let $\vartheta\in A_n'$.  By \eqref{eq:ineq1}, swap $P\ell_\vartheta$ for $P\ell_\vartheta^n$ up to an error of $B_nt_n^{1-s}$:
\[\vartheta\in A_n' \implies P\ell_\vartheta^n - P\ell_{\theta^\star} > (C_2 - 1)B_nt_n^{1-s}.\]
Using \eqref{eq:ineq2}, swap $P\ell_{\theta^\star}$ for $P\ell_{\theta^\star_n}$ up to an error of $B_nt_n^{1-s}$:
\[\vartheta\in A_n' \implies P\ell_\vartheta^n - P\ell_{\theta^\star_n} > (C_2 - 2)B_nt_n^{1-s}.\]
Finally, use \eqref{eq:ineq1} again and swap $P\ell_{\theta^\star_n}$ for $P\ell_{\theta^\star_n}^n$ up to an error of $B_nt_n^{1-s}$:
\[\vartheta\in A_n' \implies P\ell_\vartheta^n - P\ell_{\theta_n^\star}^n > (C_2 - 3)B_nt_n^{1-s}.\]
Conclude that if we choose $C_1, \,C_2$ such that $0< C_1 < C_2-3$, then we get $A_n'\subset A_n$.  Consequently, if the Gibbs posterior vanishes on $A_n$ in $P^n$-expectation, then it also vanishes on $A_n'$.

\section{A strategy for checking \eqref{eq:post.b}}
\label{proof:prop:convex_cons}

Recall, a condition for posterior concentration under Theorems~5--6 is
\[
    P^n\Pi_n(\Theta_n^\comp)\rightarrow 0\quad as\,n\rightarrow\infty.
\]
Below we describe a strategy for checking this, based on convexity of $\ell_\theta$. 

\begin{lem}
\label{prop:convex_cons}
The Gibbs posterior distribution satisfies \eqref{eq:post.b} if the following conditions hold:
\begin{enumerate}
\item $\theta\mapsto \ell_\theta(u)$ is convex, 
\item $\inf_{\{\theta:d(\theta,\theta^\star)>\delta\}} R(\theta) - R(\theta^\star) >0$ for all $\delta>0$, and
\item The prior distribution satisfies (14) in the main article. 
\end{enumerate}
\end{lem}

\begin{proof}
Let $A := \{\theta: d(\theta, \theta^\star)>\eps\}$ for any fixed $\eps>0$ and write the Gibbs posterior probability of $A$ as
\[\Pi_n(A)=\frac{N_n(A)}{D_n} := \frac{\int_A e^{-n[R_n(\theta) - R_n(\theta^\star)]}\Pi(d\theta)}{\int e^{-n[R_n(\theta) - R_n(\theta^\star)]}\Pi(d\theta)}.\]
Assumption~\ref{prop:convex_cons} implies the minimizer $\hat\theta_n$ of $R_n(\theta)$ converges to $\theta^\star$ in $P^n$-probability; see \citet{pollard.1993}, Lemmas 1--2.  Therefore, assume $d(\hat\theta_n, \theta^\star)\leq \eps/2$ since
\[\{U^n:(\Pi_n(A)> a) \cap (d(\hat\theta_n, \theta^\star)> \eps/2)\} \rightarrow 0\]
in $P^n-$probability. 
By convexity, and the fact that $\hat{\theta}_n \notin A$,
\[R_n(\theta) - R_n(\theta^\star) \geq \inf_u \{R_n(\theta^\star + (\eps/2) u)- R_n(\theta^\star)\},\]
where the infimum is over all unit vectors.  The infimum on the RHS of the above display converges to a positive number, say $\psi>0$, in $P^n-$probability by Lemma~1 in \citet{pollard.1993}.  Therefore, 
\[\lim\inf_{\theta\in A} R_n(\theta) - R_n(\theta^\star) \geq \eta\]
with $P^n-$probability converging to $1$.  Uniform convergence of the empirical risk functions implies
\[ N_n(A) \leq e^{-n\psi}\Pi(A) \]
with $P^n$-probability converging to $1$ as $n \to \infty$.  Combining this upper bound on $N_n(A)$ with the lower bound on $D_n$ provided by Lemma~1 in the main article we have
\[\Pi_n(A)\leq e^{-n(\psi-C_1\eps_n^r)}\rightarrow 0\]
where the bound vanishes because $\psi>C_1\eps_n^r$ for all large enough $n$.  By the bounded convergence theorem, $P^n \Pi_n(A)\rightarrow 0$, as claimed.
\end{proof}

\section{Proofs of posterior concentration results for examples}
\label{proofs:examples}

\subsection{Proof of Proposition~\ref{prop:quantreg}}
\label{proof:quantreg}

The proof proceeds by checking the conditions of the extended version of Theorem~2, that based on Condition~2.  First, we confirm that $R(\theta)$ is minimized at $\theta^\star$.  Write
\begin{align*}
R(\theta) = \int_\XX \Bigl[ & (\tau - 1)  \int_{-\infty}^{\theta^\top f(x)} \{y-\theta^\top f(x)\} \, p_x(y) \, dy  \\
& + \tau \int_{\theta^\top f(x)}^\infty \{y-\theta^\top f(x)\} \, p_x(y) \, dy \Bigr] \, P(dx).
\end{align*}
Assumption~\ref{asmp:quantile_reg}(1--2) implies $R(\theta)$ can be differentiated twice under the integral:
\begin{align*}
\dot{R}(\theta) & = \int f(x) \{P_x(\theta^\top f(x)) - \tau\} \, P(dx) \\
\ddot{R}(\theta) & = \int f(x)f(x)^\top \, p_x(\theta^\top f(x)) \, P(dx),
\end{align*}
where $P_x$ denotes the distribution function corresponding to the density $p_x$.  By definition, $P_x(\theta^{\star \top} f(x)) = \tau$, so it follows that $\dot R(\theta^\star)=0$.  Moreover, the following Taylor approximation holds in the neighborhood $\{\theta:\|\theta-\theta^\star\|<\delta\}$:
\[R(\theta) = \tfrac12 (\theta-\theta^\star)^\top \ddot R(\theta^\star)(\theta-\theta^\star) + o(\|\theta-\theta^\star\|^2),\]
where Assumption~\ref{asmp:quantile_reg}(2). implies $\ddot R(\theta^\star)$ is positive definite.  Then, $R(\theta)$ is convex and minimized at $\theta^\star$.  

Next, note $\ell_\theta(u)$ satisfies a Lipschitz property:
\[\|\ell_\theta - \ell_{\theta'}\|\leq L \|\theta-\theta'\|,\]
where $L = \max\{\tau, 1-\tau\}\|f(x)\|$.  This follows from considering the cases $y<\theta^\top f(x)<\theta'^\top f(x)$, $\theta^\top f(x)<y<\theta'^\top f(x)$, and $\theta^\top f(x)<\theta'^\top f(x)<y$, and the Cauchy--Schwartz inequality.  By Assumption~\ref{asmp:quantile_reg}(1), $L$ is uniformly bounded in $x$.  Then, using the Taylor approximation for $R(\theta)$ above and following the strategy laid out in Section~4.1 of the main article we have
\[Pe^{-\omega(\ell_\theta - \ell_{\theta^\star})}\leq e^{-\omega\|\theta-\theta^\star\|^2(a - \omega L^2/2)}\]
where $2a>0$ is bounded below by the smallest eigenvalue of $\ddot R(\theta^\star)$.  Therefore, Condition~2 holds for all sufficiently small learning rates, i.e., $\omega < 2aL^{-2}$.  

Assumption~\ref{asmp:quantile_reg}(3) says the prior density is bounded away from zero in a neighborhood of $\theta^\star$.  By the above computations, 
\[\{\theta:m(\theta,\theta^\star)\vee v(\theta,\theta^\star)\leq\delta\}\supset \{\theta:\|\theta-\theta^\star\|<C\delta\}\] 
for all small enough $\delta>0$ and some $C>0$.  Therefore,
\[\Pi(\{\theta:m(\theta,\theta^\star)\vee v(\theta,\theta^\star)\leq\delta\})\geq \Pi(\{\theta:\|\theta-\theta^\star\|<C\delta\})\gtrsim \delta^J,\]
verifying the prior condition in (15).

Since $\theta \mapsto \ell_\theta(u)$ is convex and the Taylor approximation for $R(\theta)$ implies that  
\[ \|\theta-\theta^\star\|>\delta \implies  m(\theta,\theta^\star)\gtrsim \delta^2, \]
the conditions of Lemma~\ref{prop:convex_cons} hold.

\subsection{Proof of Proposition~\ref{prop:auc}}
\label{proof:auc}
For $\lambda \in (0,1)$ as in Assumption~\ref{asmp:auc}(1), define 
\[\omega_n = \frac{m+n}{2mn}\Bigl(\frac{\tau_{10}}{\lambda} + \frac{\tau_{01}}{1-\lambda}\Bigr)^{-1},\]
where $\tau_{10}$ and $\tau_{01}$ are not both 0, so that
\[(m+n)\omega_n \to \{2(\lambda \tau_{01} + (1-\lambda)\tau_{10})\}^{-1}.\] 
For any deterministic sequence $a_n$, with $a_n \to \infty$, the learning rate $a_n\omega_n$ vanishes strictly more slowly than $\min(m,n)^{-1}$, and, therefore, according to Theorem~1 in \citet{wang.martin.auc}, the Gibbs posterior with learning rate $a_n \omega_n$ concentrates at rate $n^{-1/2}$ in the sense of Definition~4 in the main article. By the law of large numbers, $\hat\tau_{01}\rightarrow \tau_{01}$ and $\hat\tau_{10}\rightarrow \tau_{10}$ in $P^n$-probability, so 
\[(m+n)\hat\omega_n \rightarrow \{2(\lambda \tau_{01} + (1-\lambda)\tau_{10})\}^{-1} \quad \text{in $P^n$-probability}.\]
Therefore, for any $\alpha \in (1/2,1)$, we have 
\[P(\tfrac{1}{2}a_n\omega_n < \alpha a_n\hat\omega_n < a_n\omega_n) \to 1, \]
and it follows from Theorem~4 that the Gibbs posterior with learning rate $\alpha a_n \hat \omega_n$ also concentrates at rate $n^{-1/2}$.  Finally, since $a_n$ is arbitrary, the constant $\alpha$ is unimportant and may be implicitly absorbed into $a_n$.  Therefore, the conclusion of Proposition~\ref{prop:auc}, omitting $\alpha$, also holds.

\subsection{Proof of Proposition~\ref{prop:light}}
\label{proof:light}

The proof proceeds by applying Lemma~\ref{prop:convex_cons} and then checking the conditions of Theorem~2.  The squared-error loss is convex in $\theta$ and its corresponding risk equals $P_X(\theta-\theta^\star)^\top XX^\top (\theta-\theta^\star) \gtrsim \|\theta-\theta^\star\|_2^2$ by Assumption 1a.2, which implies condition 2 in Lemma~\ref{prop:convex_cons}.  Further, the condition on the prior in Assumption~1a.3 is sufficient for verifying condition 3 in Lemma~\ref{prop:convex_cons}.  Then, Lemma~\ref{prop:convex_cons} implies $\Pi_n(\{\theta:\|\theta-\theta^\star\|_2>\delta\})$ vanishes in $P^n-$expectation for any $\delta>0$.

Next, verify the conditions of Theorem~2. The excess loss can be written $\ell_\theta(u) - \ell_{\theta^\star}(u) = (\theta^\star - \theta)^\top x \, \{2y - (\theta+\theta^\star)^\top x\}$.  To verify Condition~1, we start by computing the conditional expectation of $e^{-\omega(\ell_\theta - \ell_{\theta^\star})}$, given $X=x$, using the fact that $P_{Y|x} Y = x^\top \theta^\star$.  Lemma~\ref{prop:convex_cons} implies we may restrict our attention to bounded $\theta$ so that Assumption~1a.2. implies $|\theta^\top X|<B$.  Therefore, we take $\omega < (4Bb)^{-1}$ where $b$ is given in Assumption~1a.1.  Specifically, $Y$ given $x$ is subexponential with parameters $(\sigma^2, \, b)$, so that $Pe^{tY}< \exp\{tx^\top \theta^\star + t^2\sigma^2/2\}$ for all $t<(2b)^{-1}$. In the excess loss, add and subtract $2\omega(\theta^\star - \theta)^\top x$ times $x^\top \theta^\star$ and apply Assumption~1a.1 to get the following bounds:
\begin{align*}
    Pe^{-\omega(\ell_\theta - \ell_{\theta^\star})} &= Pe^{-\omega(\theta^\star - \theta)^\top X \{2Y - (\theta+\theta^\star)^\top X\}}\\
    &\leq P_X P_{Y|X}e^{\omega(\theta^\star - \theta)^\top X(\theta+\theta^\star)^\top X}e^{-2\omega(\theta^\star)^\top X (\theta^\star - \theta)^\top X}e^{-2\omega(\theta^\star - \theta)^\top X(Y-X^\top \theta^\star) }\\
    &\leq P_X e^{\omega(\theta^\star - \theta)^\top X(\theta+\theta^\star)^\top X}e^{-2\omega(\theta^\star)^\top X (\theta^\star - \theta)^\top X}e^{2\omega^2\sigma^2[(\theta^\star - \theta)^\top X]^2}\\
    &\leq P_X e^{-\omega (1-2\omega\sigma^2) \{(\theta^\star - \theta)^\top X\}^2}.
\end{align*}
Apply (16) in the paper, which is Lemma~7.26 in \citet{wasserman}, using the facts from Assumption~1a.2 that $(\theta^\star - \theta)^\top (PXX^T) (\theta^\star - \theta) \geq c \|\theta-\theta^\star\|_2^2$ and from consistency that $|\theta^\top X|<B$ so that $\{(\theta^\star - \theta)^\top X\}^2 \leq 4B^2$. Then, (16) implies
\begin{align*}
    Pe^{-\omega(\ell_\theta - \ell_{\theta^\star})} &\leq P_X e^{-\omega (1-2\omega\sigma^2) \{(\theta^\star - \theta)^\top X\}^2} \\
    &\leq e^{-\omega (1-2\omega\sigma^2)P_X\{(\theta^\star - \theta)^\top X\}^2}e^{H\omega^2 (1-2\omega\sigma^2)^2P_X\{(\theta^\star - \theta)^\top X\}^4}\\
    &\leq e^{-\omega (1-2\omega\sigma^2)\|\theta - \theta^\star\|_2^2\{c - 4B^2CH\omega(1-2\omega\sigma^2)\}}\\
    &\leq e^{-c_2\omega \|\theta^\star - \theta\|^2_2}
\end{align*}
where 
\[ H = \{\omega^2(1-2\sigma^2)^216B^4\}^{-1}\bigl\{ e^{\omega(1-2\omega\sigma^2)4B^2}-1-\omega(1-2\omega\sigma^2)4B^2 \bigr\}, \]
and where the last line holds for some $c_2>0$ and for all sufficiently small $\omega>0$.  In particular, $\omega$ must satisfy $c > 4B^2CH\omega(1-2\omega\sigma^2)$.

To verify the prior condition in (13) of the paper, we need upper bounds on $m(\theta,\theta^\star)$ and $v(\theta, \theta^\star)$.  From above, we have $m(\theta,\theta^\star)=(\theta^\star - \theta)^\top (PXX^T) (\theta^\star - \theta)$ and by Assumption~1a.2 it follows $m(\theta,\theta^\star)\leq C\|\theta - \theta^\star\|_2^2$ for some $C>0$.  To bound $v(\theta,\theta^\star)$, use the  total variance formula ``$V(X) = E\{V(X \mid Y)\} + V\{E(X \mid Y)\}$,'' noting $V(Y \mid X) \leq \sigma^2$ by Assumption~1a.1. Then,
\begin{align*}
    v(\theta, \theta^\star)&\leq \sigma^2 (\theta^\star - \theta)^\top (PXX^T) (\theta^\star - \theta) + V\{(\theta^\star - \theta)^\top XX^T (\theta^\star - \theta)\}\\
    &\leq C\sigma^2\|\theta - \theta^\star\|_2^2 + P\{(\theta^\star - \theta)^\top XX^T (\theta^\star - \theta)\}^2\\
    & \leq C\sigma^2\|\theta - \theta^\star\|_2^2 + 4B^2P\{(\theta^\star - \theta)^\top XX^T (\theta^\star - \theta)\}\\
    &\leq (C\sigma^2+4CB^2)\|\theta - \theta^\star\|_2^2.
\end{align*}
Then, by Cauchy--Schwartz,
\[\{\theta:m(\theta,\theta^\star) \vee v(\theta,\theta^\star) \leq \delta\} \supset \{\theta: \|\theta-\theta^\star\|\leq C_1^{-1/2}\delta^{1/2}\}\]
where $C_1 = \max\{C,C\sigma^2+4CB^2\}$.  Assumption~1a.3 says the prior density is bounded away from zero in a neighborhood of $\theta^\star$, so that
\[\Pi(\{\theta:m(\theta,\theta^\star) \vee v(\theta,\theta^\star) \leq \delta\})\geq \Pi(\{\theta: \|\theta-\theta^\star\|\leq C^{-1/2}\delta^{1/2}\})\gtrsim \delta^{J/2},\]
which verifies (13).

\subsection{Proof of Proposition~\ref{prop:heavy}}
\label{proof:heavy}

The proof proceeds by checking the conditions of Theorem~6.

First, note that $\theta\in\Theta_n$ along with Assumption~1b together imply $P|\ell_\theta|^{s/2} < C\Delta_n^s$ which we define to be $B_n$ for some $C>0$ and where $s>2$.  This fulfills the condition in Theorem~6 that 
\begin{equation}
    \label{eq:moment_bnd}
    P|\ell_\theta|^{s/2} < B_n < \infty,\quad\forall\theta\in\Theta_n,\quad\exists s>2.
\end{equation}

Second, we verify Condition~3.  We apply the strategy sketched in Section~4.2 of the paper: verify the Bernstein condition with $\alpha = 1$ for the clipped loss function, and then apply the inequality in (16) of the paper.  For the Bernstein condition we want to show
\begin{equation}
\label{eq:heavy_bern}
    \theta\in\Theta_n, \,m_n(\theta, \theta_n^\star) > B_nt_n^{1-s/2} \Rightarrow v_n(\theta, \theta^\star_n)\leq G(\Delta_n)m_n(\theta, \theta^\star_n)
\end{equation}
for some function $G(\cdot)$ to be specified.  By construction, the second moment of excess clipped loss satisfies $P[\{\ell_\theta^n-\ell_{\theta_n^\star}^n\}^2]\leq P[\{\ell_\theta-\ell_{\theta_n^\star}\}^2]$, which implies
\begin{align*}
    v_n(\theta,\theta_n^\star)&\leq v(\theta,\theta_n^\star)+m(\theta,\theta_n^\star)^2-m_n(\theta,\theta_n^\star)^2\\
    &\leq v(\theta,\theta_n^\star)+m(\theta,\theta_n^\star)^2.
\end{align*}
So, we want to upper bound $v(\theta,\theta_n^\star)$ and $ m(\theta,\theta_n^\star)^2$, to get a bound on $v_n(\theta,\theta_n^\star)$, and then relate this bound to $m_n(\theta, \theta_n^\star)$ to determine $G(\cdot)$.

Similarly to the proof of Proposition~\ref{prop:light} we use the total variance formula to upper bound $v(\theta,\theta_n^\star)$.  For the ``$E(V(X|Y))$" part of the formula, we have
\begin{align*}
    E(V(\ell_\theta - \ell_{\theta^\star_n}|X=x))& = P\{4\sigma^2_X(\theta^\star_n - \theta)^\top XX^\top(\theta^\star_n - \theta)\}\\
    &\lesssim \|\theta^\star_n - \theta\|_2^2P\sigma_x^2\\
    &\lesssim \|\theta^\star_n - \theta\|_2^2\\
    &\leq \|\theta - \theta^\star\|_2^2+\|\theta^\star_n - \theta^\star\|_2^2
\end{align*}
where the second line follows from the fact $\|X\|_\infty$ is bounded almost surely according to Assumption~1a and where $\sigma_x^2$ denotes the marginal variance of $Y$, given $X=x$, which has finite expectation according to Assumption~1b.  

Next, for the ``$V(E(X|Y))$" part we have 
\begin{align*}
    V(E(\ell_\theta - \ell_{\theta^\star_n}|X))& = V\{2\theta^{\star^\top} XX^\top(\theta^\star_n - \theta) + (\theta - \theta^\star_n)^\top XX^\top (\theta+\theta_n^\star)\}\\
    &\leq E\{2\theta^{\star^\top} XX^\top(\theta_n^\star - \theta)\}^2\\
    & + 2E\{[2\theta^{\star^\top} XX^\top(\theta^\star_n-\theta)][(\theta_n^\star-\theta)^\top XX^\top(\theta_n^\star+\theta)]\}\\
    & + E\{(\theta_n^\star-\theta)^{\top}XX^\top(\theta_n^\star+\theta)\}^2.\\
    &\lesssim \Delta_n^2\|\theta - \theta_n^\star\|_2^2\\
    & \leq \Delta_n^2\left [\|\theta - \theta^\star\|_2^2+\|\theta_n^\star - \theta^\star\|_2^2\right],
\end{align*}
again, using the facts $\|X\|_\infty$ is bounded almost surely and that $\|\theta\|_2 < \Delta_n$ for $\theta\in\Theta_n$.

Similarly,
\begin{align*}
    m(\theta, \theta^\star_n)^2& = [P\{2\theta^{\star^\top} XX^\top(\theta^\star_n - \theta) + (\theta - \theta^\star_n)^\top XX^\top (\theta+\theta_n^\star)\}]^2\\
    &\lesssim \{\|\theta - \theta_n^\star\|_2+\Delta_n\|\theta - \theta_n^\star\|_2\}^2\\
    & \leq \Delta_n^2\{\|\theta - \theta^\star\|_2^2+\|\theta_n^\star - \theta^\star\|_2^2\}.
\end{align*}

So far, we have showed
\[v_n(\theta, \theta^\star_n) \lesssim  \Delta_n^2\left[\|\theta- \theta^\star\|_2^2 + \|\theta_n^\star - \theta^\star\|_2^2\right].\]
And, since
\[\|\theta - \theta^\star\|_2^2 \lesssim m(\theta, \theta^\star) \lesssim \|\theta - \theta^\star\|_2^2\]
this implies
\[v_n(\theta, \theta^\star_n) \lesssim \Delta_n^2\left[m(\theta, \theta^\star) + m(\theta_n^\star, \theta^\star)\right].\]
Finally, we note that the proof of Theorem~6 above shows that \eqref{eq:moment_bnd} implies
\[|m(\theta, \theta^\star) - m_n(\theta, \theta^\star_n)| \leq 2B_nt_n^{1-s/2} \quad\text{and}\quad m(\theta_n^\star, \theta^\star) < 2B_nt_n^{1-s/2}.\]
Conclude
\[v_n(\theta, \theta^\star_n) \lesssim \Delta_n^2\left[m_n(\theta, \theta^\star_n) + 4B_nt_n^{1-s/2}\right].\]
Therefore, \eqref{eq:heavy_bern} is verified if we take $G(\Delta_n) = C\Delta_n^2$ for some $C>0$.

Next, apply the inequality in (16).  Note that $L_n:=\sup_{u, \theta\in\Theta_n}|\ell_\theta^n(u) - \ell_{\theta_n^\star}(u)| \lesssim \Delta_n t_n^{1/2}$ and recall that we choose $\omega_n = (\Delta_n t_n^{1/2})^{-1}$.  Then, (16) and \eqref{eq:heavy_bern} imply
\begin{align*}
   m_n(\theta, \theta_n^\star) > B_nt_n^{1-s/2} \implies &Pe^{-\omega_n (\ell_\theta^n - \ell_{\theta_n^\star}^n)}\leq e^{-\omega_n m_n(\theta, \theta_n^\star) + \omega_n^2 v_n(\theta, \theta_n^\star)\left\{\frac{e^{L_n\omega_n} - 1 - L_n\omega_n}{(L_n\omega_n)^2}\right\}}\\
   &\leq e^{-\omega_n m_n(\theta, \theta_n^\star) + c_n\omega_n^2 v_n(\theta, \theta_n^\star)}\\
   &\leq e^{-\omega_n m_n(\theta, \theta_n^\star) + c_n'\omega_n^2\Delta_n^2 [m_n(\theta, \theta_n^\star) + 4B_nt_n^{1-s/2}]}\\
   &\leq  e^{-\omega_n m_n(\theta, \theta_n^\star)[1-c_n'\omega_n\Delta_n^2-4c_n'\omega_n\Delta_n^2B_nt_n^{1-s/2}m_n(\theta, \theta^\star_n)^{-1}]} \\
   &\leq e^{-K_n\omega_n \eps_n^2}\quad\text{for }m_n(\theta, \theta^\star_n)>\eps_n^2:=B_nt_n^{1-s/2},
\end{align*}
where $c_n,c_n' = O(1)$ by the choice of $\omega_n$, and where $K_n = 1-5c_n'\omega_n\Delta_n^2 >1/2$ for all large enough $n$.  This verifies Condition~3.

Finally, we verify the prior bound in (26).  By Assumption~\ref{asmp:light}(3) the prior is bounded away from zero in a neighborhood of $\theta^\star$, which implies, by the above bound on $\|\theta_n^\star - \theta^\star\|_2$, that the prior is also bounded away from zero in a neighborhood of $\theta_n^\star$ for all sufficiently large $n$.  Using the above bounds on the $m_n$ and $v_n$ functions
\begin{align*}
\Pi(\{\theta:m_n(\theta,\theta_n^\star)\vee v_n(\theta, \theta_n^\star) < K_n\eps_n^2\})&\geq \Pi(\{\theta:\|\theta-\theta^\star_n\|_2^2 < \Delta_n^{-2}B_nt_n^{1-s/2}\})\\
&\gtrsim (\Delta_n^{-2}B_nt_n^{1-s/2})^J \\
& \gtrsim e^{-C\log n}\quad \exists C>0,\\
&> e^{-K_nn\omega_nB_nt_n^{1-s/2}} = e^{-K_n\Delta_n^{s-1}}
\end{align*}
where the last inequality holds because $\Delta_n = \log n$ and $s>2$.  This verifies the prior bound required by Theorem~6 in (26).

\subsection{Proof of Proposition~\ref{prop:pred_sq}}
\label{proof:pred_sq}

Proposition~\ref{prop:pred_sq} follows from a slight refinement of Theorem~3.  We verify Condition~1 and a prior bound essentially equivalent to (14).  We also use an argument similar to the proof of Theorem~2 to improve the bound on expectation of the numerator of the Gibbs posterior distribution derived in the proof of Theorem~3.   

Towards verifying Condition~1, we first need to define the loss function being used.  Even though the $x_i$ values are technically not ``data'' in this inid setting, it is convenient to express the loss function in terms of the $(x,y)$ pairs.  Moreover, while the quantity of interest is the mean function $\theta$, since we have introduced the parametric representation $\theta=\theta_\beta$ and the focus shifts to the basis coefficients in the $\beta$ vector, it makes sense to express the loss function in terms of $\beta$ instead of $\theta$.  That is, the squared error loss is 
\[ \ell_\beta(x,y) = \{y - \theta_\beta(x)\}^2. \]
For $\beta^\dagger=\beta_n^\dagger$ as defined in Section~5.4, the loss difference equals
\begin{equation}
    \label{ldiff:sqrt}
    \ell_\beta(x,y) - \ell_{\beta^\dagger}(x,y) = \{\theta_\beta(x) - \theta_{\beta^\dagger}(x)\}^2 + 2\{\theta_{\beta^\dagger}(x) - \theta_\beta(x)\}\{y - \theta_{\beta^\dagger}(x)\}. 
\end{equation} 
Since the responses are independent, the expectation in Condition~1 can be expressed as the product
\begin{align*}
P^n e^{-n\omega\{r_n(\beta)-r_n(\beta^\dagger)\}} &= \prod_{i=1}^n e^{-\omega\{ \theta_\beta(x_i) - \theta_{\beta^\dagger}(x_i)\}^2} P_i e^{-2\omega\{\theta_{\beta^\dagger}(x_i) - \theta_\beta(x_i)\}\{Y_i - \theta_{\beta^\dagger}(x_i)\}} \\
&=e^{-n\|\theta_\beta - \theta_{\beta^\dagger}\|_{n,2}^2} \prod_{i=1}^n P_i e^{-2\omega \{\theta_{\beta^\dagger}(x_i) - \theta_\beta(x_i)\}\{Y_i - \theta_{\beta^\dagger}(x_i)\}}.
\end{align*}
According to Assumption~\ref{asmp:pred_sq}(2), $Y_i$ is sub-Gaussian, so the product in the last line above can be upper-bounded by
\[ e^{4n\omega^2\sigma^2\|\theta_\beta - \theta_{\beta^\dagger}\|_{n,2}^2} \times e^{-2\omega_n \sum_{i=1}^n \{ \theta_{\beta^\dagger}(x_i) - \theta_\beta(x_i)\}\{\theta^\star(x_i)-\theta_{\beta^\dagger}(x_i)\}}. \]
The second exponential term above is identically $1$ because the exponent vanishes---a consequence of the Pythagorean theorem.  To see this, first write the quantity in the exponent as an inner product
\[ (\beta-\beta^\dagger)^\top F_n^\top \{\theta^\star(x_{1:n}) - F_n \beta^\dagger\} = (\beta-\beta^\dagger)^\top \{ F_n^\top \theta^\star(x_{1:n}) - F_n^\top F_n \beta^\dagger\}. \]
Recall from the discussion in Section~5.4 that $\beta^\dagger$ satisfies $(F_n^\top F_n) \beta^\dagger = F_n^\top \theta^\star(x_{1:n})$; from this, it follows that the above display vanishes for all vectors $\beta$.  Therefore, 
\[ P^n e^{-n\omega\{r_n(\beta)-r_n(\beta^\dagger)\}} \leq e^{-n\omega\|\theta_\beta - \theta_{\beta^\dagger}\|_{n,2}^2 (1-2\omega\sigma^2)}, \]
and Condition~1 is satisfied since provided the learning rate $\omega$ is less than $(2\sigma^2)^{-1}$.

Next, we derive a prior bound similar to (14). By Assumption~\ref{asmp:pred_sq}(3), all eigenvalues of $F_n^\top F_n$ are bounded away from zero and infinity, which implies
\begin{equation}
\label{eq:rm.sq.2}
    \|\beta-\beta^\dagger\|_2^2 \lesssim \|\theta_{\beta}-\theta_{\beta^\dagger}\|_{n,2}^2\lesssim \|\beta-\beta^\dagger\|_2^2.
\end{equation}
In the arguments that follow, we show that $\|\theta_{\beta^\star}-\theta_{\beta^\dagger}\|_{n,2} \lesssim \eps_n$ which implies, by \eqref{eq:rm.sq.2}, that $\|\beta^\star_n-\beta^\dagger\|_2^2 \lesssim \eps_n^2$.  The approximation property in (43) implies that $\|\beta_n^\star\|_\infty < H$.  Therefore, $\|\beta^\dagger\|_\infty$ is bounded because, if it were not bounded, then \eqref{eq:rm.sq.2} would be contradicted.

Since $\|\beta^\dagger\|_\infty$ is bounded, Assumption~\ref{asmp:pred_sq}(5) implies that the prior for $\beta$ satisfies 
\begin{equation}
\label{eq:rm.sq.1}
\Pitilde(\{\beta:\|\beta-\beta^\dagger\|_2\leq \eps_n\}) \gtrsim \eps_n^{J_n}e^{-J_n\log C}. 
\end{equation}
Recall that (14) involves the mean and variance of the empirical risk, and we can directly calculate these.  For the mean, 
\begin{align*}
m(\theta_\beta, \theta_{\beta^\dagger}) & = \bar r_n(\beta) - \bar r_n(\beta^\dagger) \\
& = \|\theta_\beta - \theta_{\beta^\dagger}\|_{n,2}^2 + \sum_{i=1}^n \{\theta_{\beta^\dagger}(x_i) - \theta^\star(x_i)\}\{\theta_{\beta^\dagger}(x_i) - \theta_\beta(x_i)\} \\
& = \|\theta_\beta - \theta_{\beta^\dagger}\|_{n,2}^2,
\end{align*}
where the last equality is by the same Pythagorean theorem argument presented above.  Similarly, the variance is given by $v(\theta_\beta, \theta_{\beta^\dagger}) = 4\sigma^2 n^{-1} \|\theta_\beta - \theta_{\beta^\dagger}\|_{n,2}^2$.  Therefore, 
\[ \{ m(\theta_\beta, \theta_{\beta^\dagger}) \vee v(\theta_\beta, \theta_{\beta^\dagger}) \} \lesssim \|\theta_\beta - \theta_{\beta^\dagger}\|_{n,2}^2, \]
and \eqref{eq:rm.sq.1} and \eqref{eq:rm.sq.2} imply
\begin{align*}
\tilde\Pi(\{\beta:m(\theta_\beta, \theta_{\beta^\dagger}) \vee v(\theta_\beta, \theta_{\beta^\dagger})\leq \eps_n^2\})&\gtrsim \eps_n^{J_n}e^{-CJ_n}.
\end{align*}
The fact that $n\eps_n^2 = J_n$ along with Lemma~\ref{lem:den} implies 
\[D_n \gtrsim \eps_n^{J_n}e^{-(2\omega+\log C)J_n}\]
with $P^n$-probability converging to $1$ as $n\rightarrow\infty$.

We briefly verify the claim made just before the statement of Proposition~2, i.e., that (44) holds for a independence prior consisting of strictly positive densities.  To see this, note that the volume of a $J_n$-dimensional sphere with radius $\eps$ is a constant multiple of $\eps^{J_n}$, so that, if the joint prior density is bounded away from zero by $C^{J_n}$ on the sphere, then we have $\Pitilde(\{\beta:\|\beta-\beta'\|_2 \leq \eps\}) \gtrsim (C\eps)^{J_n}$, which is (44).  So we must verify the bound on the prior density.  Suppose $\tilde \Pi$ has a density $\tilde \pi$ equal to the a product of independent prior densities $\tilde \pi_j$, $j = 1,\ldots,J_n$.  Since the cube $\{\beta:\|\beta-\beta'\|_\infty \leq J_n^{1/2}\eps\}$ contains the ball $\{\beta:\|\beta-\beta'\|_2 \leq \eps\}$ by Cauchy--Schwartz, the infimum of the prior density on the ball is no smaller than the infimum of the prior density on the cube.  To bound the prior density on the cube consider any of the $J$ components and note $|\beta_j-\beta'_j|\leq J_n^{1/2}\eps_n$ implies $\beta_j \in [-H-J_n^{1/2}\eps_n,  H+J_n^{1/2}\eps_n]$ since $\|\beta'\|_\infty \leq H$.  Moreover, since $\alpha\geq 1/2$ we have $J_n^{1/2}\eps_n\rightarrow 0$ so that this interval lies within a compact set, say, $[-H-1, H+1]$.  And, since the prior density $\tilde \pi_j$ is strictly positive, it is bounded away from zero by a constant $C$ on this compact set.  Finally, by independence, we have $\pi(\beta) \geq C^J$ for all $\beta\in \{\beta:\|\beta-\beta'\|_2 \leq \eps\}$, which verifies the claim concerning (44).    

In order to obtain the optimal rate of $\eps_n = n^{-\alpha/(1+2\alpha)}$, without an extra logarithmic term, we need a slightly better bound on $P^n N_n(A_n)$ than that used to prove Theorem~3.  Our strategy, as in the proof of Theorem~2, will be to split the range of integration in the numerator into countably many disjoint pieces, and use bounds on the prior probability on those pieces to improve the numerator bound.  Define $\tilde{A}_n:=\{\beta: \|\theta_\beta - \theta_{\beta^\dagger}\|_{n,2}>M_n\eps_n\}$ and write the numerator $N_n(\tilde A_n)$ as follows 
\begin{align*}
N_n(\tilde A_n) & = \int_{\|\theta_\beta - \theta_{\beta^\dagger}\|_{n,2} > M_n \eps_n} e^{-n\omega \{r_n(\beta) - r_n(\beta^\dagger)\}} \, \tilde\Pi(d\beta) \\
& = \sum_{t=1}^\infty \int_{tM_n \eps_n < \|\theta_\beta - \theta_{\beta^\dagger}\|_{n,2} < (t+1)M_n \eps_n} e^{-n\omega \{r_n(\beta) - r_n(\beta^\dagger)\}} \, \tilde\Pi(d\beta).  
\end{align*}
Taking expectation of the left-hand side and moving it under the sum and under the integral on the right-hand side, we need to bound 
\[ \int_{tM_n \eps_n < \|\theta_\beta - \theta_{\beta^\dagger}\|_{n,2} < (t+1)M_n \eps_n} P^n e^{-n\omega \{r_n(\beta) - r_n(\beta^\dagger)\}}  \, \tilde\Pi(d\beta), \quad t=1,2,\ldots. \]
By Condition~1, verified above, on the given range of integration the integrand is bounded above by $e^{-n\omega (t M_n \eps_n)^2/2}$, so the expectation of the integral itself is bounded by 
\[ e^{-n\omega_n (t M_n \eps_n)^2/2} \Pi(\{\beta: \|\theta_\beta - \theta_{\beta^\dagger}\|_{n,2} < (t+1) M_n \eps_n\}), \quad t=1,2,\ldots \]
Since $\tilde\Pi$ has a bounded density on the $J_n-$dimensional parameter space, we clearly have
\[ \Pi(\{\beta: \|\theta_\beta - \theta_{\beta^\dagger}\|_{n,2} < (t+1) M_n \eps_n\}) \lesssim \{(t+1)M_n\eps_n\}^{J_n}. \]
Plug all this back into the summation above to get 
\[ P^n N_n(\tilde A_n) \lesssim (M_n \eps_n)^{J_n} \sum_{t=1}^\infty (t+1)^{J_n} e^{-\omega (t M_n)^2J_n/2}. \]
The above sum is finite for all $n$ and bounded by a multiple of $e^{-\omega M_n^2J_n/4}$.  Consequently, we find that the expectation of the Gibbs posterior numerator is bounded by a constant multiple of $(M_n\eps_n)^{J_n}e^{-\omega M_n^2J_n/4}$. 

Combining the in-expectation and in-probability bounds on $N_n(\tilde A_n)$ and $D_n$, respectively, as in the proof of Theorem~1, we find that
\[P^n\Pi_n(\tilde A_n)\lesssim e^{-J_n( \omega M_n^2/4 - \log M_n - 2\omega - \log C)}\]
which vanishes as $n\rightarrow\infty$ for any sufficiently large constant $M_n\equiv M>0$.

The above arguments establish that the Gibbs posterior $\Pitilde_n$ for $\beta$ satisfies  
\begin{equation}
    \label{set1}
P^n \Pitilde_n(\{\beta \in \RR^J: \|\theta_\beta - \theta_{\beta^\dagger}\|_{n,2} > M_n\eps_n\}) \to 0.
\end{equation}  
But this is equivalent to the proposition's claim, with $\theta_{\beta^\dagger}$ replaced by $\theta^\star$.  To see this, first recall that Assumption~2.4 implies the existence of a $J$-vector $\beta^\star=\beta_n^\star$ such that $\|\theta_{\beta^\star} - \theta^\star\|_\infty$ is  small.  Next, use the triangle inequality to get 
\[ \|\theta_\beta - \theta^\star\|_{n,2} \leq \|\theta_\beta - \theta_{\beta^\dagger}\|_{n,2} + \|\theta_{\beta^\dagger} - \theta^\star\|_{n,2}. \]
Now apply the Pythagorean theorem argument as above to show that 
\[ \|\theta_{\beta^\star} - \theta^\star\|_{n,2}^2 = \|\theta_{\beta^\star} - \theta_{\beta^\dagger}\|_{n,2}^2 + \|\theta_{\beta^\dagger} - \theta^\star\|_{n,2}^2. \]
Since the sup-norm dominates the empirical $L_2$ norm, the left-hand side is bounded by $CJ^{-2\alpha}$ for some $C > 0$.  But both terms on the right-hand side are non-negative, so it must be that the right-most term is also bounded by $CJ^{-2\alpha}$.  Putting these together, we find that
\[ \|\theta_\beta - \theta^\star\|_{n,2} > M_n'\eps_n \implies \|\theta_\beta - \theta_{\beta^\dagger}\|_{n,2} > M_n'\eps_n - C^{1/2} J^{-\alpha}. \]
Therefore, with $\eps_n=n^{-\alpha/(2\alpha + 1)} \log n$ and $J=J_n=n^{1/(2\alpha+1)}$, the lower bound on the right-hand side of the previous display is a constant multiple of $\eps_n$.  We can choose $M_n'$ such that the aforementioned sequence is at least as big as $M_n$ above.  In the end, 
\[ \Pitilde_n(\{\beta: \|\theta_\beta - \theta^\star\|_{n,2} > M_n'\eps_n\}) \leq \Pitilde_n(\{\beta: \|\theta_\beta - \theta_{\beta^\dagger}\|_{n,2} > M_n \eps_n\}), \]
so the Gibbs posterior concentration claim in the proposition follows from that established above.  Finally, by definition of the prior and Gibbs posterior for $\theta$, we have that 
\[ P^n \Pi_n(\{\theta \in \Theta: \|\theta-\theta^\star\|_{n,2} > M_n'\eps\}) \to 0, \]
which completes the proof.

\subsection{Proof of Proposition~\ref{prop:massart}}

The proof proceeds by checking the conditions of Theorem~1.  We begin by verifying (12). Evaluate $m(\theta,\theta^\star)$ and $v(\theta,\theta^\star)$ for the loss function $\ell_\theta$ as defined above:
\begin{align*}
m(\theta;\theta^\star) & = P\{Y \neq \phi_\theta(X)\} - P\{Y \neq \phi_{\theta^\star}(X)\} \\
& = \int_{x^\top \theta < 0, x^\top \theta^\star > 0} (2\eta - 1) \, dP + \int_{x^\top \theta > 0, x^\top \theta^\star < 0} (1-2\eta) \, dP \\ 
v(\theta,\theta^\star) & \leq P(\ell_\theta - \ell_{\theta^\star})^2 \\
& = P(X^\top \theta < 0, X^\top \theta^\star > 0) + P(X^\top \theta > 0, X^\top \theta^\star < 0) \\
& = P(\phi_\theta - \phi_{\theta^\star})^2.
\end{align*}
It follows from arguments in \citet{tsybakov} that, under the margin condition in Assumption~\ref{asmp:massart}(5), with $h \in (0,1)$, we have 
\[  hP(\phi_\theta - \phi_{\theta^\star})^2\leq  m(\theta,\theta^\star). \]
Further, rewrite $m(\theta,\theta^\star)$ as  
\begin{align}
\label{eq.int.sparse}
m(\theta;\theta^\star) & = \int \eta(x) \{\phi_{\theta^\star}(x) - \phi_\theta(x)\} \, P(dx) + \int [1-\eta(x)]\{ \phi_\theta(x) - \phi_{\theta^\star}(x) \} \, P(dx) \nonumber\\ 
& \leq 2 \int |\phi_\theta(x) - \phi_{\theta^\star}(x)| \, P(dx), 
\end{align}
where the latter inequality follows since $\eta < 1$.  Since $\phi_\theta - \phi_{\theta^\star}$ is a difference of indicators
\[P(\phi_\theta - \phi_{\theta^\star})^2=P|\phi_\theta - \phi_{\theta^\star}| \quad \text{and} \quad m(\theta,\theta^\star)\lesssim P(\phi_\theta - \phi_{\theta^\star})^2.\]
This latter inequality will be useful below.  Under the stated assumptions, the integrand in \eqref{eq.int.sparse} can be handled exactly like in Lemma~4 of \citet{jiang.tanner.2008}.  That is, if $\|\beta - \beta^\star\|_1$ is sufficiently small, then $m(\theta; \theta^\star) \lesssim \|\beta - \beta^\star\|_1$.  Since $v(\theta;\theta^\star) \lesssim m(\theta; \theta^\star)$, it follows that 
\[ \{\theta: m(\theta; \theta^\star) \vee v(\theta;\theta^\star) \leq \eps^2\} \supseteq \{\theta = (\alpha,\beta): \|\beta-\beta^\star\|_1 \leq c \eps^2\}, \]
for a constant $c > 0$.  To lower-bound the prior probability of the event on the right-hand side, we follow the proof of Lemma~2 in \citet{castillo.etal.2015}.  First, for $S^\star$ the configuration of $\beta^\star$, we can immediately get 
\[ \Pi(\{\beta: \|\beta-\beta^\star\|_1 \leq c \eps^2\}) \geq \pi(S^\star) \int_{\|\beta_{S^\star} - \beta_{S^\star}^\star\|_1 \leq c \eps^2} g_{S^\star}(\beta_{S^\star}) \, d\beta_{S^\star}. \]
Now make a change of variable $b = \beta_{S^\star}-\beta_{S^\star}^\star$ and note that 
\[ g_{S^\star}(\beta_{S^\star}) = g_{S^\star}(b + \beta_{S^\star}^\star) \geq e^{-\lambda \|\beta^\star\|_1} g_{S^\star}(b). \]
Therefore, 
\[ \Pi(\{\beta: \|\beta-\beta^\star\|_1 \leq c \eps^2\}) \geq \pi(S^\star) e^{-\lambda\|\beta^\star\|_1} \int_{\|b\|_1 \leq c \eps^2} g_{S^\star}(b) \, db, \]
and after plugging in the definition of $\pi(S^\star)$, using the bound in Equation~(6.2) of \citet{castillo.etal.2015}, and simplifying, we get 
\[ \Pi(\{\beta: \|\beta-\beta^\star\|_1 \leq c \eps^2\}) \gtrsim f(|S^\star|) \, q^{-2|S^\star|} e^{-\lambda\|\beta^\star\|_1}. \]
From the form of the complexity prior $f$, the bounds on $\lambda$, and the assumption that $\|\beta^\star\|_\infty=O(1)$, we see that the lower bound is no smaller than $e^{-C |S^\star|\log q}$ for some constant $C > 0$, which implies (12), i.e., 
\[ \Pi(\{\theta: m(\theta; \theta^\star) \vee v(\theta;\theta^\star) \leq \eps_n^2\}) \gtrsim e^{-C n\eps_n^2}, \]
where $\eps_n = \{n^{-1} |S^\star| \log q\}^{1/2}$.

Next, we verify Condition~1.  By direct computation, we get 
\begin{align*}
Pe^{-\omega(\ell_\theta - \ell_{\theta^\star})} &= 1-P(\phi_\theta - \phi_{\theta^\star})^2 + e^{-\omega} \int \eta(x) 1\{x^\top\theta\leq 0, x^\top\theta^\star>0\}\,P(dx) \\
& \qquad + e^{-\omega} \int (1-\eta(x)) 1\{x^\top\theta> 0, x^\top\theta^\star\leq 0\}\,P(dx) \\
& \qquad + e^{\omega} \int (1-\eta(x)) 1\{x^\top\theta\leq 0, x^\top\theta^\star>0\}\,P(dx) \\
& \qquad + e^\omega  \int \eta(x) 1\{x^\top\theta> 0, x^\top\theta^\star\leq 0\}\,P(dx).
\end{align*}  
Using the Massart margin condition, we can upper bound the above by 
\begin{align*}
Pe^{-\omega(\ell_\theta - \ell_{\theta^\star})} &\leq 1-\min\{a,b\}P(\phi_\theta - \phi_{\theta^\star})^2
\end{align*} 
where $a = 1-e^{-\omega}-(\tfrac12 - \tfrac h2)(e^{\omega-e^{-\omega}})$ and $b = 1-e^{\omega}+(\tfrac12 + \tfrac h2)(e^{\omega-e^{-\omega}})$.  For all small enough $\omega$, both $a$ and $b$ are $O(h\omega)$, so for some constants $c,c'>0$,
\begin{align*}
Pe^{-\omega(\ell_\theta - \ell_{\theta^\star})} &\leq 1-ch\omega P(\phi_\theta - \phi_{\theta^\star})^2\\
&\leq  1-c'\omega m(\theta,\theta^\star).
\end{align*}
Then Condition~1 follows from the elementary inequality $1-t\leq e^{-t}$ for $t>0$. 

\subsection{Proof of Proposition~\ref{prop:tsybakov}}
\label{proof:prop:sparse_reg_tsy}

The proof proceeds by checking the conditions of Theorem~3.  We begin by verifying (14).  The upper bounds on $m(\theta,\theta^\star)$ and $v(\theta,\theta^\star)$ given in the proof of Proposition~\ref{prop:massart} are also valid in this setting, which means 
\[\Pi(\{\theta: m(\theta,\theta^\star) \vee v(\theta,\theta^\star) \leq \eps^2\})\geq \Pi(\{\theta: \|\beta-\beta^\star\|_1\leq C_1\eps^2\})\]
for some $C_1>0$.   Further, Assumption~\ref{asmp:tsybakov} implies \[ \Pi(\{\theta: \|\beta-\beta^\star\|_1\leq C_1\eps^2\}) \gtrsim \eps^{2q}=e^{-2q\log\eps}. \]
By definition of $\eps_n$ and $\omega_n$ we have that 
\[ \exp(-2q\log\eps_n) \geq \exp\{-n\omega_n\eps_n^{(2+2\gamma)/\gamma}\}, \]
which, combined with the previous display, verifies (14).    

Next, to verify Condition~1 we note the misclassification error loss function difference $\ell_\theta(u) - \ell_{\theta^\star}(u)$ is bounded in absolute value by $1$, so we can proceed with using the moment-generating function bound from Lemma 7.26 in \citet{wasserman}; see (16).  Proposition~1 in \citet{tsybakov} provides the lower bound on $m(\theta,\theta^\star)$ we need, i.e., Tsybakov proves that 
\[ \text{our Assumption~\ref{asmp:tsybakov}(2)} \implies m(\theta,\theta^\star)\gtrsim d(\theta,\theta^\star)^{(2+2\gamma)/\gamma}. \] With this and the upper bound on $v(\theta,\theta^\star)$ derived above, (16) implies
\[Pe^{-\omega_n (\ell_\theta - \ell_{\theta^\star})} \leq e^{-C_2 \omega_n\eps_n^{(2+2\gamma)/\gamma}}\]
for some $C_2>0$, which verifies Condition~1.

\subsection{Proof of Proposition~\ref{prop:quant_reg}}
\label{proof:quant_reg}
The proof proceeds by checking the conditions of Theorem~5. First, we verify (21).  Starting with $m(\theta,\theta^\star)$, by direct calculation, 
\begin{align*}
m(\theta,\theta^\star) = \frac12 \int_\XX \Bigl\{ & \int \bigl( |\theta(x) \vee \theta^\star(x) - y| - |\theta(x) \wedge \theta^\star(x) - y| \bigr) \, P_{x}(dy) \\ 
& + (1-2\tau)|\theta(x) - \theta^\star(x)| \Bigr\} \, P(dx). 
\end{align*}
Partitioning the range of integration with respect to $y$, for given $x$, into the disjoint intervals $(-\infty, \theta \wedge \theta^\star]$, $(\theta \wedge \theta^\star, \theta \vee \theta^\star)$, and $[\theta \vee \theta^\star, \infty)$, and simplifying, we get 
\begin{equation}
\label{eq:m_theta}
     m(\theta,\theta^\star) = \frac12 \int_\XX \int_{\theta(x) \wedge \theta^\star(x)}^{\theta(x) \vee \theta^\star(x)} \{ \theta(x) \vee \theta^\star(x) - y \} \, P_{x}(dy) \, P(dx).
\end{equation}
It follows immediately that $m(\theta,\theta^\star_n) \lesssim \|\theta-\theta^\star\|_{L_1(P)}$.  Also, since $\theta \mapsto \ell_\theta(x,y)$ clearly satisfies the fixed-$x$ Lipschitz bound 
\begin{equation}
\label{eq:lip.like}
|\ell_\theta(x,y)-\ell_{\theta^\star}(x,y)| \leq |\theta(x) - \theta^\star(x)|, \quad \text{for all $(x,y)$}, 
\end{equation}
we get a similar bound for the variance, i.e., $v(\theta,\theta^\star_n) \leq \|\theta-\theta^\star\|^2_{L_2(P)}$.  
Let $J_n=n^{1/(1+2\alpha)}$ and $\hat\theta_{J, \beta}:=\beta^\top f$, and define a sup-norm ball around $\theta^\star$
\[B_n^\star  = \{(\beta,J): \beta \in \RR^J, J=J_n, \|\theta^\star - \hat\theta_{J, \beta}\|_\infty\leq CJ_n^{-\alpha}\}.\]
By the above upper bounds on $m(\theta,\theta^\star)$ and $v(\theta,\theta^\star)$ in terms of $\|\theta-\theta^\star\|_{L_2(P)}^2$, we have
\[\|\theta^\star - \hat\theta_{J, \beta}\|_\infty\leq CJ_n^{-\alpha} \implies \{m(\theta,\theta^\star)\vee v(\theta,\theta^\star)\}\lesssim J_n^{-2\alpha}.\]
Then, by Assumptions~\ref{asmp:quant_reg}(3-4), and using the same argument as in the proof of Theorem~1 in \citet{shen.ghosal.2015} we have
\begin{align*}
    \Pi^{(n)}(\{m(\theta,\theta^\star)\vee v(\theta,\theta^\star)\}\lesssim J_n^{-2\alpha}) &= \Pi(\{m(\theta,\theta^\star)\vee v(\theta,\theta^\star)\}\lesssim J_n^{-2\alpha})/\Pi(\Theta_n)\\
    &\geq\Pi(\{m(\theta,\theta^\star)\vee v(\theta,\theta^\star)\}\lesssim J_n^{-2\alpha})\\
    &\geq \Pi(B_n^\star) \gtrsim  e^{-C_1J_n\log n},
\end{align*}
for some $C_1>0$. By the definitions of $\eps_n$ and $\omega_n$ in Proposition~\ref{prop:quant_reg} it follows that $J_n^{-2\alpha}\leq\Delta_n^{-2}\eps_n^2$ for $\Delta_n$ as defined in Condition~2. Define $K_n\propto\Delta_n^{-1}$, with precise proportionality determined below.  Then,
\[C_1J_n(\log n)\leq CnK_n^2\omega_n\eps_n^2,\]
for a sufficiently small $C>0$ and all large enough $n$, which verifies the prior condition in (21) with $r=2$.

Next we verify Condition~2. Define the sets $A_n:=\{\theta: \|\theta-\theta^\star\|_{L_2(P)} > M\eps_n\}$ and $\Theta_n:= \{\theta: \|\theta\|_\infty \leq \Delta_n\}$.  Note that Assumption~\ref{asmp:quant_reg}(4) implies that $\Pi_n(A_n \cap \Theta_n^\comp) = 0$, and, therefore,
\[\Pi_n(A_n) = \Pi_n(A_n\cap \Theta_n).\]
The following computations provide a lower bound on $m(\theta,\theta^\star)$ for $\theta\in \Theta_n$.  Partition $\XX$ as $\XX = \XX_1 \cup \XX_2$ where $\XX_1 =\{x:|\theta(x)-\theta^\star(x)|\geq \delta\}$ and $\XX_2 = \XX_1^\comp$ and where $\delta>0$ is as in Assumption~4(2).  Using \eqref{eq:m_theta}, the mean function can be expressed as 
\begin{align*}
    2 m(\theta,\theta^\star) & = \int_{\XX_1} \int_{\theta(x) \wedge \theta^\star(x)}^{\theta(x) \vee \theta^\star(x)} \{ \theta(x) \vee \theta^\star(x) - y \} \, P_{x}(dy) \, P(dx)\\
    & \qquad + \int_{\XX_2} \int_{\theta(x) \wedge \theta^\star(x)}^{\theta(x) \vee \theta^\star(x)} \{ \theta(x) \vee \theta^\star(x) - y \} \, P_{x}(dy) \, P(dx).
\end{align*}
For convenience, refer to the two integrals on the right-hand side of the above display as $I_1$ and $I_2$, respectively.  Using Assumption~\ref{asmp:quant_reg}(2) and replacing the range of integration in the inner integral by a $(\delta/2)$-neighborhood of $\theta^\star(x)$ we can lower bound $I_1$ as 
\begin{align*}
    I_1 & \geq \int_{\XX_1 \cap \{x: \theta^\star(x)>\theta(x)\}}\int_{\theta^\star(x)-\delta}^{\theta^\star(x)-\delta/2} \{\theta^\star(x) - y\} \, P_{x}(dy) \, P(dx) \\
    & \qquad + \int_{\XX_1 \cap \{x: \theta^\star(x)\leq\theta(x)\}}\int_{\theta^\star(x)}^{\theta^\star(x)+\delta/2} \{\theta(x) - y\} \, P_{x}(dy) \, P(dx) \\
    & \geq (\beta \delta^2/4) \, P(\XX_1).
\end{align*}
Next, for $I_2$, we can again use Assumption~\ref{asmp:quant_reg}(2) to get the lower bound
\begin{align*}
    I_2 & \geq \int_{\XX_2}\int_{\theta(x)\wedge \theta^\star(x)}^{\{\theta(x)+\theta^\star(x)\}/2}  \{ \theta(x) \vee \theta^\star(x) - y \} \, P_{x}(dy) \, P(dx)\\
    & \geq \frac{\beta}{2} \int_{\XX_2} |\theta(x)-\theta^\star(x)|^2\, P(dx).
\end{align*}
Similarly, for sufficiently large $n$, if $\theta \in \Theta_n$, then the $L_2(P)$ norm is bounded as
\[\|\theta - \theta^\star\|_{L_2(P)}^2\leq \int_{\XX_2}  |\theta(x)-\theta^\star(x)|^2\, P(dx) + (\Delta_n)^2 \, P(\XX_1).\]
Comparing the lower and upper bounds for $m(\theta,\theta^\star)$ and $\|\theta - \theta^\star\|_{L_2(P)}^2$ in terms of integration over $\XX_1$ and $\XX_2$ we have
\[\int_{\XX_2}  |\theta(x)-\theta^\star(x)|^2\, P(dx)\lesssim I_2,\]
and
\[(\Delta_n)^{-2}\int_{\XX_1}  |\theta(x)-\theta^\star(x)|^2\, P(dx) \lesssim I_1,\]
which together imply
\[ m(\theta,\theta^\star) \gtrsim (\Delta_n)^{-2}\|\theta - \theta^\star\|_{L_2(P)}^2.\]

Recall, from \eqref{eq:lip.like}, that $\theta \mapsto \ell_\theta(x,y)$ is $1-$Lipschitz. Therefore, if $\theta\in \Theta_n$, for large enough $n$, then the loss difference is bounded by $\Delta_n$, so Lemma~7.26 in \citet{wasserman}, along with the lower and upper bounds on $m(\theta,\theta^\star)$ and $v(\theta,\theta^\star)$, can be used to verify Condition~2.  That is, there exists a $K>0$ such that for all $\theta \in \Theta_n$
\begin{align*}
    Pe^{-\omega_n(\ell_\theta -\ell_{\theta^\star})}&\leq \exp\{2\omega_n^2v(\theta,\theta^\star) - K\omega_n m(\theta,\theta^\star)\} \\
    & \leq \exp[-K\omega_n\Delta_n^{-2} \|\theta-\theta^\star\|_{L_2(P)}^2\{1-2\omega_n\Delta_n^2\}]\\
    & \leq \exp(-K_n\omega_n\eps_n^2) 
\end{align*}
where the last inequality holds for $K_n=(K/2)\Delta_n^{-2}$.  Given Assumption~\ref{asmp:quant_reg}(4), the above inequality verifies Condition~2 for $\omega_n$ and $\Delta_n$ as in Proposition~4.

\subsection{Proof of Proposition~\ref{prop:pers_mcid}}
\label{proof:pers_mcid}

Proposition~\ref{prop:pers_mcid} follows from Theorem~1 by verifying (12) and Condition~1.

First, we verify (12).  By definition
\begin{align*}
    m(\theta,\theta^\star) &= R(\theta)-R(\theta^\star) \\
    & = \int_\mathbb{Z} \int_{\theta(z)\wedge \theta^\star(z)}^{\theta(z)\vee \theta^\star(z)} |2\eta_z(x)-1|P(dx)P(dz).
\end{align*}
And, since $\ell_\theta$ is bounded by $1$, 
\[v(\theta,\theta^\star)\leq\int_\mathbb{Z} \int_{\theta(z)\wedge \theta^\star(z)}^{\theta(z)\vee \theta^\star(z)} P(dx)P(dz)  = d(\theta,\theta^\star).\]
Since $|2\eta_z(x) - 1| \leq 1$ and, by Assumption~5(5),
\[ \int_{\theta(z)\wedge \theta^\star(z)}^{\theta(z)\vee \theta^\star(z)} P(dx) \lesssim  |\theta(z) - \theta^\star(z)|, \]
it follows that
\[\{m(\theta,\theta^\star)\vee v(\theta,\theta^\star)\} \lesssim \|\theta-\theta^\star\|_{L_1(P)}:=\int_\mathbb{Z} |\theta(z)-\theta^\star(z)| \, P(dz).\]
Let $J_n = n^{1/(1+\alpha)}$ and define a sup-norm ball around $\theta^\star$
\[B_n^\star  = \{(\beta,J): \beta \in \RR^J, J=J_n, \|\theta^\star - \hat\theta_{J, \beta}\|_\infty\leq CJ_n^{-\alpha}\}.\]
Then, by Assumption~\ref{assump:bdd_density}(3), and using the same argument as in the proof of Theorem~1 in \citet{shen.ghosal.2015} we have
\[\Pi(B_n^\star) \gtrsim  e^{-CJ_n\log n},\]
for some $C>0$.  Since $J_n\log n \lesssim n\omega\eps_n$ and $\theta\in B_n^\star$ implies $\{ m(\theta,\theta^\star)\vee  v(\theta, \theta^\star)\}\lesssim \eps_n$, it follows that (14) holds with $r=1$.

Next, we verify Condition~1.  By Assumption~\ref{assump:bdd_density}(4)
\begin{align*}
    m(\theta,\theta^\star) & \geq h \int_\mathbb{Z} \int_{\theta(z)\wedge \theta^\star(z)}^{\theta(z)\vee \theta^\star(z)} P(dx)P(dz) = h \, d(\theta,\theta^\star).
\end{align*}
Then, Lemma 7.26 in \citet{wasserman} implies
\begin{align*}
Pe^{-\omega(\ell_\theta - \ell_{\theta^\star})} & \leq \exp\{C\omega^2v(\theta,\theta^\star)-\omega m(\theta,\theta^\star)\}\\
& \leq \exp\{ -\omega (h - C\omega) \,  d(\theta,\theta^\star) \}, 
\end{align*}
where $C>0$ depends only on $\omega$.  For small $\omega$, $C = O(1+\omega)$, so if $\omega(1+\omega) \leq h$, then 
\[Pe^{-\omega(\ell_\theta - \ell_{\theta^\star})} \leq \exp\{-K\omega d(\theta,\theta^\star)\} \] 
for a constant $K$ depending on $h$, which verifies Condition~1 with $r=1$.

\end{document}